\documentclass[11pt,a4paper]{amsart}

 \usepackage{amsfonts,amsmath,amscd,amssymb,amsbsy,amsthm,amstext,amsopn}
 \usepackage{fullpage,mathrsfs,subfigure}
 \usepackage{pstricks,pst-node,pst-text,pst-tree}
 \usepackage{stmaryrd}  
 \usepackage{extarrows}

 \numberwithin{equation}{section}
\newtheorem{theorem}{Theorem}[section]
\newtheorem{proposition}[theorem]{Proposition}
\newtheorem{lemma}[theorem]{Lemma}

\newtheorem{corollary}[theorem]{Corollary}
\newtheorem{conjecture}[theorem]{Conjecture}

\theoremstyle{definition}
\newtheorem{definition}[theorem]{Definition}
\newtheorem{example}[theorem]{Example}

\theoremstyle{remark}
\newtheorem{remark}[theorem]{Remark}

\newcommand{\kk}{\ensuremath{\Bbbk}} 

\newcommand{\NN}{\ensuremath{\mathbb{N}}} 
 
\newcommand{\QQ}{\ensuremath{\mathbb{Q}}} 
\newcommand{\RR}{\ensuremath{\mathbb{R}}} 
\newcommand{\ZZ}{\ensuremath{\mathbb{Z}}} 
\newcommand{\one}{\ensuremath{(\mathrm{i})}}
\newcommand{\two}{\ensuremath{(\mathrm{ii})}}
\newcommand{\three}{\ensuremath{(\mathrm{iii})}}

\newcommand{\codim}{\operatorname{cod}}

\newcommand{\conv}{\operatorname{conv}}

\renewcommand{\div}{\operatorname{div}} 
\newcommand{\head}{\operatorname{\mathsf{h}}}
\newcommand{\ghilb}{\ensuremath{G}\operatorname{-Hilb}}
\newcommand{\git}{\ensuremath{/\!\!/\!}}

\newcommand{\hilbg}{\operatorname{Hilb}^{\ensuremath{G}}}

\newcommand{\inc}{\operatorname{inc}}

\newcommand{\mult}{\operatorname{mult}}

\newcommand{\relint}{\operatorname{relint}}

\newcommand{\supp}{\operatorname{supp}}
\newcommand{\tail}{\operatorname{\mathsf{t}}} 
\newcommand{\tile}{\operatorname{\mathcal{T}}}

\newcommand{\wt}{\operatorname{wt}} 

\newcommand{\ACcyc}{\ensuremath{\mathscr{C}_{\mathsf{ac}}(Q)}}

 \newcommand{\Cl}{\operatorname{Cl}}

\newcommand{\End}{\operatorname{End}}
 
\newcommand{\GL}{\operatorname{GL}}  
  
\newcommand{\Hom}{\operatorname{Hom}}

\newcommand{\Ker}{\operatorname{ker}}

\newcommand{\Proj}{\operatorname{Proj}}

\newcommand{\SL}{\operatorname{SL}} 
\newcommand{\Spec}{\operatorname{Spec}} 
 
\newcommand{\Tile}{\operatorname{\mathcal{P}}} 
\newcommand{\Wt}{\operatorname{Wt}}

\newcommand{\modAPhi}{\operatorname{mod}(\ensuremath{A_W})}

\def\lbdot{_{{^{_{_{_{_{_{\bullet}}}}}}}}}

\title{Cellular resolutions of noncommutative toric algebras from superpotentials}
\thanks{MSC 2010: Primary  14A22, 16S38;  Secondary 14M25, 16E05, 16G20,}

 \author{Alastair Craw and Alexander Quintero V\'{e}lez} 
 \address{Department of Mathematics\\ University of Glasgow\\ Glasgow\\ G12 8QW\\ United Kingdom}
 \email{Alastair.Craw@glasgow.ac.uk, Alexander.QuinteroVelez@glasgow.ac.uk}

\begin{document}
\bibliographystyle{plain}

 \begin{abstract}
  This paper constructs cellular resolutions for classes of noncommutative algebras,  analogous to those introduced by Bayer--Sturmfels~\cite{BayerSturmfels} in the  commutative case. To achieve this we generalise the dimer model construction of noncommutative crepant resolutions of three-dimensional toric algebras by associating a superpotential and a notion of consistency to toric algebras of arbitrary dimension. For abelian skew group algebras and algebraically consistent dimer model algebras, we introduce a cell complex $\Delta$ in a real torus whose cells describe uniformly all maps in the minimal projective bimodule resolution of $A$. We illustrate the general construction of $\Delta$ for an example in dimension four arising from a tilting bundle on a smooth toric Fano threefold to highlight the importance of the incidence function on $\Delta$.
  \end{abstract}

 \maketitle

 \section{Introduction}
 Cellular resolutions were introduced for classes of monomial modules by Bayer--Sturmfels~\cite{BayerSturmfels}, generalising the simplicial resolutions for monomial ideals by Bayer--Peeva--Sturmfels~\cite{BPS} and Peeva--Sturmfels~\cite{PeevaSturmfels}. In this paper we  develop a noncommutative analogue for certain classes of noncommutative algebra, including skew group algebras for finite abelian subgroups of $\SL(n,\kk)$ and superpotential algebras of global dimension three arising from algebraically consistent dimer models. In each case, the minimal bimodule resolution of the algebra is encoded by a collection of cells in a real torus that we call the \emph{toric cell complex}.

We first recall the construction of cellular resolutions of monomial modules over a polynomial ring from Bayer--Sturmfels~\cite{BayerSturmfels}. For a field $\kk$, set $S:=\kk[x_1,\dots,x_n]$ and consider a monomial $S$-module $M$ with generators $m_1,\dots, m_r \in S$.  Let $\Delta$ be a regular cell complex of dimension $n$ with vertex set $\Delta_0=\{1,\dots, r\}$. Label each face $\eta\in \Delta$ by the least common multiple $m_{\eta}$ of the monomials that label its vertices. For any choice of  incidence function $\varepsilon$ on $\Delta$, the labelled regular cell complex $\Delta$ defines a complex of free $\ZZ^n$-graded $S$-modules
\begin{align}
\label{eqn:BSresolution}
\begin{split}
0 \longrightarrow\bigoplus_{\eta \in \Delta_n}S(-m_{\eta}) & \xlongrightarrow{\partial_{n}}  \bigoplus_{\eta^\prime \in \Delta_{n-1}}S(-m_{\eta^\prime})  \xlongrightarrow{\partial_{n-1}} \cdots \\
\cdots & \xlongrightarrow{\partial_2}  \bigoplus_{e \in \Delta_1}S(-m_{e}) \xlongrightarrow{\partial_1}  \bigoplus_{j \in \Delta_0}S(-m_{j})  \xlongrightarrow{\partial_0} M \longrightarrow 0,
 \end{split}
 \end{align}
where $S(-m_{\eta})$ is the free $S$-module with generator $\eta$ in degree $\deg(m_{\eta})$. The maps satisfy 
$$
\partial_{k}(\eta)=\sum_{ \codim(\eta',\eta)=1}\varepsilon(\eta,\eta') \frac{m_{\eta}}{m_{\eta'}}\eta'
$$
for $\eta \in \Delta_k$, where the sum is taken over all codimension-one faces of $\eta$. Necessary and sufficient conditions for the complex \eqref{eqn:BSresolution} to be acyclic are given~\cite[Proposition~1.2]{BayerSturmfels}, and several classes of examples are presented that satisfy the conditions, in which case the complex \eqref{eqn:BSresolution}  is called  a \emph{cellular resolution} of $M$.

Before describing our main results we sketch the notion of consistency for toric algebras. Let $\mathscr{E}=(E_0,\dots,E_r)$ denote a collection  of reflexive sheaves of rank one on a Gorenstein affine toric variety $X$ of dimension $n$. Our main object of study is the \emph{toric algebra} $A:= \End(\bigoplus_{i=0}^r E_i)$ associated to $\mathscr{E}$.  Following Craw--Smith~\cite{CrawSmith}, we introduce the quiver of sections $Q$ of $\mathscr{E}$, and present $A$ as the quotient of the path algebra of $Q$ by an ideal of relations $J_{\mathscr{E}}$. We use the labelling of arrows in $Q$ to define the superpotential $W$ of $\mathscr{E}$ as a formal sum of cycles in the quiver and, on taking certain higher order derivatives of $W$, we obtain an auxilliary ideal of relations $J_W$ in the path algebra of $Q$. The toric algebra $A$ is \emph{consistent} if the ideals $J_{\mathscr{E}}$ and $J_W$ coincide.  Examples include skew group algebras $\kk[x_1,\dots, x_n]*G$ for finite abelian subgroups $G\subset \SL(n,\kk)$ and algebraically consistent dimer model algebras as defined by Broomhead~\cite{Broomhead} in his study of quivers and superpotentials in dimension three.

The notion of consistency is enough to provide a link between $A$ and the toric variety $X$. Indeed, let $\mathcal{M}_{\theta}$ denote the fine moduli space of $\theta$-stable $A$-modules of dimension vector $(1,\dots,1)$ for a generic weight $\theta$, and write $Y_\theta$ for the unique irreducible component of $\mathcal{M}_\theta$ that is birational to $X$.  We establish the following result in Theorem~\ref{thm:cohcomp}:

\begin{theorem}
\label{thm:1.1}
For consistent toric algebras $A$, we present an explicit GIT construction of $Y_\theta$ such that the projective birational morphism $Y_\theta\to X$ is obtained by variation of GIT quotient.
\end{theorem}

\noindent  Theorem~\ref{thm:1.1} unifies and extends results by Craw--Maclagan--Thomas~\cite{CMT1} on moduli of McKay quiver representations, and by Mozgovoy~\cite{Mozgovoy} on algebraically consistent dimer models. 

\medskip

We now describe our main result, namely, the construction of the minimal projective bimodule resolution for classes of consistent toric algebras of global dimension $n$ from the toric cell complex $\Delta$ in a real $n$-torus. The key lies in constructing $\Delta$. This is straightforward when $A$ is the skew group algebra for a finite abelian subgroup of $\SL(n,\kk)$, in which case $\Delta $ is a regular cell complex. It is considerably more difficult when $A$ is an algebraically consistent dimer model algebra, and in this case the resulting subdivision $\Delta$ of the torus is not even a CW-complex. Nevertheless,  $\Delta$ shares several key properties with regular cell complexes which explains our use of the `cellular' terminology (see Remark~\ref{rem:CWcomplex}). Notably, it admits an incidence function $\varepsilon\colon \Delta\times \Delta\to \{0,\pm 1\}$. In each class of examples as above and for any choice of incidence function $\varepsilon$ on $\Delta$, the toric cell complex $\Delta$ defines a complex of projective $(A,A)$-bimodules
\begin{align}\label{eqn:NCBSresolution}
\begin{split}
0 \longrightarrow \bigoplus_{\eta \in \Delta_n}A e_{\head(\eta)}\otimes [\eta] \otimes e_{\tail(\eta)}A & \xlongrightarrow{d_{n}} \bigoplus_{\eta^\prime \in \Delta_{n-1}}A e_{\head(\eta^\prime)}\otimes [\eta^\prime] \otimes e_{\tail(\eta^\prime)} \xlongrightarrow{d_{n-1}}\cdots  \\
 \cdots & \xlongrightarrow{d_2}  \bigoplus_{a \in \Delta_1}A e_{\head(a)}\otimes [a] \otimes e_{\tail(a)}A  \xlongrightarrow{d_1}  A \otimes A  \xlongrightarrow{\mu} A \longrightarrow 0,
 \end{split}
 \end{align} 
where each $e_i\in A$ is a primitive idempotent,  where $[\eta]$ are symbols indexed by cells that encode a semigroup grading, and where $\mu\colon  A\otimes A \rightarrow A$ is the multiplication map.  The maps satisfy
 $$
 d_{k}(1 \otimes [\eta] \otimes 1)=\sum_{\codim(\eta',\eta)=1}\varepsilon(\eta,\eta') \overleftarrow{\partial}_{\!\eta'}\eta\otimes[\eta'] \otimes \overrightarrow{\partial}_{\!\eta'}\eta.
 $$
Here, the expressions $\overleftarrow{\partial}_{\!\eta'}\eta$ and $\overrightarrow{\partial}_{\!\eta'}\eta$ are elements of $A$ obtained by right-\ and left-differentiation of cells (see, for example, Definitions~\ref{def:leftrightderivativesmckay} and \ref{def:leftrightderivativesdimer}). These elements measure the difference between $\eta$ and $\eta^\prime$, and provide the noncommutative analogue of the monomial $m_{\eta}/m_{\eta'}$ from \eqref{eqn:BSresolution}. The following result combines Theorems~\ref{thm:McKayresolution} and \ref{thm:dimerresolution}. 
 
\begin{theorem}
\label{thm:1.2}
Let $\Delta$ denote the toric cell complex of an abelian skew group algebra or an algebraically consistent dimer model algebra. 
Then the complex \eqref{eqn:NCBSresolution} is the minimal projective $(A,A)$-bimodule resolution of $A$.
\end{theorem}

In each case we refer to \eqref{eqn:NCBSresolution} as the \emph{cellular resolution} of $A$. For the skew group algebra, we recover the Koszul resolution of $A$ for a suitable choice of $\varepsilon$, and our presentation is reminiscent of that from Tate--Van den Bergh~\cite[\S3]{TateVandenbergh}. For an algebraically consistent dimer model algebra, we exhibit an incidence function $\varepsilon$ for which \eqref{eqn:NCBSresolution} is the standard resolution associated to a quiver with superpotential in dimension three studied by Ginzburg~\cite{Ginzburg}, Mozgovoy-Reineke~\cite{MozgovoyReineke}, Davison~\cite{Davison} and Broomhead~\cite{Broomhead}.

To conclude, we conjecture that the toric cell complex can be constructed for any consistent toric algebra $A$ whose global dimension $n$ is equal to the dimension of $X$ and, moreover, that the resulting complex \eqref{eqn:NCBSresolution} is an  $(A,A)$-bimodule resolution of $A$. We provide further evidence for this conjecture by examining a representative example arising from a tilting bundle on a smooth toric Fano threefold. More generally, we anticipate a link between the toric cell complex and the coamoeba from Futaki--Ueda \cite{FutakiUeda} that would describe concretely the mirror Landau-Ginzburg models for smooth toric Fano $n$-folds in the context of Homological Mirror Symmetry. 

A direct application of the construction presented here can be made in the study of quiver gauge theories with AdS/CFT gravity duals. As explained by Davey et.\ al.~\cite{DHMT1}, dimer models can be used to describe the gauge theories duals of a class of AdS/CFT backgrounds arising from M$2$-branes placed at a conical Calabi-Yau fourfold. However, the real meaning of dimers in this context is not yet clear. Developing the relationship between the quivers with superpotentials obtained from our construction in dimension four and those arising from the dimer model will hopefully lead to a deeper understanding of this problem.

\medskip

We now describe the structure of the paper. Section~\ref{sec:toricalgebras} defines toric algebras and investigates the geometry arising from labelled quivers of sections. The superpotential $W$ and the notion of consistency are presented in Section~\ref{sec:superpotential}, leading to a proof of Theorem~\ref{thm:1.1}. Section~\ref{sec:McKay} constructs the toric cell complex $\Delta$ and the resolution \eqref{eqn:NCBSresolution} in the motivating example of an abelian skew group algebra. We prove in Section~\ref{sec:dimers} that our superpotential coincides up to sign with the superpotential for an algebraically consistent dimer model algebra, and we use this result to construct $\Delta$ and the resolution \eqref{eqn:NCBSresolution} in this case. This completes the proof of Theorem~\ref{thm:1.2}. We present in Section~\ref{sec:conjecture} the fourfold example which explains why our superpotentials do not involve signs and we conclude with the statement of the main conjecture.

\medskip
\noindent\textbf{Conventions}
  Write $\kk$ for an algebraically closed field, $\kk^\times$ for the one-dimensional algebraic torus over
   $\kk$, and $\NN$ for the semigroup of nonnegative integers. We do not assume that toric
   varieties are normal. We often write $p^\pm$ as shorthand for `$p^+$ and $p^-$'.  Our pictures of cyclic quivers are drawn `unwrapped' to simplify the illustration, and we label each vertex to indicate those vertices that must be identified to reproduce the cyclic quiver from the picture.

\medskip
\noindent\textbf{Acknowledgements.}  The first author benefited greatly from many conversations with Greg Smith, particular during the MSRI programme in algebraic geometry in 2009.  Thanks also to Christian Haase, Akira Ishii,  Alastair King, Sergey Mozgovoy, Jan Stienstra,  Bal\'{a}zs Szendr\H{o}i and Michael Wemyss  for useful comments and questions.  In addition, we thank the anonymous referee for comments. Both authors are supported by EPSRC grant EP/G004048.

  \section{Toric algebras from a quiver of sections}
  \label{sec:toricalgebras}
  This section introduces the noncommutative toric algebra associated to any collection of rank one reflexive sheaves on a normal affine toric variety $X$.  The labelled quivers that encode these algebras also encode the action of an algebraic torus on an auxilliary toric variety, and variation of the resulting GIT quotient produces partial resolutions of $X$. Our toric algebras generalise slightly those from Broomhead~\cite{Broomhead} (compare also the notion of toric $R$-order from Bocklandt~\cite{Bocklandt}).

\subsection{Toric geometry} 
Let $X = \Spec R$ be a normal affine toric variety of dimension $n$ with a torus-invariant point. Let $M$ denote the character lattice of the dense torus $T_M:=\Spec \kk[M]$ in $X$, and write $N:=\Hom_\ZZ(M,\ZZ)$ for the dual lattice. There is a strongly convex rational polyhedral cone $\sigma\subset N\otimes_\ZZ \RR$ such that $R=\kk[\sigma^\vee\cap M]$.  Write $\sigma(1)$ for the set of one-dimensional faces of $\sigma$,  set $d:=\vert\sigma(1)\vert$, and let $v_\rho\in N$ denote the primitive lattice point on $\rho\in \sigma(1)$. Each $\rho\in \sigma(1)$ determines an irreducible $T_M$-invariant Weil divisor $D_\rho$ in $X$. These divisors generate the lattice $\ZZ^d$ of $T_M$-invariant Weil divisors and the semigroup $\NN^d$ of effective $T_M$-invariant Weil divisors. The map $\deg\colon \ZZ^d\to \Cl(X)$ sending  $D$ to the rank one reflexive sheaf $\mathcal{O}_X(D)$ fits in to the short exact sequence of abelian groups
 \begin{equation}
 \label{eqn:Coxsequence}
 \begin{CD}   
    0@>>> M @>>> \ZZ^d @>{\deg}>> \Cl(X)@>>> 0,
\end{CD}
\end{equation}
where the injective map sends $u$ to $\sum_{\rho\in \sigma(1)} \langle u,v_\rho\rangle D_\rho$. The Cox ring of $X$ is the polynomial ring $\kk[x_\rho :\rho\in \sigma(1)]$ obtained as the semigroup algebra of $\NN^d$, and we have $R\cong \kk[\NN^d\cap \Ker(\deg)]$. Since $X$ is normal, every rank one reflexive sheaf on $X$ is of the form $\mathcal{O}_X(D)$ for some Weil divisor class $D\in \Cl(X)$, and conversely. For an $R$-module $E$, write $E^\vee:=\Hom_{R}(E,R)$. 

The following result is trivial if the sheaves $E$ and $F$ are invertible. 

\begin{lemma}
\label{lem:reflexive}
Given rank one reflexive sheaves $E= \mathcal{O}_X(D)$ and $F=\mathcal{O}_X(D^\prime)$, we have that 
\begin{equation}
\label{eqn:reflexive}
\Hom_{\mathcal{O}_X}(E, F) \cong  H^0\big(\mathcal{O}_X(D^\prime-D)\big).
\end{equation}
\end{lemma}
\begin{proof}
For any $R$-module $E$ and for any reflexive $R$-module $F$, adjunction gives 
\[
\Hom_{R}(E,F)\cong \Hom_{R}\big(E,\Hom_{R}(F^\vee,R)\big) \cong 
\Hom_{R}(E\otimes F^\vee,R) = (E\otimes F^\vee)^\vee,
\]
hence $\Hom_{R}(E,F)^{\vee\vee}\cong (E\otimes F^\vee)^{\vee\vee\vee}$. Now, the global sections functor is an equivalence between the category of coherent sheaves on $X$ and the category of finitely generated $R$-modules, and the composition of $H^0(-)$ with the functor $\mathcal{H}om_{\mathscr{O}_X}(-,\mathscr{O}_X)$ is simply $\Hom_R(-,R)$. In particular, 
\[
\Hom_{\mathcal{O}_X}(E,F)^{\vee\vee} \cong  H^0\big(X,(E\otimes F^\vee)^{\vee\vee\vee}\big).
\]
We now assume that $E= \mathcal{O}_X(D)$ and $F=\mathcal{O}_X(D^\prime)$. Then $F^\vee \cong \mathcal{O}_X(-D^\prime)$ and $(E\otimes F^\vee)^{\vee\vee}\cong \mathcal{O}_X(D-D^\prime)$, see for example Cox--Little--Schenck~\cite[Proposition 8.0.6]{CLS}. Substitute this into the above and apply $\mathcal{O}_X(D-D^\prime)^\vee = \mathcal{O}_X(D^\prime-D)$ to obtain
\[
\Hom_{\mathcal{O}_X}(E, F)^{\vee\vee} \cong  H^0\big(\mathcal{O}_X(D^\prime-D)\big).
\]
The left hand side is reflexive by Benson~\cite[Lemma~3.4.1(iv)]{Benson}. This completes the proof.
\end{proof}

 \subsection{Quivers of sections}
 Let $Q$ be a finite connected quiver with vertex set $Q_0$, arrow set $Q_1$, and maps $\head, \tail \colon Q_1 \to Q_0$ indicating the vertices at the head and tail of each arrow.  A nontrivial path in $Q$ is a sequence of arrows $p = a_k \dotsb a_1$ with $\head(a_{j}) = \tail(a_{j+1})$ for $1 \leq j < k$.  We set $\tail(p) = \tail(a_{1}), \head(p) = \head(a_k)$ and $\supp(p)=\{a_1,\dots, a_k\}$.  A cycle is a path $p$ with $\tail(p) = \head(p)$. Each $i \in Q_0$ gives a trivial path $e_i$ where $\tail(e_i) = \head(e_i) = i$.  The path algebra $\kk Q$ is the $\kk$-algebra whose underlying $\kk$-vector space has a basis of paths in $Q$, where the product of basis elements is the basis element defined by concatenation of the paths if possible, or zero otherwise.  Let $[\kk Q,\kk Q]$ denote the $\kk$-vector space spanned by all commutators in $\kk Q$, so $\kk Q_{\text{cyc}}:=\kk Q/[\kk Q,\kk Q]$ has a basis of elements corresponding to cyclic paths in the quiver.

 For $r\geq 0$, consider a collection $\mathscr{E}:=(E_0,E_1,\dots,E_r)$ of distinct rank one reflexive sheaves on the affine toric variety $X$.  Since $X$ is normal, every such sheaf is of the form $E_i=\mathcal{O}_X(D_i^\prime)$ for some $D_i^\prime\in \Cl(X)$. For $0\leq i, j\leq r$, a $T_M$-invariant section $s \in \Hom_{\mathcal{O}_X}(E_i,E_j)$ is said to be \emph{irreducible} if it does not factor through some $E_k$ with $k\neq i,j$, that is, the section does not lie in the image of the multiplication map 
\[
H^0\big(X,\mathcal{O}_X(D_j^\prime-D_k^\prime)\big)\otimes_\kk H^0\big(X,\mathcal{O}_X(D_k^\prime-D_i^\prime)\big)\longrightarrow H^0\big(X,\mathcal{O}_X(D_j^\prime-D_i^\prime)\big)
\]
for any $k\neq i,j$, where we use the isomorphism \eqref{eqn:reflexive} from Lemma~\ref{lem:reflexive}.    
    
  \begin{definition}
  The \emph{quiver of sections} of $\mathscr{E}$ is the finite quiver $Q$ in which the vertex set $Q_0 = \{ 0, \dotsc,r \}$ corresponds to the sheaves in $\mathscr{E}$, and where the arrows from $i$ to $j$ correspond to the irreducible sections in $\Hom_{\mathcal{O}_X}(E_i,E_j)$.
\end{definition}

\begin{remark}
\begin{enumerate}
\item Lemma~\ref{lem:reflexive} writes $\Hom_{\mathcal{O}_X}(E_i,E_j)$ in terms of Weil divisors. To construct $Q$ in practice,  write each $E_i\in \mathscr{E}$ as $E_i=\mathcal{O}_X(D_i^\prime)$ for some $D_i^\prime\in \Cl(X)$, and compute for every $i,j\in Q_0$ the vertices of the polyhedron $\conv(\NN^d\cap \deg^{-1}(D^\prime_i-D_j^\prime))$ to obtain the $T_M$-invariant $R$-module generators of $\Hom_{\mathcal{O}_X}(E_i,E_j)$. The arrows of $Q$ correspond to the generators of irreducible maps.
\item The quiver $Q$ depends only on differences of effective Weil divisors on $X$. We normalise by choosing $E_0:= \mathcal{O}_X$.
\end{enumerate}
\end{remark}

For $a \in Q_1$, write $\div(a) := \div(s) \in \NN^{d}$ for the divisor of zeroes of the defining section $s \in \Hom_{\mathcal{O}_X}(E_i,E_j)$ and, more generally, for any path $p$ in $Q$ we call $\div(p) := \sum_{a\in \supp(p)} \div(a)$ the \emph{label} of $p$. The labelling monomial is $x^{\div(p)}:= \prod_{a\in \supp(p)} x^{\div(a)}\in \kk[x_\rho : \rho\in \sigma(1)]$.
     
\begin{definition}
Consider the two-sided ideal 
  \[
  J_{\mathscr{E}}:= \big( p^+-p^- \in \kk Q \mid \head(p^+)=\head(p^-),
 \tail(p^+)=\tail(p^-), \div(p^+) = \div(p^-)\big)
 \]
 in the path algebra $\kk Q$. The quotient $A_{\mathscr{E}}:= \kk Q/J_{\mathscr{E}}$ is the \emph{toric algebra} of the collection $\mathscr{E}$, and the pair $(Q,J_{\mathscr{E}})$ is the \emph{bound quiver of sections} of the collection $\mathscr{E}$. The phrase `bound quiver' is a synonym for `quiver with relations'.
 \end{definition}

 \begin{lemma}
 \label{lem:algebra}
 For $r\geq 0$ and for $\mathscr{E}=(E_0,E_1,\dots,E_r)$, the quiver of sections $Q$ of $\mathscr{E}$ is strongly connected and $A_{\mathscr{E}}\cong \End_R\bigl( \bigoplus_{i\in Q_0} E_i \bigr)$. In particular, the centre $Z(A_{\mathscr{E}})$ is isomorphic to $R$.
 \end{lemma}
 \begin{proof}
 The top-dimensional cone $\sigma^\vee\subset M\otimes_\ZZ \RR$ is obtained by slicing the cone $\RR^d_{\geq 0}$ by $\Ker(\deg)$, so there exists $u\in \sigma^\vee\cap M$ such that the lattice point $\sum_\rho \langle u,v_\rho\rangle D_\rho$ lies in the interior of $\RR^d_{\geq 0}$. For each $\rho\in \sigma(1)$, set $\mu_\rho:= \langle u,v_\rho\rangle >0$. For $i\in Q_0$, write $E_i=\mathcal{O}_X(D)$ where $D=\sum_{\rho\in \sigma(1)} \lambda_\rho D_\rho$ and choose $k, l\in \ZZ$ satisfying $k\mu_\rho\leq \lambda_\rho\leq l\mu_\rho$ for all $\rho\in \sigma(1)$. Then $\sum_{\rho\in \sigma(1)} (l\mu_\rho-\lambda_\rho)D_\rho$ and $\sum_\rho (\lambda_\rho-k\mu_\rho)D_\rho$ are effective divisors, so both $\Hom_R(E_0,E_i)$ and $\Hom_R(E_i,E_0)$ are nonempty. It follows that $Q$ is strongly connected. The stated isomorphism of $\kk$-algebras follows as in the proof of \cite[Proposition~3.3]{CrawSmith}. To compute the centre, consider the $\kk$-linear map $\kk Q_{\mathsf{cyc}}\to R$ determined by sending a cycle $p$ to the section $x^{\div(p)}$. This map is surjective by construction of $Q$.  Since the centre of $A_{\mathscr{E}}$ is generated by $J_{\mathscr{E}}$-equivalence classes of cycles in $Q$,  the isomorphism $Z(A_{\mathscr{E}})\to R$ follows after taking equivalence classes modulo $J_{\mathscr{E}}$.
  \end{proof}

\begin{example}
For any $X=\Spec(R)$, the quiver of sections $Q$ of the trivial collection $\mathscr{E} = (\mathcal{O}_X)$ has one vertex. If $X=\Spec(\kk)$ then $Q_1=\emptyset$ and $A\cong \kk$. Otherwise, $Q$ has $m$ loops where the labelling divisors $\div(a_1),\dots,\div(a_m)$ are the elements in the Hilbert basis of the semigroup $\sigma^\vee\cap M$. The $\kk$-algebra epimorphism $\kk Q\to R$ sending $a_i\mapsto x^{\div(a_i)}$ for $1\leq i\leq m$ has kernel $J_{\mathscr{E}}$, so the toric algebra $A_{\mathscr{E}}$ is isomorphic to the coordinate ring $R$. In particular, coordinate rings of normal affine toric varieties are toric algebras. 
\end{example}

\begin{example}
\label{exa:F1tilting}
Let $\sigma$ be the cone in $\RR^3$ generated by $v_1=(1,0,1)$, $v_2= (0,1,1)$, $v_3= (-1,1,1)$, $v_4= (0,-1,1)$, so $\sigma$ is the cone over the lattice polygon shown in Figure~\ref{fig:tiltingF1}(a). 
   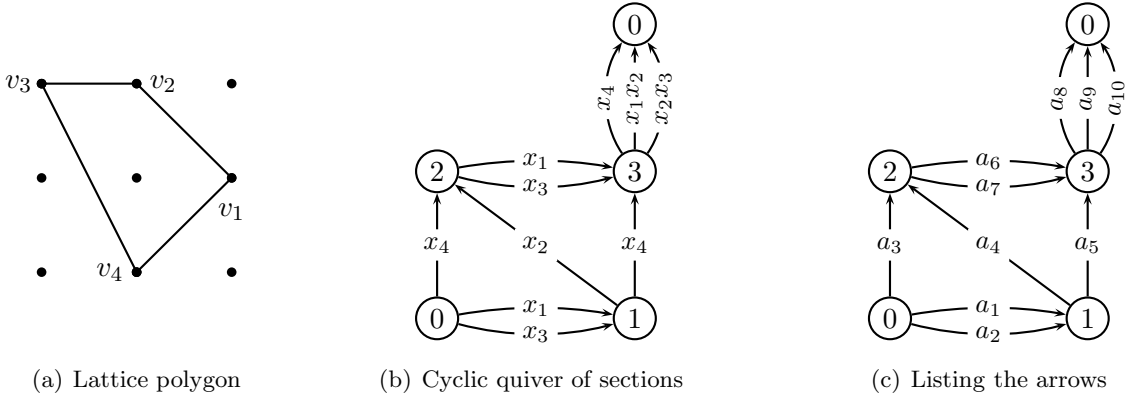
\begin{figure}[!ht]
    \centering
      \subfigure[Lattice polygon]{
       \psset{unit=1cm}
     \begin{pspicture}(0,-1)(2.5,2.6)
       \psline{*-*}(0,2.5)(1.25,0)
       \psline{*-}(1.25,0)(2.5,1.25)
       \psline{*-}(2.5,1.25,0)(1.25,2.5)
       \psline{*-*}(1.25,2.5)(0,2.5)
       \psdot(0,0)
       \psdot(1.25,0)
       \psdot(2.5,0)
       \psdot(0,1.25)
       \psdot(1.25,1.25)
       \psdot(2.5,1.25)
       \psdot(0,2.5)
       \psdot(1.25,2.5)
       \psdot(2.5,2.5)
       \rput(-0.3,2.5){$v_3$}       
        \rput(0.9,0){$v_4$}
       \rput(1.6,2.5){$v_2$}
       \rput(2.5,0.8){$v_1$}
       \end{pspicture}
       }
      \qquad  \qquad
      \subfigure[Cyclic quiver of sections]{
        \psset{unit=1.3cm}
        \begin{pspicture}(-0.6,-0.3)(2.6,3.1)
        \cnodeput(0,0){A}{0}
        \cnodeput(2,0){B}{1} 
        \cnodeput(0,1.5){C}{2}
        \cnodeput(2,1.5){D}{3}
        \cnodeput(2,3){E}{0} 
        \psset{nodesep=0pt}
        \nccurve[angleA=10,angleB=170]{->}{A}{B}\lput*{:U}{\small{$x_1$}}
        \nccurve[angleA=345,angleB=195]{->}{A}{B}\lput*{:U}{\small{$x_3$}}
        \nccurve[angleA=90,angleB=270]{->}{A}{C}\lput*{:270}{\small{$x_4$}}
        \ncline{->}{B}{C}\lput*{:215}{\small{$x_2$}}
        \nccurve[angleA=90,angleB=270]{->}{B}{D}\lput*{:270}{\small{$x_4$}}
        \nccurve[angleA=10,angleB=170]{->}{C}{D}\lput*{:U}{\small{$x_1$}}
        \nccurve[angleA=345,angleB=195]{->}{C}{D}\lput*{:U}{\small{$x_3$}}
        \nccurve[angleA=125,angleB=235]{->}{D}{E}\lput*{:U}{\small{$x_4$}}
        \nccurve[angleA=90,angleB=270]{->}{D}{E}\lput*{:U}{\small{$x_1x_2$}}
        \nccurve[angleA=55,angleB=305]{->}{D}{E}\lput*{:U}{\small{$x_2x_3$}}
        \end{pspicture}}
      \qquad \qquad  
      \subfigure[Listing the arrows]{
        \psset{unit=1.3cm}
 \begin{pspicture}(-0.6,-0.3)(2.6,3)
        \cnodeput(0,0){A}{$0$}
        \cnodeput(2,0){B}{1} 
        \cnodeput(0,1.5){C}{2}
        \cnodeput(2,1.5){D}{3}
        \cnodeput(2,3){E}{0} 
        \psset{nodesep=0pt}
        \nccurve[angleA=10,angleB=170]{->}{A}{B}\lput*{:U}{\small{$a_1$}}
        \nccurve[angleA=345,angleB=195]{->}{A}{B}\lput*{:U}{\small{$a_2$}}
        \nccurve[angleA=90,angleB=270]{->}{A}{C}\lput*{:270}{\small{$a_3$}}
        \ncline{->}{B}{C}\lput*{:215}{\small{$a_4$}}
        \nccurve[angleA=90,angleB=270]{->}{B}{D}\lput*{:270}{\small{$a_5$}}
        \nccurve[angleA=10,angleB=170]{->}{C}{D}\lput*{:U}{\small{$a_6$}}
        \nccurve[angleA=345,angleB=195]{->}{C}{D}\lput*{:U}{\small{$a_7$}}
        \nccurve[angleA=125,angleB=235]{->}{D}{E}\lput*{:U}{\small{$a_8$}}
        \nccurve[angleA=90,angleB=270]{->}{D}{E}\lput*{:U}{\small{$a_9$}}
        \nccurve[angleA=55,angleB=305]{->}{D}{E}\lput*{:U}{\small{$a_{10}$}}
        \end{pspicture}
        }
    \caption{A quiver of sections for a collection}
  \label{fig:tiltingF1}
  \end{figure}
  For $1\leq \rho\leq 4$, write $D_\rho$ for the Weil divisor in $X=\Spec \kk[\sigma^\vee\cap \ZZ^3]$ corresponding to the ray of $\sigma$ generated by $v_\rho$. The group $\Cl(X)$ is the quotient of the free abelian group generated by $\mathcal{O}_X(D_1)$ and $\mathcal{O}_X(D_4)$, by the subgroup generated by $\mathcal{O}_X(D_1+2D_4)$.  The quiver of sections $Q$ of $\mathscr{E} = \big(\mathcal{O}_X, \mathcal{O}_X(D_1),\mathcal{O}_X(D_4), \mathcal{O}_X(D_1+D_4)\big)$ is the cyclic quiver from Figure~\ref{fig:tiltingF1}(b); the quiver is shown in $\ZZ^2$, but $\mathcal{O}_X\sim\mathcal{O}_X(D_1+2D_4)$. For $a\in Q_1$ we have $x^{\div(a)}\in \kk[x_1,x_2,x_3,x_4]$, and
  \[
J_{\mathscr{E}}= \left(\begin{array}{c} \! a_6a_3-a_5a_1, \; a_7a_3-a_5a_2, \; a_7a_4a_1-a_6a_4a_2,\; a_3a_9-a_4a_1a_8,\; a_{3}a_{10}-a_4a_2a_8 \! \\
 \! a_2a_9-a_{1}a_{10}, \; a_1a_8a_7-a_2a_8a_6, \; a_9a_7-a_{10}a_{6},\; a_8a_6a_4-a_9a_5,\; a_{10}a_{5}-a_8a_7a_4 \! 
 \end{array}\right).
\]
defines the noncommutative toric algebra $A_{\mathscr{E}}=\kk Q/J_{\mathscr{E}}$.
\end{example}

\subsection{Polyhedral geometry}
The characteristic functions $\chi_{i} \colon Q_0 \to \ZZ$ and $\chi_{a} \colon Q_1 \to \ZZ$ for $i \in Q_0$ and $a \in Q_1$ form the
 standard integral bases of the vertex space $\ZZ^{Q_0}$ and the arrow space $\ZZ^{Q_1}$ respectively.  The incidence map 
 $\inc
 \colon \ZZ^{Q_1} \to \ZZ^{Q_0}$ defined by setting 
 $\inc(\chi_{a})=\chi_{\head(a)} - \chi_{\tail(a)}$ has image equal to
 the sublattice $\Wt(Q) \subset \ZZ^{Q_0}$ of functions $\theta \colon
 Q_0 \to \ZZ$ satisfying $\sum_{i \in Q_0} \theta_i = 0$. Generalising \cite[Definition~3.2]{CMT1} (compare also \cite{CrawSmith}), we define
 \[
 \pi := (\inc,\div)\colon \ZZ^{Q_{1}} \to \Wt(Q) \oplus \ZZ^{d}
 \] 
to be the $\ZZ$-linear map sending $\chi_{a}$ to $\bigl( \chi_{\head(a)} - \chi_{\tail(a)}, \div(a) \bigr)$ for $a\in Q_1$. Let $\ZZ(Q)$ and $\NN(Q)$ denote the image under $\pi$ of the lattice $\ZZ^{Q_1}$ and the subsemigroup $\NN^{Q_1}$ respectively, and write $\kk[\ZZ(Q)]$ and $\kk[\NN(Q)]$ for the semigroup algebras. Let $\pi_1\colon \ZZ(Q)\to \Wt(Q)$ and $\pi_2\colon \ZZ(Q)\to \ZZ^{d}$ denote the first and second projections respectively, and define a group homomorphism $\nu\colon\Wt(Q)\to \Cl(X)$ by setting $\nu(\chi_i)=E_i$ for all $i\in Q_0$. 
  
 \begin{lemma}
 \label{lem:diagram}
 There is a commutative diagram of abelian groups
 \begin{equation}
 \label{eqn:diagram}
  \begin{CD}   
    0@>>> M  @>>> \ZZ(Q)    @>{\pi_1}>> \Wt(Q)@>>> 0\\
     @.   @|            @V{\pi_2}VV   @V{\nu}VV      @.          \\
0 @>>> M @>>> \ZZ^{d}  @>{\deg}>> \Cl(X)   @>>> 0 \\
 \end{CD}
 \end{equation}
 where $\pi_2$ identifies the subsemigroup $\NN(Q)\cap \Ker(\pi_1)$ with $\sigma^\vee\cap M=\NN^d\cap \Ker(\deg)$. In particular, the rank of the lattice $\ZZ(Q)$ is $n+r$.
 \end{lemma}
\begin{proof}
The right-hand square commutes and the bottom row is exact, so it enough to prove that $\pi_2$ yields a $\ZZ$-linear isomorphism $\Ker(\pi_1)\cong M$ which restricts to an isomorphism of semigroups $\NN(Q)\cap \Ker(\pi_1)\cong\NN^d\cap \Ker(\deg)$. The proof of \cite[Proposition 4.1]{CMT1} generalises to our setting.
\end{proof}

 \begin{remark}
 The semigroup $\NN(Q)$ need not be saturated, see Remark~\ref{rem:notsaturated}.
 \end{remark}

Consider the commutative diagram 
 \begin{equation}
 \label{eqn:dualdiagram}
  \begin{CD}   
    0@<<< N  @<{\psi^*}<<  \ZZ(Q)^\vee   @<<< \Wt(Q)^\vee@<<< 0\\
     @.   @|            @A{\pi_2^*}AA   @AAA      @.          \\
0 @<<< N @<{\iota^*}<< \ZZ^{d}  @<<< \Cl(X)^\vee   @<<< 0 \\
 \end{CD}
 \end{equation}
dual to \eqref{eqn:diagram}.   Let $\{\chi_\rho \mid \rho\in \sigma(1)\}$ the standard basis of $\ZZ^d$. For each $\rho\in \sigma(1)$, the image of $\chi_\rho$ under the map $\iota^*\colon \ZZ^d\to N$ is the primitive generator $v_\rho\in \rho$, so the image of the positive orthant $\{\sum_\rho c_\rho \chi_\rho \mid c_\rho\geq 0\}$ under the linear map $\iota^*\otimes_\ZZ \RR\colon \RR^d \to N\otimes_\ZZ \RR$ is the cone $\sigma$. To establish a similar statement for the top row of \eqref{eqn:dualdiagram}, consider the convex polyhedral cone
 \[
C:= \big\{v\in \ZZ(Q)^\vee\otimes_\ZZ \RR \mid \langle u,v\rangle \geq 0 \text{ for all }u\in \NN(Q)\big\}.
 \]
 
\begin{lemma}
The image of $C$ under $\psi^*\colon \ZZ(Q)^\vee\otimes_\ZZ \RR \to N\otimes_\ZZ \RR$ is the cone $\sigma$.
\end{lemma}
\begin{proof}
Lemma~\ref{lem:diagram} shows that $\pi_2$ identifies the semigroup $\NN(Q)\cap \Ker(\pi_1)$ with $\sigma^\vee\cap M$, so the $\RR$-linear extension of $\pi_2$ identifies the slice $C\cap \Ker(\pi_1)$ with the cone $\sigma^\vee$. The result is now immediate from Craw--Maclagan~\cite[Corollary~2.10]{CrawMaclagan}.
\end{proof}

The semigroup $\NN(Q)$ is generated by the vectors $\pi(\chi_a)\in \ZZ(Q)$ arising from arrows $a\in Q_1$, so $C=\{v\in \ZZ(Q)^\vee\otimes_\ZZ \RR \mid \langle v,\pi(\chi_a)\rangle \geq 0 \text{ for all }a\in Q_1\}$. For any face $\Pi$ of $C$, let $\relint(\Pi)$ denote that relative interior of $\Pi$ and define the \emph{support} of $\Pi$ to be 
\[
\supp(\Pi):= \big\{a\in Q_1 \mid  \langle v,\pi(\chi_a)\rangle >0\text{ for all }v\in \relint(\Pi)\big\}.
\]
To explain the geometric significance of the support, note that the toric variety $\Spec \kk[C^\vee\cap \ZZ(Q)]$ is the normalisation of $\Spec\kk[\NN(Q)]$ because $C^\vee\cap \ZZ(Q)$ is the saturation of $\NN(Q)$. As is standard in toric geometry, a face $\Pi$ of $C$ defines the torus-orbit closure in $\Spec\kk[\NN(Q)]$ parametrising points $(w_a) \in\Spec\kk[\NN(Q)]\subseteq \mathbb{A}^{Q_1}_\kk$ whose coordinates satisfy $w_a= 0$ if and only if $a\in \supp(\Pi)$.

\begin{definition}
\label{def:perfectmatching}
A \emph{perfect matching} $\Pi$ of $Q$ is the primitive lattice point on a one-dimensional face of the cone $C$. We also refer to the face itself, or even to the set of arrows $\supp(\Pi)$, as the perfect matching. A perfect matching $\Pi$ is \emph{extremal} if $\psi^*(\Pi) = v_\rho$ for some $\rho\in \sigma(1)$.
\end{definition}

The terminology `perfect matching' is taken from the special case where the algebra $A$ arises from a dimer model as described in Section~\ref{sec:dimers}.

 \begin{proposition}
 \label{prop:perfmatchlabels}
 For $\rho\in \sigma(1)$,  the vector $\Pi_\rho:= \pi_2^*(\chi_\rho)$ is an extremal perfect matching of $Q$. In addition, for every $a\in Q_1$ we have
 \begin{equation}
 \label{eqn:supportlabels}
 a\in \supp(\Pi_\rho)\iff x_\rho \text{ divides }x^{\div(a)}.
 \end{equation}
  \end{proposition}
  \begin{proof}
  We begin by proving the second statement.  An arrow $a$ in $Q$ lies in $\supp(\Pi_\rho)$ if and only if $\langle \Pi_\rho,\pi(\chi_a)\rangle >0$. Since $\Pi_\rho:= \pi_2^*(\chi_\rho)$, we have
   \begin{equation}
  \label{eqn:adjoint}
  \big\langle \Pi_\rho,\pi(\chi_a)\big\rangle = \big\langle \chi_\rho, \pi_2(\pi(\chi_a))\big \rangle = \big\langle \chi_\rho, \div(a)\big\rangle,
 \end{equation}
 and this is positive if and only if $x_\rho$ divides $x^{\div(a)}$. 
 
 For the first statement, note that $\div(a)\in \NN^d$ for all $a\in Q_1$ and hence $  \big\langle \Pi_\rho,\pi(\chi_a)\big\rangle\geq 0$ by \eqref{eqn:adjoint}. It follows that $\Pi_\rho\in C$. Commutativity of diagram \eqref{eqn:dualdiagram} shows that $\psi^*(\Pi_\rho)$ is equal to the primitive lattice point $v_\rho$ in $\rho$, so $\Pi_\rho$ is a primitive lattice point in some face of $C$ that we also denote $\Pi_\rho$. To deduce that $\Pi_\rho$ is an extremal perfect matching it remains to show that the face $\Pi_\rho$ has dimension one or, equivalently, that the dual face $F$ in $C^\vee$ has dimension $n+r-1$. The identification of $\NN(Q)\cap \Ker(\pi_1)$ with $\sigma^\vee\cap M$ from Lemma~\ref{lem:diagram} and saturatedness of $\sigma^\vee\cap M$ enables us to identify $C^\vee\cap \Ker(\pi_1)$ with $\sigma^\vee\cap M$. The face $F_\rho$ of $\sigma^\vee$ dual to the cone $\rho$ has dimension $n-1$, and hence \cite[Lemma~2.5]{CrawMaclagan} gives $F_\rho=F\cap \Ker(\pi_1)$. We claim that $F$ intersects $\Ker(\pi_1)$ transversely, so $F$ has dimension $n-1+r$ as required. To prove the claim, it suffices by Thaddeus~\cite[Lemma~3.3]{Thaddeus} to show that $F$ is $0$-stable or, equivalently, that the quiver $Q^\prime$ with vertex set $Q_0$ and arrow set $Q_1\setminus\supp(\Pi_\rho)$ is strongly connected. In light of \eqref{eqn:supportlabels}, this quiver is obtained from $Q$ by deleting each $a\in Q_1$ for which $x_\rho$ divides $x^{\div(a)}$. It follows that $Q^\prime$ is the quiver of sections on the affine toric variety $D_\rho$ defined by the collection $\mathscr{E}^\prime = (E_i\vert_{D_\rho} : i\in Q_0)$. Lemma~\ref{lem:algebra} implies that $Q^\prime$ is strongly connected.
  \end{proof}

\begin{remark}
Together with the multiplicities from \eqref{eqn:adjoint}, Proposition~\ref{prop:perfmatchlabels} records the fact that extremal perfect matchings encode the labels in a quiver of sections.
\end{remark}

\subsection{Variation of GIT quotient}
The incidence map of $Q$ determines a $\Wt(Q)$-grading of the polynomial ring $\kk[y_a : a\in Q_1]$ obtained as the semigroup algebra of $\NN^{Q_1}$. The algebraic torus $T:=\Hom(\Wt(Q),\kk^\times)$ of rank $r$ then acts on the affine space $\mathbb{A}^{Q_1}_\kk:=\Spec \kk[\NN^{Q_1}]$, where for $(t_i)\in T$ and $(w_a)\in \mathbb{A}^{Q_1}_\kk$ we have
\begin{equation}\label{eqn:Taction}(t\cdot w)_{a} = t_{\head(a)}^{\,} w_{a} t_{\tail(a)}^{-1}.\end{equation} 
 For any weight $\theta \in \Wt(Q)$, let $\kk[\NN^{Q_1}]_{\theta}$ denote the $\theta$-graded piece of the coordinate ring of $\mathbb{A}^{Q_1}_\kk$. The GIT quotient $\mathbb{A}^{Q_1}_\kk\git_\theta T= \Proj(\bigoplus_{j\geq 0} \kk[\NN^{Q_1}]_{j\theta})$ is the categorical quotient of the open subscheme of  $\theta$-semistable points in $\mathbb{A}^{Q_1}_\kk$ by the action of $T$. We say that a weight $\theta\in \Wt(Q)$ is \emph{generic} if every $\theta$-semistable point of $\mathbb{A}^{Q_1}_\kk$ is $\theta$-stable, in which case, $\mathbb{A}^{Q_1}_\kk\git_\theta T$ is the geometric quotient of the open subscheme of  $\theta$-stable points in $\mathbb{A}^{Q_1}_\kk$ by $T$. 

The map $\pi$ induces a surjective map of semigroup algebras $\pi_*\colon \kk[\NN^{Q_1}]\to \kk[\NN(Q)]$ with kernel 
 \begin{equation} 
 \label{eqn:IQ} 
 I_{\mathscr{E}} := \big(y^u-y^v \in \kk[\NN^{Q_1}] \mid u-v\in \Ker(\pi)\big)
 \end{equation}
that cuts out the affine toric subvariety $\mathbb{V}(I_{\mathscr{E}})$ of $\mathbb{A}_\kk^{Q_1}$. The incidence map factors through $\NN(Q)$ to define a $\Wt(Q)$-grading on $\kk[\NN(Q)]$, so the $T$-action on $\mathbb{A}_\kk^{Q_1}$ restricts to an action on $\mathbb{V}(I_{\mathscr{E}})$. For $\theta \in \Wt(Q)$, let $\kk[\NN(Q)]_{\theta}$ denote the $\theta$-graded piece and write 
\[Y_\theta:= \mathbb{V}(I_{\mathscr{E}})\git_\theta T= \Proj\Big(\bigoplus_{j\geq 0} \kk[\NN(Q)]_{j\theta}\Big)\]
for the categorical quotient of the open subset of $\theta$-semistable points in $\mathbb{V}(I_{\mathscr{E}})$.

\begin{proposition}\label{prop:Ytheta}
For any $\theta\in \Wt(Q)$, the toric variety $Y_\theta = \mathbb{V}(I_{\mathscr{E}})\git_\theta T$ admits a projective birational morphism $\tau_\theta\colon Y_\theta \longrightarrow X=\Spec R$ obtained by variation of GIT quotient.\end{proposition}

\begin{proof}Lemma~\ref{lem:diagram} implies that $R=\kk[\NN^d\cap \Ker(\deg)]\cong \kk[\NN(Q)\cap \Ker(\pi_1)]=\kk[\NN(Q)]^T$, so the variety $X$ is isomorphic to $Y_0=\Spec \kk[\NN(Q)]^T$. Variation of GIT quotient gives the projective morphism $\tau_\theta\colon Y_\theta \to Y_0$, and it remains to show that $\tau_\theta$ is birational. Each $\theta$-semiinvariant monomial in $\kk[\NN^{Q_1}]$ is nowhere zero on the dense torus $\Spec\kk[\ZZ(Q)]$ of $\mathbb{V}(I_\mathscr{E})$ because the coordinate entries of every such point are all nonzero under the embedding of $\Spec(\kk[\ZZ(Q)])$ in the dense torus of $\mathbb{A}^{Q_1}_\kk$. It follows that every point of $\Spec \kk[\ZZ(Q)]$ is $\theta$-semistable. Since every point of $\Spec\kk[\ZZ(Q)]$ is also $0$-semistable, we deduce that the dense torus $\Spec\kk[\ZZ(Q)]\git_\theta T$ of $Y_\theta$ is isomorphic to the dense torus $\Spec\kk[\ZZ(Q)]^T$ of $Y_0$.\end{proof}

The morphism $\tau_\theta\colon Y_\theta\to X$ provides a resolution of singularities precisely when $Y_\theta$ is smooth. Note however that $Y_\theta$ need not even be normal, see Remark~\ref{rem:notsaturated}.

  \section{Consistency for superpotential algebras}
  \label{sec:superpotential}
  This section introduces the superpotential $W$ and the superpotential algebra $A_W$ of a quiver of sections $Q$ on $X$. This algebra need not be isomorphic to the toric algebra, but when it is we say that the toric algebra is consistent. This implies in particular that the toric variety $Y_\theta$ for generic $\theta$ is the coherent component of a fine moduli space $\mathcal{M}_\theta$ of $\theta$-stable $A_W$-modules.
   
  \subsection{Superpotential from anticanonical cycles}
Assume from now on that $X$ is Gorenstein, so $(1,\dots,1)\in \ZZ^d$ lies in the sublattice $M$ and hence $\sum_{\rho\in \sigma(1)} D_\rho$ is linearly equivalent to zero, giving $\omega_X\cong \mathcal{O}_X$. The primitive lattice point $v_\rho\in N$ on each ray in $\sigma$ lies in an affine hyperplane in $N\otimes_\ZZ \RR$, and $\sigma$ is the cone over the convex polytope $P=\conv(v_\rho \mid \rho \in \sigma(1))$. For $r\geq 0$, let $\mathscr{E}:=(\mathcal{O}_X,E_1,\dots,E_r)$ be a collection of distinct rank one reflexive sheaves on $X$ with quiver of sections $Q$. The anticanonical divisor $\sum_{\rho\in \sigma(1)} D_\rho$, or equivalently the monomial $\prod_{\rho\in \sigma(1)} x_\rho$ in the Cox ring of $X$,  singles out a preferred set of cycles in $Q$ as follows.

\begin{definition}
\label{defn:superpotentialnosigns}
A cycle $p$ in $Q$ is an \emph{anticanonical cycle} if $x^{\div(p)}=\prod_{\rho\in \sigma(1)} x_\rho$. Let $\ACcyc$ denote the set of anticanonical cycles. The \emph{superpotential} of the collection $\mathscr{E}$ is the formal sum of cycles $W := \sum_{p\in \ACcyc} p\in \kk Q_{\mathrm{cyc}}$. \end{definition}

Given paths $p,q$ in $Q$, the partial (left) derivative of $q$ with respect to $p$ is
\[
\partial_qp:=\left\{\begin{array}{cl} r & \text{if }p=rq; \\ 0 & \text{otherwise.}\end{array}\right.
\]
Extending by $\kk$-linearity enables us to take the partial derivative of any element of $\kk Q$. Define the partial derivative of the superpotential by setting $\partial_qW:= \partial_q(e_{\head(q)}W e_{\head(q)})$ for any path $q$. The expression $\partial_q W$ is simply the sum of all paths $p$ in $Q$ with tail at vertex $\head(q)$,  head at vertex $\tail(q)$ and labelling monomial $x^{\div(p)} = x_1x_2\cdots x_{d}/x^{\div(q)}$. For example,  $\partial_{e_i}W$ is the sum of all anticanonical cycles that pass through vertex $i\in Q_0$. Consider now the set of paths
\[
\mathscr{P}:= \left\{ q \text{ in }Q \; \bigg| \begin{array}{c} \partial_q W \text{ is the sum of precisely two paths} \\ \text{ that share neither initial nor final arrow}\end{array} \right\}.
\]
The condition that both summands of $\partial_q W$  share neither initial nor final arrow ensures that neither $\partial_{aq} W$ nor $\partial_{qa} W$ is the sum of precisely two paths for $a\in Q_1$.

\begin{definition}
\label{def:Ftermequiv}
The \emph{ideal of superpotential relations} is the two-sided ideal in $\kk Q$ given by
\[
J_{W}:= \big( p^+-p^-\in \kk Q \mid \exists\; q\in \mathscr{P} \text{ such that }\partial_qW = p^++p^-\big).
\] 
The \emph{superpotential algebra} of $\mathscr{E}$ is $A_{W}:= \kk Q/J_{W}$. Two paths $p_\pm$ in $Q$ are F\emph{-term equivalent} if there is a finite sequence of paths $p_+=p_0, p_1, \dots, p_{k+1}=p_-$ in $Q$ such that for every $0\leq j\leq k$ we have $p_j- p_{j+1}=q_1(p^+-p^-)q_2$ for paths $q_1, q_2$ in $Q$ and some relation $p^+-p^-\in J_W$. 
\end{definition}

\begin{remark}
\begin{enumerate}
\item It is sometimes possible to introduce signs in $W$ so that the relevant partial derivatives of $W$ reproduce precisely the generators of $J_W$.  Indeed, this is part of the defining data for dimer model algebras, and it is demonstrated for skew group algebras by Bocklandt--Schedler--Wemyss~\cite{BSW}. However,  we present in Section~\ref{sec:conjecture} a relatively simple example in dimension four for which this cannot be done. 
\item The F-term equivalence classes of paths form a $\kk$-vector space basis for $A_W$.
\end{enumerate}
\end{remark}

For each generator $p^+-p^-$ of $J_W$,  the paths $p^\pm$ share the same head, tail and label so $J_W$ is contained in the ideal $J_{\mathscr{E}}$. If this inclusion is equality then the toric algebra $A_{\mathscr{E}}$ is isomorphic to the superpotential algebra $A_W$. However,  this need not be the case as we now illustrate.

\begin{example}
\label{exa:superpotentialF1tilting}
We consider three collections on the threefold $X$ from Example~\ref{exa:F1tilting}:
\begin{enumerate}
\item[\one] For the collection $\mathscr{E}$ from Example~\ref{exa:F1tilting}, the quiver of sections from Figure~\ref{fig:tiltingF1}(b), contains six cycles $p$ with $\div(p)=x_1x_2x_3x_4$, giving 
\[
W= a_8a_7a_4a_1 + a_8a_6a_4a_2 +  a_9a_5a_2 + a_9a_7a_3 +a_{10}a_6a_3 + a_{10}a_5a_1.
\]
It is easy to check that $J_W$ equals the ideal $J_\mathscr{E}$ from Example~\ref{exa:F1tilting}, so $A_W\cong A_\mathscr{E}$.
 \item[\two] The quiver of sections of $\mathscr{E}^\prime = \big(\mathcal{O}_X, \mathcal{O}_X(D_1),\mathcal{O}_X(D_4)\big)$ and the list of arrows are both shown in Figure~\ref{fig:F1subandsupertilting}(a). We have $W= a_9a_3 + a_7a_4a_1 + a_6a_4a_2$, but in this case $A_W\not\cong A_{\mathscr{E}^\prime}$ because $a_6a_3 - a_5a_1 \in J_{\mathscr{E}^\prime}\setminus J_W$. 
    \begin{figure}[!ht]
    \centering
   \subfigure[Quiver of sections for $\mathscr{E}^{\prime}$]{
      \psset{unit=1.2cm}
   \begin{pspicture}(-0.2,-0.3)(4.7,3.1)
        \cnodeput(0,0){A}{0}
        \cnodeput(2,0){B}{1} 
        \cnodeput(0,1.5){C}{2}
         \cnodeput(2,3){F}{0} 
        \psset{nodesep=0pt}
        \nccurve[angleA=10,angleB=170]{->}{A}{B}\lput*{:U}{\small{$x_1$}}
        \nccurve[angleA=345,angleB=195]{->}{A}{B}\lput*{:U}{\small{$x_3$}}
        \nccurve[angleA=90,angleB=270]{->}{A}{C}\lput*{:270}{\small{$x_4$}}
        \ncline{->}{B}{C}\lput*{:215}{\small{$x_2$}}
        \nccurve[angleA=90,angleB=270]{->}{B}{F}\lput*{:270}{\small{$x_4^2$}}
        \nccurve[angleA=90,angleB=180]{->}{C}{F}\lput*{:U}{\tiny{$x_1x_4$}}
        \nccurve[angleA=70,angleB=200]{->}{C}{F}\lput*{:U}{\tiny{$x_3x_4$}}
        \nccurve[angleA=50,angleB=220]{->}{C}{F}\lput*{:U}{\tiny{$x_1^2x_2$}}
        \nccurve[angleA=30,angleB=240]{->}{C}{F}\lput*{:U}{\tiny{$x_1x_2x_3$}}
        \nccurve[angleA=5,angleB=260]{->}{C}{F}\lput*{:U}{\tiny{$x_2x_3^2$}}
        \rput(3,0){  \cnodeput(0,0){A}{0}
        \cnodeput(2,0){B}{1} 
        \cnodeput(0,1.5){C}{2}
         \cnodeput(2,3){F}{0} 
        \psset{nodesep=0pt}
        \nccurve[angleA=10,angleB=170]{->}{A}{B}\lput*{:U}{\small{$a_1$}}
        \nccurve[angleA=345,angleB=195]{->}{A}{B}\lput*{:U}{\small{$a_2$}}
        \nccurve[angleA=90,angleB=270]{->}{A}{C}\lput*{:270}{\small{$a_3$}}
        \ncline{->}{B}{C}\lput*{:215}{\small{$a_4$}}
        \nccurve[angleA=90,angleB=270]{->}{B}{F}\lput*{:270}{\small{$a_5$}}
        \nccurve[angleA=90,angleB=180]{->}{C}{F}\lput*{:U}{\tiny{$a_6$}}
        \nccurve[angleA=70,angleB=200]{->}{C}{F}\lput*{:U}{\tiny{$a_7$}}
        \nccurve[angleA=50,angleB=220]{->}{C}{F}\lput*{:U}{\tiny{$a_8$}}
        \nccurve[angleA=30,angleB=240]{->}{C}{F}\lput*{:U}{\tiny{$a_9$}}
        \nccurve[angleA=5,angleB=260]{->}{C}{F}\lput*{:U}{\tiny{$a_{10}$}}}
        \end{pspicture}}
      \qquad \qquad  
       \subfigure[Quiver of sections for $\mathscr{E}^{\prime\prime}$]{
        \psset{unit=1.2cm}
        \begin{pspicture}(-0.2,-0.3)(5,3.1)
        \cnodeput(0,0){A}{0}
        \cnodeput(2,0){B}{1} 
        \cnodeput(0,1.5){C}{2}
        \cnodeput(2,1.5){D}{3}
        \cnodeput(0,3){E}{4}  
         \cnodeput(2,3){F}{0} 
        \psset{nodesep=0pt}
        \nccurve[angleA=10,angleB=170]{->}{A}{B}\lput*{:U}{\small{$x_1$}}
        \nccurve[angleA=345,angleB=195]{->}{A}{B}\lput*{:U}{\small{$x_3$}}
        \nccurve[angleA=90,angleB=270]{->}{A}{C}\lput*{:270}{\small{$x_4$}}
        \ncline{->}{B}{C}\lput*{:215}{\small{$x_2$}}
        \nccurve[angleA=90,angleB=270]{->}{B}{D}\lput*{:270}{\small{$x_4$}}
        \nccurve[angleA=10,angleB=170]{->}{C}{D}\lput*{:U}{\small{$x_1$}}
        \nccurve[angleA=345,angleB=195]{->}{C}{D}\lput*{:U}{\small{$x_3$}}
        \nccurve[angleA=90,angleB=270]{->}{C}{E}\lput*{:270}{\small{$x_4$}}
        \nccurve[angleA=10,angleB=170]{->}{E}{F}\lput*{:U}{\small{$x_1$}}
        \nccurve[angleA=345,angleB=195]{->}{E}{F}\lput*{:U}{\small{$x_3$}}
        \nccurve[angleA=90,angleB=270]{->}{D}{F}\lput*{:270}{\small{$x_4$}}        
        \ncline{->}{D}{E}\lput*{:215}{\small{$x_2$}}
           \rput(3,0){        \cnodeput(0,0){A}{0}
        \cnodeput(2,0){B}{1} 
        \cnodeput(0,1.5){C}{2}
        \cnodeput(2,1.5){D}{3}
        \cnodeput(0,3){E}{4}  
         \cnodeput(2,3){F}{0} 
        \psset{nodesep=0pt}
        \nccurve[angleA=10,angleB=170]{->}{A}{B}\lput*{:U}{\small{$a_1$}}
        \nccurve[angleA=345,angleB=195]{->}{A}{B}\lput*{:U}{\small{$a_2$}}
        \nccurve[angleA=90,angleB=270]{->}{A}{C}\lput*{:270}{\small{$a_3$}}
        \ncline{->}{B}{C}\lput*{:215}{\small{$a_4$}}
        \nccurve[angleA=90,angleB=270]{->}{B}{D}\lput*{:270}{\small{$a_5$}}
        \nccurve[angleA=10,angleB=170]{->}{C}{D}\lput*{:U}{\small{$a_6$}}
        \nccurve[angleA=345,angleB=195]{->}{C}{D}\lput*{:U}{\small{$a_7$}}
        \nccurve[angleA=90,angleB=270]{->}{C}{E}\lput*{:270}{\small{$a_8$}}
        \nccurve[angleA=10,angleB=170]{->}{E}{F}\lput*{:U}{\small{$a_{11}$}}
        \nccurve[angleA=345,angleB=195]{->}{E}{F}\lput*{:U}{\small{$a_{12}$}}
        \nccurve[angleA=90,angleB=270]{->}{D}{F}\lput*{:270}{\small{$a_{10}$}}        
        \ncline{->}{D}{E}\lput*{:215}{\small{$a_{9}$}}}
        \end{pspicture}
         }
    \caption{}
  \label{fig:F1subandsupertilting}
  \end{figure}
   \item[\three] The quiver of sections of $\mathscr{E}^{\prime\prime} = \big(\mathcal{O}_X, \mathcal{O}_X(D_1),\mathcal{O}_X(D_4), \mathcal{O}_X(D_1+D_4), \mathcal{O}_X(2D_4)\big)$ and the list of arrows are both shown in Figure~\ref{fig:F1subandsupertilting}(b).  The superpotential is
    \begin{align*}
    W &=  a_{10}a_7a_4a_1 + a_{10}a_6a_4a_2 + a_{11}a_8a_4a_2+ a_{11}a_9a_5a_2 \\ 
     & \quad+ a_{11}a_9a_7a_3 +a_{12}a_8a_4a_1 + a_{12}a_9a_6a_3 + a_{12}a_9a_5a_1,
     \end{align*}
 and  the superpotential algebra $A_W$ is isomorphic to the toric algebra $A_{\mathscr{E}^{\prime\prime}}$ since 
\[J_W= \left(\begin{array}{c} \! a_5a_1 - a_6a_3,\;    a_7a_4a_1 - a_6a_4a_2, \;   a_5a_2 - a_7a_3, \;   a_{10}a_6 - a_{11}a_8 \;  \\ a_{12}a_9a_6 - a_{11}a_9a_7,\: a_{10}a_7 - a_{12}a_8, \: a_{8}a_4 - a_9a_5  \end{array}\right) = J_{\mathscr{E}^{\prime\prime}}.
\] 
 \end{enumerate}

  \end{example}
  
  \subsection{Consistency}
  The following notion is adapted from that of algebraic consistency given by Broomhead~\cite{Broomhead} for algebras that arise from superpotentials in a dimer model (see Section~\ref{sec:dimers}).
 
 \begin{definition}
 \label{def:consistent}
 A collection $\mathscr{E}$ of rank one reflexive sheaves that encodes a superpotential $W$ is \emph{consistent} if the algebras $A_{\mathscr{E}}$ and $A_W$ are isomorphic. In this case, we say that $A_\mathscr{E}$ is consistent, and write $A$ for brevity if the collection $\mathscr{E}$ is clear from the context.
  \end{definition}

  We begin our study of consistent toric algebras by establishing an important property of the labels on arrows.

\begin{proposition}
\label{prop:div(a)}
 If $A$ is consistent then $x^{\div(a)}$ divides $\prod_{\rho\in \sigma(1)} x_\rho$ for every $a\in Q_1$. 
  \end{proposition}
 \begin{proof}
For $a\in Q_1$,  Lemma~\ref{lem:algebra} implies that $e_{\tail(a)}Ae_{\tail(a)}\cong R$.  We consider two cases. Suppose first that there exists $b\in Q_1\setminus\{a\}$ with $\tail(b)=\tail(a)$. Since $Q$ is strongly connected, there exist paths $p, q$ in $Q$ so that the compositions $pa$ and $qb$ are cycles in $Q$ beginning at vertex $\tail(a)$. Composing in two ways defines cycles $paqb$ and $qbpa$ with the same head, tail and divisor, so $paqb-qbpa\in J_\mathscr{E}$. Consistency forces $J_\mathscr{E}=J_W$, so $paqb$ and $qbpa$ are F-term equivalent. Since $b\neq a$, there must be a relation $p^+-p^-\in J_W$ with $\tail(p^\pm)=\tail(a)$ for which $a$ lies in the support of one of $p^\pm$. Every such relation is obtained as a partial derivative of $W$, so $x^{\div(a)}$ divides $\prod_{\rho\in \sigma(1)} x_\rho$ as required. Suppose otherwise, so there does not exist $b\in Q_1\setminus\{a\}$ with $\tail(b)=\tail(a)$.  Then every cycle in $Q$ from $\tail(a)$ traverses arrow $a$ and hence for every element $u$ in the Hilbert basis of $\sigma^\vee\cap M$, the corresponding monomial $x^u\in R$ is divisible by $x^{\div(a)}$. The monomial $\prod_{\rho\in \sigma(1)} x_\rho$ is a product of such monomials, so $x^{\div(a)}$ divides $\prod_{\rho\in \sigma(1)} x_\rho$ as required.
 \end{proof}

 \begin{corollary}
 \label{cor:allarrowsinW}
  If $A$ is consistent then every arrow in $Q$ arises in an anticanonical cycle and hence in a term of the superpotential $W$.
  \end{corollary}
  
 We reinterpret this result by lifting $Q$ to an $M$-periodic quiver in $\RR^d$ using the sequence \eqref{eqn:Coxsequence}. The \emph{covering quiver} $\widetilde{Q}$ is the quiver with vertex set $\widetilde{Q}_0=\bigoplus_{i\in Q_0} \deg^{-1}(E_i)$, and with arrow set comprising an arrow $\widetilde{a}$ from each $u\in \deg^{-1}(E_i)$ to $u+\div(a)\in \deg^{-1}(E_j)$ for every $a\in Q_1$ from $i$ to $j$. The \emph{label} of $\widetilde{a}$ in $\widetilde{Q}_1$ is the vector $\div(\widetilde{a}):=\head(\widetilde{a})-\tail(\widetilde{a}) =\div(a)\in \NN^d$.

\begin{remark}
\label{rem:QinQuotient}
The quiver $Q$ can be recovered from $\widetilde{Q}$ by taking the quotient by the action of $M$. The given embedding of $\widetilde{Q}$ in $\RR^d=\ZZ^d\otimes_\ZZ \RR$ induces an embedding of $Q$ in $\RR^d/M$. 
\end{remark}

We now lift the anticanonical cycles from $Q\subset \RR^d/M$ to $\widetilde{Q}\subset \RR^d$. For each $u\in \widetilde{Q}_0$, let $p$ be an anticanonical cycle in $Q$ that passes through vertex $i:=\deg(u)\in Q_0$. An \emph{anticanonical path from $u$ covering $p$} is any path $\widetilde{p}_u$ in $\widetilde{Q}$ from $u$ to $u +(1,\dots,1)\in \deg^{-1}(E_i)$ whose image in $\RR^d/M$ is the cycle $p$ in $Q$.  For $u, u^\prime\in \deg^{-1}(E_i)$, the anticanonical paths from $u$ differ only by translation from the anticanonical paths from $u^\prime$, so we need only study paths from one such vertex. For this, pick a fundamental region in $\RR^d$ for the action of $M$ by choosing a spanning tree in $Q$, and lift to a connected tree in $\widetilde{Q}$. Each $i\in Q_0$ then has a preferred lift $u_i\in \deg^{-1}(E_i)$.

\begin{definition}
\label{def:coveringquiver}
For $i\in Q_0$, let $\widetilde{Q}(i)$ be the quiver in $\RR^d$ with vertex set
\[
\widetilde{Q}_0(i):=\Big\{v\in \widetilde{Q}_0 \mid\exists\; \text{anticanonical path }\widetilde{p}_{u_i} \text{ that touches }v\Big\}
\]
and arrow set
\[
\widetilde{Q}_1(i):=\Big\{a\in \widetilde{Q}_1\mid\exists\; \text{anticanonical path }\widetilde{p}_{u_i} \text{ that traverses }a\Big\}.
\]
\end{definition}

\begin{remark}
\label{rem:widetildeQi}
\begin{enumerate}
\item An anticanonical cycle that passes through a vertex more than once gives rise to more than one anticanonical path in a given quiver $\widetilde{Q}(i)$, see Examples~\ref{exa:trivialaction}-\ref{exa:conifold}.
\item Since we lift only anticanonical cycles,  the vertex set $\widetilde{Q}_0(i)$ is a subset of the set of vertices of the unit hypercube 
$\mathsf{C}(u_i):=\{u_i+(\lambda_1,\dots,\lambda_d)\in \RR^d \mid 0\leq \lambda_j\leq 1 \text{ for }1\leq j\leq d\}$.
\end{enumerate}
\end{remark}

 \begin{example}
 \label{exa:trivialaction}
The quiver of sections $Q$ for the trivial collection $\mathscr{E}=(\mathcal{O}_X)$ on $X=\mathbb{A}^n_\kk$ has one vertex and $n$ loops labelled $x_1, \dots, x_n$. There are $(n-1)!$ anticanonical cycles, and each lifts to $n$ anticanonical paths that emanate from each vertex $u\in \ZZ^n$. The support of the quiver $\widetilde{Q}(i)$ in $\RR^n$ is precisely the support of the set of edges of the unit cube $\mathsf{C}(u_0)$.
\end{example}

\begin{example}
\label{exa:conifold}
For the conifold $X=\Spec \kk[x_1,x_2,x_3,x_4]/(x_1x_2-x_3x_4)$, the quiver of sections $Q$ from Figure~\ref{fig:conifold}(a) defines the consistent toric algebra $A$ studied by Szendr\H{o}i~\cite[Figure~1]{Szendroi}.
  \begin{figure}[!ht]
    \centering
      \subfigure[]{
       \psset{unit=0.8cm}
   \begin{pspicture}(0.1,-0.3)(6,4.8)
      \cnodeput(1.2,2.7){A}{\tiny{0}}
      \cnodeput(4.8,2.7){B}{\tiny{1}}
      \nccurve[angleA=60,angleB=120]{->}{A}{B}\lput*{:U}{\small{$x_1$}}
      \nccurve[angleA=30,angleB=150]{->}{A}{B}\lput*{:U}{\small{$x_2$}}
      \nccurve[angleA=210,angleB=330]{->}{B}{A}\lput*{:180}{\small{$x_3$}}
      \nccurve[angleA=240,angleB=300]{->}{B}{A}\lput*{:180}{\small{$x_4$}}
   \end{pspicture}      }
      \quad\quad      \quad\quad
      \subfigure[]{    
       \psset{unit=0.8cm}
       \begin{pspicture}(0.1,-0.3)(6,5.4)
       \cnodeput(3,5.4){E1}{\tiny{0}}
       \cnodeput(2.2,4.2){D2}{\tiny{1}}
       \cnodeput(5.4,4.2){D4}{\tiny{1}}
       \cnodeput(0,2.7){C1}{\tiny{0}}
       \cnodeput(2.4,2.7){C3}{\tiny{0}}
       \cnodeput(3.6,2.7){C4}{\tiny{0}}
       \cnodeput(6,2.7){C6}{\tiny{0}}  
       \cnodeput(5.4,1.2){B4}{\tiny{1}}
       \cnodeput(2.2,1.2){B2}{\tiny{1}}
       \cnodeput(3,0){A1}{\tiny{0}}
      \psset{linecolor=lightgray}
       \cnodeput(0.6,4.2){D1}{\tiny{}}
       \cnodeput(3.8,4.2){D3}{\tiny{}}
       \cnodeput(1.2,2.7){C2}{\tiny{}}
      \cnodeput(4.8,2.7){C5}{\tiny{}}
       \cnodeput(0.6,1.2){B1}{\tiny{}}
       \cnodeput(3.8,1.2){B3}{\tiny{}}    
       \ncline{->}{A1}{B1}\ncline{->}{A1}{B3}\ncline{->}{B1}{C1}\ncline{->}{B1}{C2}
       \ncline{->}{B1}{C4}\ncline{->}{B2}{C1}\ncline{->}{B2}{C5}\ncline{->}{B3}{C2}
       \ncline{->}{B3}{C3}\ncline{->}{B3}{C5}\ncline{->}{B4}{C5}
       \ncline{->}{C1}{D1}\ncline{->}{C2}{D1}\ncline{->}{C2}{D3}
       \ncline{->}{C3}{D1}\ncline{->}{C4}{D2}\ncline{->}{C4}{D3}\ncline{->}{C5}{D2}
       \ncline{->}{C5}{D4}\ncline{->}{C6}{D3}\ncline{->}{D1}{E1}
       \ncline{->}{D2}{E1}\ncline{->}{D3}{E1}
        \psset{linecolor=black,linewidth=0.04}
       \ncline{->}{A1}{B2}\lput*{:180}{\tiny{$x_1$}}
       \ncline{->}{A1}{B4}\lput*{:U}{\tiny{$x_2$}}
       \ncline{->}{B2}{C1}\lput*{:180}{\tiny{$x_3$}}
       \ncline{->}{B2}{C3}\lput*{:U}{\tiny{$x_4$}} 
       \ncline{->}{B4}{C4}\lput*{:180}{\tiny{$x_3$}} 
       \ncline{->}{B4}{C6}\lput*{:U}{\tiny{$x_4$}} 
       \ncline{->}{C1}{D2}\lput*{:U}{\tiny{$x_2$}}
       \ncline{->}{C3}{D4}\lput*{:U}{\tiny{$x_2$}}
       \ncline{->}{C6}{D4}\lput*{:180}{\tiny{$x_1$}}
       \ncline{->}{D2}{E1}\lput*{:U}{\tiny{$x_4$}}
       \ncline{->}{D4}{E1}\lput*{:180}{\tiny{$x_3$}}
       \ncline{->}{C4}{D2}\lput*{:180}{\tiny{$x_1$}}
      \end{pspicture}
          }
  \caption{(a) Quiver of sections; (b) unit 4-cube $\mathsf{C}(u_0)$ in $\RR^4$ containing $\widetilde{Q}(0)$.}
  \label{fig:conifold} 
  \end{figure}  
  The edges of the unit 4-cube $\mathsf{C}(u_0)$ are shown in grey in Figure~\ref{fig:conifold}(b), where the vertices $u_0$ and $u_0+(1,1,1,1)$ are labelled 0 at the bottom and top of the figure respectively. The four anticanonical paths from $u_0$ which cover the pair of anticanonical cycles in $Q$ define the quiver $\widetilde{Q}(0)$ whose vertices and arrows are shown in black in Figure~\ref{fig:conifold}(b). The quiver $\widetilde{Q}(1)$ is similar.
   \end{example}

\begin{corollary}
\label{cor:arrowsinA}
If $A$ is consistent then the vertex set and arrow set of $\widetilde{Q}\subset \RR^d$ coincides with the $M$-translates in $\RR^d$  of the vertex set and arrow set of $\bigcup_{i\in Q_0} \widetilde{Q}(i)$.
\end{corollary}
\begin{proof}
This is little more than a restatement of Corollary~\ref{cor:allarrowsinW}.
\end{proof}

\subsection{Moduli of quiver representations}
%We continue our study of consistent toric algebras by providing an alternative description of the lattice $\ZZ(Q)$. 

A walk $\gamma$ in $Q$ is an alternating sequence $i_l a_l \dotsb a_1 i_{1}$ of vertices $i_1, \dotsc, i_l$ and arrows
$a_1,\dotsc, a_l$ where $a_{k}$ is an arrow between $i_{k}$ and
$i_{k+1}$.   A walk $\gamma$ is closed if $i_1 = i_l$. If $\tail(a_k) = i_{k}$ and $\head(a_k) = i_{k+1}$ then $a_k$ is
a \emph{forward} arrow in $\gamma$; otherwise $\tail(a_k) = i_{k+1}$, $\head(a_k)
= i_{k}$ and $a_k$ is a \emph{backward} arrow.  If $p$ is a path in $Q$ then $p^{-1}$ denotes the walk from $\head(p)$ to $\tail(p)$ that traverses backwards each arrow from the support of $p$. For  a walk $\gamma$ in $Q$ and for $a\in Q_1$, let $\mult_\gamma(a)\in \ZZ$ be the number of times $a$ appears as a forward arrow in $\gamma$ minus the number of times it appears as a backwards arrow. Set $v(\gamma):= \sum_{a\in Q_1} \mult_\gamma(a) \chi_a\in \ZZ^{Q_1}$.  Consider the abelian group 
  \[
  \Lambda = \ZZ^{Q_1}/\big(v(p^+) - v(p^-)\in \ZZ^{Q_1} \mid\exists\; q\in \mathscr{P} \text{ such that }\partial_qW = p^++p^-\big)
  \]
and the quotient map $\wt\colon \ZZ^{Q_1}\to \Lambda$. Define the semigroup $\Lambda_+:= \wt(\NN^{Q_1})$.

 \begin{lemma} 
 \label{lem:Kergens}
 If $A$ is consistent then the maps $\pi\colon \ZZ^{Q_1}\to \ZZ(Q)$ and $\wt\colon \ZZ^{Q_1}\to\Lambda$ coincide. In particular, $A$ is graded by the semigroup $\Lambda_+= \NN(Q)$.
 \end{lemma}
 \begin{proof}
 It suffices to prove that $\Ker(\pi) = L:=\big(v(p^+) - v(p^-)\in \ZZ^{Q_1} :  p^+-p^-\in J_W\big)$. For $p^+-p^-\in J_W$, the paths $p^{\pm}$ share the same head, tail and divisor, so $v(p^+) - v(p^-)\in \Ker(\pi)$. For the opposite inclusion, consider $v\in \Ker(\pi)$. Since $\pi = (\inc, \div)$ there is a closed walk $\gamma$ in $Q$ with $v(\gamma) \in \Ker(\div)$. We now use an `elongation' operation to replace $\gamma$ by a more convenient closed walk $\gamma'$ satisfying $v(\gamma')\in \Ker(\div)$. First, write $\gamma$ as a sequence $\alpha_1^{\,} \alpha_2^{-1} \alpha_3^{\,} \dotsb  \alpha_{2\ell -1}^{\,} \alpha_{2 \ell}^{-1}$ that alternates between paths $\alpha_{2i-1}$ ($1\leq i\leq \ell$) comprising forward arrows, and walks $\alpha_{2i}^{-1}$ ($1\leq i\leq \ell$) comprising backward arrows. For $1\leq i\leq \ell$, choose a path $\beta_{2i-1}$ from $\head(\alpha_{2i-1})$ to $0\in Q_0$ and a path $\beta_{2i}$ from $0\in Q_0$ to $\tail(\alpha_{2i})$.  For $\beta_{0} = \beta_{2 \ell}$, consider
    \[
    \beta_{0}^{\,} \alpha_1^{\,}\beta_1^{\,}
    \beta_1^{-1}\alpha_2^{-1} \beta_{2}^{-1}\beta_2^{\,} \alpha_3^{\,}
    \beta_3^{\,} \beta_3^{-1} \dotsb \alpha_{2 \ell}^{-1}
    \beta_{2\ell}^{-1} \, .
    \]  
   Each composition $p_{2i+1}:=\beta_{2i}\alpha_{2i+1}\beta_{2i+1}$ is a cycle from $0$, while each $p_{2i}^{-1}:= \beta_{2i-1}^{-1}\alpha_{2i}^{-1}\beta_{2i}^{-1}$ is a closed walk from 0 comprising backwards arrows. Set $\gamma' = p_1p_3\cdots p_{2\ell-1}p_{2\ell}^{-1}p_{2\ell-2}^{-1}\dots p_{2}^{-1}$. Since $v(\gamma')=v(\gamma)= v\in \Ker(\pi)$, the cycles $\gamma'_+:=p_1p_3\cdots p_{2\ell-1}$ and $\gamma'_-:=p_2p_4\dots p_{2\ell}$ share the same divisor and hence determine the same element in $A_{\mathscr{E}}$. Since $A_{\mathscr{E}}$ is consistent, the cycles $\gamma'_{\pm}$ determine the same element in $A_W$ so they are F-term equivalent. Thus,  there exists a finite sequence of cycles $\gamma'_+=\gamma_0, \gamma_1,\dotsb, \gamma_{k+1}=\gamma'_-$ such that for each $0\leq j\leq k$, we have $\gamma_j-\gamma_{j+1} = q_1(p^+-p^-)q_2$ for some relation $p^+-p^-\in J_W$ and paths $q_1, q_2$ in $Q$. Expand 
   \[
   v = v(\gamma') = v(\gamma'_+) - v(\gamma'_-) =  \sum_{0\leq j\leq k} \big(v(\gamma_j) - v(\gamma_{j+1})\big).
   \]
 The first statement follows as $v(\gamma_j) - v(\gamma_{j+1})\in L$ for $0\leq j\leq k$. The second statement follows immediately from the first by setting $\deg(p) = \pi(v(p))\in \NN(Q)$ for any path $p$ in $Q$.
  \end{proof}

 To uncover the geometry encoded in a consistent toric algebra we construct fine moduli spaces of quiver representations. For a path $p$ in $Q$ define $y_p:=\prod_{a\in  \supp(p)} y_a\in \kk[\NN^{Q_1}]$ and
 \[ I_W:= \big(y_{p^+}-y_{p^-}\in \kk[\NN^{Q_1}] \mid p^+-p^-\in J_W\big). \]
 This ideal is homogeneous in the $\Wt(Q)$-grading, so the subscheme $\mathbb{V}(I_W)\subseteq \mathbb{A}^{Q_1}_\kk$ cut out by $I_W$ is invariant under the $T$-action from \eqref{eqn:Taction}. For generic $\theta\in \Wt(Q)$,  the GIT quotient 
\[
\mathcal{M}_\theta:= \mathbb{V}(I_W)\git_\theta T = \Proj \Big{(}\bigoplus_{j \geq 0} \big(\kk[\NN^{Q_1}]/I_W)_{j \theta} \Big{)}\]
is the geometric quotient of the open subscheme of $\theta$-stable points of $\mathbb{V}(I_W)$ by the action of $T$. Following King~\cite{King}, $\mathcal{M}_\theta$ is the fine moduli space of isomorphism classes of $\theta$-stable representations of $Q$ with dimension vector $(1,1,\dots,1)\in \NN^{Q_0}$ that satisfy the relations $J_W$.  If a strictly $\theta$-semistable representation does exist, the resulting categorical quotient $\overline{\mathcal{M}_\theta}:= \mathbb{V}(I_W)\git_\theta T$ is merely the coarse moduli space parametrising S-equivalence classes of $\theta$-semistable representations with dimension vector $(1,1,\dots,1)$ that satisfy $J_W$.
    
 \begin{theorem}
 \label{thm:cohcomp}
 Let $A$ be consistent. For generic $\theta\in \Wt(Q)$ there is a commutative diagram
 \[
 \begin{CD}   
  Y_\theta @>>> \mathcal{M}_\theta \\
     @V{\tau_\theta}VV   @VVV    \\
X @>>> \overline{\mathcal{M}_0} \\
 \end{CD}
 \]
 where the horizontal maps are closed immersions and the vertical maps are projective morphisms arising from variation of GIT quotient. Moreover, the toric variety $Y_\theta$ is the unique irreducible component of $\mathcal{M}_\theta$ containing the $T$-orbit closures of points of $\mathbb{V}(I_W)\cap (\kk^\times)^{Q_1}$. 
 \end{theorem}
 
 \begin{remark}
For generic $\theta\in \Wt(Q)$, we call $Y_\theta$ the \emph{coherent component} of $\mathcal{M}_\theta$. 
\end{remark}
\begin{proof}
To construct the diagram, note that for each $p^+-p^-\in J_W$, the paths $p^{\pm}$ share head, tail and divisor, so $I_W\subseteq I_{\mathscr{E}}$. Therefore $Y_\theta\subseteq \overline{\mathcal{M}_\theta}$ for all $\theta\in \Wt(Q)$ which gives the lower horizontal map, and the top map follows since $\overline{\mathcal{M}_\theta} = \mathcal{M}_\theta$ for $\theta$ generic. The vertical maps of the diagram are well known.  Lemma~\ref{lem:Kergens} shows that the vectors $\{v(p^+)-v(p^-)\in \ZZ^{Q_1}\}$ generate the lattice $\Ker(\pi)$. The proof of \cite[Theorem~3.10]{CMT1} applies verbatim to show that $\mathbb{V}(I_{\mathscr{E}})$ is the unique irreducible component of $\mathbb{V}(I_W)$ that does not lie in any coordinate hyperplane of $\mathbb{A}^{Q_1}_\kk$.  The proof of the final statement now follows precisely as in \cite[Theorem~4.3\two]{CMT1}.
\end{proof}

\begin{remark}
 The category of finite dimensional representations satisfying the relations $J_W$ is equivalent to the category $\modAPhi$ of finite dimensional left $A_W$-modules.  For $V=\bigoplus_{i \in Q_0} \kk e_i$, this equivalence takes representations of dimension vector $(1,1,\dots,1)\in \NN^{Q_0}$ to $A_W$-modules that are isomorphic as a $V$-module to $V$.  For a consistent algebra $A$ and for generic $\theta\in \Wt(Q)$, Lemma~\ref{lem:algebra} then implies that $\mathcal{M}_\theta$ is the fine moduli space of $\theta$-stable $\End_R\bigl( \bigoplus_{i \in Q_0} E_i \bigr)$-modules that are isomorphic as a $V$-module to $V$ (see King~\cite{King} for the notion of $\theta$-stability for modules).
 \end{remark}

\section{Cellular resolution for abelian skew group algebras}
\label{sec:McKay}
This section realises the skew group algebra arising from a finite abelian subgroup of $\GL(n,\kk)$ as a consistent toric algebra. The geometry encoded by this toric algebra specialises to geometry that arises in the study of the McKay correspondence. The main result introduces the toric cell complex for the skew group algebra and constructs the cellular resolution in this case.

\subsection{The McKay quiver of sections}
Let $G$ be a finite abelian group of $\GL(n,\Bbbk)$ containing no quasireflections, where $\Bbbk$ is a field of characteristic not dividing the order of $G$.  We may assume that $G$ is contained in the subgroup $(\Bbbk^{\times})^n$ of diagonal matrices with nonzero entries in $\GL(n,\kk)$.   Setting $\rho_i(g)$ to be the $i$th diagonal element of the matrix $g$ defines $n$ elements $\rho_1,\dots,\rho_n$ of the character group $G^{*}=\Hom(G,\Bbbk^{\times})$.  The \emph{McKay quiver} of $G\subset \GL(n,\Bbbk)$ is the quiver $Q$ with vertex set $G^{*}$, and an arrow $a^{\rho}_i$ from $\rho\rho_i$ to $\rho$ for each $\rho \in G^{*}$ and $1 \leq i \leq n$.

The dual action of $G$ on the coordinate ring $\kk[x_1,\dots,x_n]$ of $\mathbb{A}^n_\kk$ defines a $G^*$-grading with $\deg(x_i)=\rho_i$, and the $G$-invariant subalgebra $R=\kk[x_1,\dots,x_n]^G$ defines the normal affine toric variety $X=\Spec R = \mathbb{A}_{\Bbbk}^n/G$. Since $G$ contains no quasireflections, the map assigning to each $\rho \in G^{*}$ the reflexive $R$-module $E_\rho$ spanned over $\kk$ by semi-invariant polynomials of degree $\rho$ defines an isomorphism $\Cl(X)\cong G^*$.  The quiver of sections of the collection 
\begin{equation}
\label{eqn:McKaycollection}
\mathscr{E}=(E_{\rho} \mid \rho \in G^{*})
\end{equation}
on $X$ coincides with the McKay quiver of $G\subset \GL(n,\Bbbk)$, and the labelling monomial of each arrow $a_{i}^{\rho}$ is $x^{\div(a_{i}^{\rho})}=x_i$. It follows that the ideal of relations $J_{\mathscr{E}}$ is generated by elements of the form $a_i^{\rho \rho_j}a_j^{\rho}-a_j^{\rho \rho_i}a_i^{\rho}$ with $\rho \in G^{*}$ and $1 \leq i,j \leq n$. Apply \cite[Proposition~2.8]{CMT2} to obtain:

\begin{lemma}
\label{lem:skewgroup}
The toric algebra $A_{\mathscr{E}}$ is isomorphic to the skew group algebra $\Bbbk[x_1,\dots,x_n]\ast G$.
\end{lemma}

Assume now that $G \subset \SL(n,\Bbbk)$, so $X$ is Gorenstein. Every anticanonical cycle in $Q$ traverses precisely $n$ arrows, one with each labelling monomial $x_i$ for $1\leq i\leq n$, and without loss of generality we choose the final arrow of any such cycle to have labelling monomial $x_n$ and head at vertex $\rho\in G^*$. Thus, every anticanonical cycle can be written uniquely in the form
\[
a_{n}^{\rho}a_{\sigma(n-1)}^{\rho\rho_{\sigma(n-1)}}\cdots   a_{\sigma(2)}^{\rho\rho_{\sigma(2)}\cdots \rho_{\sigma(n-1)}}a_{\sigma(1)}^{\rho\rho_{\sigma(1)}\cdots \rho_{\sigma(n-1)}} 
\]
for some $\rho\in G^*$ and some permutation $\sigma$ on $n-1$ letters, so the superpotential for $\mathscr{E}$ is
\begin{equation}
\label{eqn:McKaysuperpotential}
W=\sum_{\rho \in G^*}\sum_{\sigma \in \mathfrak{S}_{n-1}} a_{n}^{\rho}a_{\sigma(n-1)}^{\rho\rho_{\sigma(n-1)}}\cdots   a_{\sigma(2)}^{\rho\rho_{\sigma(2)}\cdots \rho_{\sigma(n-1)}}a_{\sigma(1)}^{\rho\rho_{\sigma(1)}\cdots \rho_{\sigma(n-1)}}  \in \kk Q_{\mathrm{cyc}},
\end{equation}
where $\mathfrak{S}_{n-1}$ is the set of permutations on $n-1$ letters.  It is straightforward to verify that the ideal of superpotential relations $J_W$ coincides with $J_{\mathscr{E}}$, so we obtain the following result.

\begin{proposition}
The toric algebra $A_{\mathscr{E}}$ is consistent.
\end{proposition}

\begin{remark}
\label{rem:notsaturated}
\begin{enumerate}
\item The superpotential $\Phi$ for the McKay quiver of $G \subset \SL(n,\Bbbk)$ introduced by Bocklandt--Schedler--Wemyss~\cite[\S4]{BSW} counts every anticanonical cycle precisely $n$ times. Thus, ignoring the sign of each term, the superpotential $W$ from \eqref{eqn:McKaysuperpotential} equals $\frac{1}{n}\Phi$.
\item Craw--Maclagan--Thomas~\cite{CMT1} introduce the coherent component $Y_{\theta}$ of the fine moduli space $\mathcal{M}_{\theta}$ of $\theta$-stable McKay quiver representations. For $\mathcal{M}_{\theta} = \ghilb$, this recovers Nakamura's irreducible version  $Y_{\theta}=\hilbg$. Thus, for the subgroup $G\subset \GL(6,\kk)$ of order 625 from \cite[Example~5.7]{CMT2}, the coherent component $\hilbg$ and the variety $\mathbb{V}(I_Q)$ are not normal. In particular, the semigroup $\NN(Q)$ defining $\mathbb{V}(I_Q)$ need not be saturated.
\end{enumerate}
\end{remark}

\subsection{The toric cell complex}
\label{sec:CWcomplex}
A \emph{cell} in a topological space is a subspace that is homeomorphic to the closed $k$-dimensional ball $B^k=\{ x \in \RR^k \mid \|x\| \leq 1\}$ for some $k\in \NN$. We use the term $k$-cell when we wish to make explicit the dimension of the cell. A finite \emph{regular cell complex} $\Delta$ is a finite collection of cells in a Hausdorff topological space $\vert\Delta\vert:=\bigcup_{\eta \in \Delta}\eta$ such that we have each of the following: \one\ $\emptyset \in \Delta$; \two\ the interiors of the nonempty cells partition $\vert\Delta\vert$; and  \three\ the boundary of any cell in $\Delta$ is a union of cells in $\Delta$. Denote by $\Delta_k$ the set of $k$-cells in $\Delta$. The \emph{faces} of a cell $\eta\in \Delta$ are the cells $\eta^\prime$ satisfying $\eta^\prime\subset \eta$, and \emph{facets} of a cell are faces of codimension-one. The prototypical example of a finite regular cell complex is the set of faces of a convex polytope. Note that our cells are the closures of the open cells in the regular cell complexes described in Bruns--Herzog~\cite[\S6.2]{BrunsHerzog}.  

The most important property of a regular cell complex for this article is the existence of an incidence function $\varepsilon\colon \Delta\times\Delta\rightarrow \{0,\pm 1\}$.  To state the definition, recall from \cite[\S6.2]{BrunsHerzog} that regular cell complexes satisfy the following property:
\begin{equation}
\label{eqn:facetsproperty}
\left\{\begin{array}{c} \text{If $\eta\in \Delta_k$ and $\eta^{\prime\prime}\in \Delta_{k-2}$ is a face of $\eta$, there exist precisely two cells} \\
\text{$\eta_1^\prime, \eta_2^\prime\in \Delta_{k-1}$ such that $\eta_j^\prime$ is a face of $\eta$ and $\eta^{\prime\prime}$ is a face of $\eta_j^{\prime}$ for $j=1,2$.}\end{array}\right.
\end{equation}
An \emph{incidence function} on $\Delta$ is a function $\varepsilon \colon \Delta\times \Delta \longrightarrow  \{0,\pm1\}$ such that $\varepsilon(\eta,\eta^\prime)= 0$ unless $\eta^\prime$ is a facet of $\eta$, that $\varepsilon(\eta,\emptyset) = 1$ for all 0-cells $\eta$ and, moreover, that if $\eta\in \Delta_k$ and $\eta^{\prime\prime}\in \Delta_{k-2}$ is a face of $\eta$, then for the cells $\eta_1^\prime, \eta_2^\prime\in \Delta_{k-1}$ from \eqref{eqn:facetsproperty} we have 
\begin{equation}
\label{eqn:signcondition}
\varepsilon(\eta,\eta_1^\prime) \varepsilon(\eta_1^\prime,\eta^{\prime\prime}) +\varepsilon(\eta,\eta_2^\prime) \varepsilon(\eta_2^\prime,\eta^{\prime\prime})= 0.
\end{equation}
Every regular cell complex $\Delta$ admits an incidence function, and any two such differ by the choice of orientation of each cell, see Bruns--Herzog~\cite[Lemma 6.2.1, Theorem 6.2.2]{BrunsHerzog}.

\medskip

We now associate a regular cell complex $\Delta$ to the consistent collection $\mathscr{E}$ from \eqref{eqn:McKaycollection} on the Gorenstein quotient  $X=\mathbb{A}^n_\kk/G$ for a finite abelian subgroup $G\subset \SL(n,\kk)$ of order $r+1$. Let $\{\chi_i \mid 1\leq i\leq n\}$ denote the standard basis of $\ZZ^n$. The short exact sequence \eqref{eqn:Coxsequence} for $X$ is 
 \begin{equation}
 \label{eqn:McKaysequence}
 \begin{CD}   
    0@>>> M @>>> \ZZ^n @>{\deg}>> G^*@>>> 0,
\end{CD}
\end{equation}
where $\deg(\chi_i)=\rho_i$. The covering quiver $\widetilde{Q}\subset \RR^n$ of the McKay quiver $Q$ has vertex set $\widetilde{Q}_0=\ZZ^n$, and for each $u\in \ZZ^n$ there is an arrow from $u$ to $u+\chi_i$ for $1\leq i\leq n$.   Each arrow in $\widetilde{Q}$ is supported on an edge of a unit hypercube $\mathsf{C}(u)\subset \RR^n$. Remark~\ref{rem:QinQuotient} shows that the image of $\widetilde{Q}$ under the natural projection to the real $n$-torus  $\RR^n \rightarrow \mathbb{T}^n:=\RR^n/M$ defines an embedding of the McKay quiver $Q$ in $\mathbb{T}^n$, so each arrow of $Q$ with tail at vertex $\rho\in G^*$ is supported on an edge of the image of an $n$-cell $\mathsf{C}(u)$ in $\mathbb{T}^n$ for some $u\in \deg^{-1}(\rho)$. We let $\Delta(\rho)$ denote the set of all cells in $\mathbb{T}^n$ obtained as the image of a face of the hypercube $\mathsf{C}(u)$ for some $u\in \deg^{-1}(\rho)$. The union $\Delta:=\bigcup_{\rho\in G^*} \Delta(\rho)$ is a regular cell complex in $\mathbb{T}^n$ comprising all cells obtained as the projection to $\mathbb{T}^n$ of the faces of the $r+1$ hypercubes $\Delta(\rho)\subset \RR^n$ for $\rho\in G^*$.

\begin{definition}
\label{def:McKaycellcomplex}
The \emph{toric cell complex} for the McKay quiver $Q$ is the finite regular cell complex $\Delta$ in $\mathbb{T}^n$. We also refer to $\Delta$ as the toric cell complex of the subgroup $G \subset \SL(n,\Bbbk)$ or, equivalently, of the collection $\mathscr{E}$ from \eqref{eqn:McKaycollection}.
\end{definition}

 \begin{lemma}
 \label{lem:McKaybijections}
There are canonical bijections between $\Delta_0$ and $Q_0$, between $\Delta_1$ and $Q_1$, and between $\Delta_2$ and the set $\{a_i^{\rho \rho_j}a_j^{\rho}-a_j^{\rho \rho_i}a_i^{\rho} \mid \rho \in G^{*}, 1 \leq i<j \leq n\}$ of minimal generators of $J_\mathscr{E}$. 
\end{lemma}
\begin{proof}
The bijections for $\Delta_0$ and $\Delta_1$ are described in the construction of $\Delta$ above.  As for the final bijection, the
closed walk in $Q$ obtained by first traversing the path $a_i^{\rho \rho_j}a_j^{\rho}$ with orientation and then traversing the path $a_j^{\rho \rho_i}a_i^{\rho}$ against orientation lifts to a closed walk in $\widetilde{Q}$ that traverses the boundary of a 2-dimensional face $F$ of the cube $\mathsf{C}(u)$ for each $u\in \deg^{-1}(\rho^\prime)$ with $\rho^\prime = \rho\rho_i\rho_j$. The arrows in the boundary of the resulting 2-cell $\eta\in \Delta$ are precisely the arrows in the relation $a_i^{\rho \rho_j}a_j^{\rho}-a_j^{\rho \rho_i}a_i^{\rho}$. Conversely, every 2-cell arises from a unique relation in this way.
\end{proof}

Every cell $\eta\in\Delta$ is the image in $\mathbb{T}^n$ of a face $F\subset\mathsf{C}(u)$ where $u\in \ZZ^n$. Write $u_{\head}$ and $u_{\tail}$ for the vertices of $F$ that intersect the family of affine hyperplanes $H_\lambda:= \{u\in \RR^n \mid \sum_i u_i=\lambda\}$ at the maximum and minimum value of $\lambda$ respectively. The vertices $u_{\head}$ and $u_{\tail}$ depend on the choice of $F$, but their images $\head(\eta)\in \Delta_0$ and $\tail(\eta)\in\Delta_0$ in $\mathbb{T}^n$ do not, and we call these the \emph{head} and \emph{tail} vertices of $\eta$. The \emph{divisor} of $\eta$ is the element $\div(\eta):=u_{\head}-u_{\tail}\in \NN^n$. The following duality property of the toric cell complex is evident from the construction.

\begin{proposition}
\label{prop:dualityMcKay}
The map $\tau\colon \Delta\to \Delta$ that assigns to each $\eta\in \Delta_k$ the unique cell $\eta^\prime\in \Delta_{n-k}$ with $\tail(\eta^\prime)=\head(\eta)$, $\head(\eta^\prime)=\tail(\eta)$ and $x^{\div(\eta^\prime)} = \prod_{\rho\in \sigma(1)}x_\rho/x^{\div(\eta)}$ is an involution.
\end{proposition}

We now introduce the notion of right- and left-differentiation of cells with respect to faces. Let $\eta\in \Delta$. For any face $\eta^\prime\subset \eta$ there is a path in $\widetilde{Q}$ from $\head(\eta^\prime)$ to $\head(\eta)$. While this path need not be unique, its image in $Q$ is a well-defined F-equivalence class of paths that we denote $\overleftarrow{\partial}_{\!\eta'}\eta\in A$. Similarly, there is a path in $\widetilde{Q}$ from $\tail(\eta)$ to $\tail(\eta^\prime)$ that defines an F-equivalence class of paths in $Q$, denoted $\overrightarrow{\partial}_{\!\eta'}\eta\in A$.

\begin{definition}
\label{def:leftrightderivativesmckay}
For $\eta\in \Delta$ and any face $\eta^\prime\subset \eta$, the element $\overleftarrow{\partial}_{\!\eta'}\eta\in A$ is the \emph{left-derivative} of $\eta$ with respect to $\eta'$. Similarly, $\overrightarrow{\partial}_{\!\eta'}\eta\in A$ is the \emph{right-derivative} of $\eta$ with respect to $\eta'$.
\end{definition}

\begin{example}
For group action of type $\frac{1}{6}(1,2,3)$, Figure~\ref{fig:McKayCW} illustrates a fundamental region for $\Delta$ in $\RR^3$ (some $k$-cells are repeated for $k<3$). Observe that $\vert\Delta_0\vert = \vert \Delta_3\vert = 6$ and $\vert \Delta_1\vert = \vert \Delta_2\vert = 18$.\begin{figure}[!ht]
    \centering
       \psset{unit=1.3cm}
   \begin{pspicture}(0,-0.3)(9.9,2)
 \psset{linecolor=black}
\cnodeput(0,0){A}{\tiny{0}} \cnodeput(1.5,0){B}{\tiny{1}}\cnodeput(3,0){C}{\tiny{2}} \cnodeput(4.5,0){D}{\tiny{3}}\cnodeput(6,0){E}{\tiny{4}}\cnodeput(7.5,0){F}{\tiny{5}} \cnodeput(9,0){G}{\tiny{0}} %front vertices of cubes
\cnodeput(0,1.5){A1}{\tiny{3}} \cnodeput(1.5,1.5){B1}{\tiny{4}}\cnodeput(3,1.5){C1}{\tiny{5}} \cnodeput(4.5,1.5){D1}{\tiny{0}}\cnodeput(6,1.5){E1}{\tiny{1}}\cnodeput(7.5,1.5){F1}{\tiny{2}} \cnodeput(9,1.5){G1}{\tiny{3}} %front vertices of cubes
\rput(1,0.4){\cnodeput(0,0){A2}{\tiny{2}} \cnodeput(1.5,0){B2}{\tiny{3}}\cnodeput(3,0){C2}{\tiny{4}} \cnodeput(4.5,0){D2}{\tiny{5}}\cnodeput(6,0){E2}{\tiny{0}}\cnodeput(7.5,0){F2}{\tiny{1}} \cnodeput(9,0){G2}{\tiny{2}} %front vertices of cubes
\cnodeput(0,1.5){A3}{\tiny{5}} \cnodeput(1.5,1.5){B3}{\tiny{0}}\cnodeput(3,1.5){C3}{\tiny{1}} \cnodeput(4.5,1.5){D3}{\tiny{2}}\cnodeput(6,1.5){E3}{\tiny{3}}\cnodeput(7.5,1.5){F3}{\tiny{4}} \cnodeput(9,1.5){G3}{\tiny{5}}} %front vertices of cubes
\ncline{-}{A}{B}\ncline{-}{B}{C}\ncline{-}{C}{D}\ncline{-}{D}{E}\ncline{-}{E}{F}\ncline{-}{F}{G}  %edges of cube
\ncline{-}{A2}{B2}\ncline{-}{B2}{C2}\ncline{-}{C2}{D2}\ncline{-}{D2}{E2}\ncline{-}{E2}{F2}\ncline{-}{F2}{G2}  %edges of cube
\ncline{-}{A1}{B1}\ncline{-}{B1}{C1}\ncline{-}{C1}{D1}\ncline{-}{D1}{E1}\ncline{-}{E1}{F1}\ncline{-}{F1}{G1}  %edges of cube
\ncline{-}{A3}{B3}\ncline{-}{B3}{C3}\ncline{-}{C3}{D3}\ncline{-}{D3}{E3}\ncline{-}{E3}{F3}\ncline{-}{F3}{G3}  %edges of cube
\ncline{-}{A}{A1}\ncline{-}{B}{B1}\ncline{-}{C}{C1}\ncline{-}{D}{D1}\ncline{-}{E}{E1}\ncline{-}{F}{F1}\ncline{-}{G}{G1}  %edges of cube
\ncline{-}{A}{A2}\ncline{-}{B}{B2}\ncline{-}{C}{C2}\ncline{-}{D}{D2}\ncline{-}{E}{E2}\ncline{-}{F}{F2}\ncline{-}{G}{G2}  %edges of cube
\ncline{-}{A1}{A3}\ncline{-}{B1}{B3}\ncline{-}{C1}{C3}\ncline{-}{D1}{D3}\ncline{-}{E1}{E3}\ncline{-}{F1}{F3}\ncline{-}{G1}{G3}  %edges of cube
\ncline{-}{A2}{A3}\ncline{-}{B2}{B3}\ncline{-}{C2}{C3}\ncline{-}{D2}{D3}\ncline{-}{E2}{E3}\ncline{-}{F2}{F3}\ncline{-}{G2}{G3}  %edges of cube
  \end{pspicture}      
    \caption{The McKay cell complex $\Delta$ in $\mathbb{T}^3$ for the action of type $\frac{1}{6}(1,2,3)$}
  \label{fig:McKayCW} 
  \end{figure}
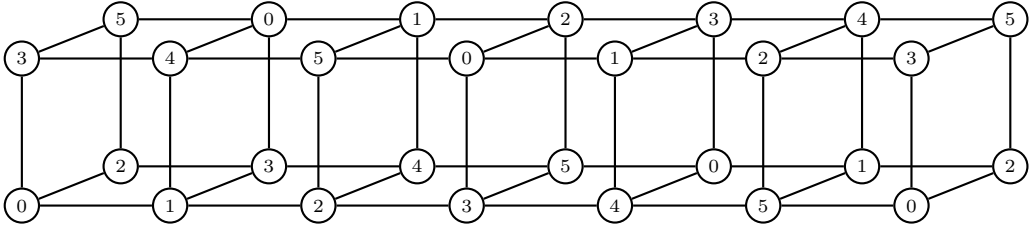
     Let $\eta\in \Delta_3$ denote the 3-cell on the far left of Figure~\ref{fig:McKayCW} and $\eta^\prime\in \Delta_2$ the facet (drawn horizontally) that contains the 0-cells 0, 1, 2, 3. Then $\head(\eta^\prime)=3$ and $\tail(\eta^\prime)=\tail(\eta)=\head(\eta)=0$. The right derivative is $\overrightarrow{\partial}_{\!\eta'}\eta = e_0\in A$ and the left derivative is the vertical arrow $\overleftarrow{\partial}_{\!\eta'}\eta = a_3^0\in A$.
\end{example}

\subsection{The cellular resolution}
For $0 \leq k \leq n$, consider the $\kk$-vector space $U_k=\bigoplus_{\eta \in \Delta_k}\Bbbk \cdot [\eta]$, where $[\eta]$ is a  formal symbol. Lemma~\ref{lem:Kergens} shows that the skew group algebra $A$ admits a natural $\Lambda_{+}$-grading, and since each cell $\eta \in  \Delta_{k}$ is homogeneous we obtain a $\Lambda_{+}$-grading on $U_{k}$. Identify the semisimple algebra $U_0$ with the subalgebra of $\Bbbk Q$ generated by the trivial paths. Note that $U_k$ is a $(U_0,U_0)$-bimodule, and consider the induced $(A,A)$-bimodule
\[
P_k=A \otimes_{U_0} U_k \otimes_{U_0} A=\bigoplus_{\eta \in \Delta_k}Ae_{\head(\eta)} \otimes [\eta] \otimes e_{\tail(\eta)}A.
\]
Note that $P_k$ inherits a $\Lambda_{+}$-grading, called the \emph{total $\Lambda_+$-grading},  in which the degree of a product of homogeneous elements is given by the sum of the degrees in each of the three positions. For $0 \leq k \leq n$, define a morphism $d_k \colon P_k \rightarrow P_{k-1}$ of $\Lambda_+$-graded graded $(A,A)$-bimodules by setting
$$
d_k(1 \otimes [\eta] \otimes 1)=\sum_{\codim(\eta^\prime, \eta) = 1}\varepsilon(\eta,\eta')\; \overleftarrow{\partial}_{\!\eta'}\eta\otimes[\eta'] \otimes \overrightarrow{\partial}_{\!\eta'}\eta,
$$
where $\varepsilon$ is an incidence function on $\Delta$. It is convenient to choose $\varepsilon$ to be compatible with the orientation of arrows in $Q$ as follows.  Identify each $\eta\in \Delta_1$ with an arrow $a\in Q_1$ according to Lemma~\ref{lem:McKaybijections}, so the 1-cell contains precisely two 0-cells $\head(a), \tail(a)\in \Delta_0$. Choosing $\varepsilon(a, \head(a)) = 1$ forces $\varepsilon(a, \tail(a)) = -1$ by \eqref{eqn:signcondition} and hence
$$
d_1(1 \otimes [a] \otimes 1)=1 \otimes [\head(a)] \otimes a  - a \otimes [\tail(a)] \otimes 1.
$$
Let $\mu \colon P_0=A\otimes_{U_0}A \rightarrow A$ denote the multiplication map.

\begin{proposition}
\label{prop:McKaycomplex}
For any choice of incidence function $\varepsilon$, the sequence
\begin{equation}\label{eqn:Koszul}
0 \longrightarrow P_n \xlongrightarrow{d_n} \cdots \xlongrightarrow{d_2} P_1 \xlongrightarrow{d_1} P_0 \xlongrightarrow{\mu} A \longrightarrow 0,
\end{equation}
is a complex of $\Lambda_+$-graded $(A,A)$-bimodules. Moreover, an alternative incidence function determines a new complex that is naturally isomorphic to that from \eqref{eqn:Koszul}. 
\end{proposition}

\begin{proof}
Note that, by our preceding remark, the cokernel of $d_1$ at $P_0$ is just the multiplication map $\mu$. Let us assume $k \geq 2$ and take $\eta \in \Delta_k$. Then 
\begin{align*}
d_{k-1}(d_k(1 \otimes [\eta] \otimes 1)) &= \sum_{\codim(\eta^\prime, \eta) = 1}\varepsilon(\eta,\eta^\prime)  \sum_{\codim(\eta^{\prime\prime}, \eta^\prime) = 1}\varepsilon(\eta',\eta'')  \overleftarrow{\partial}_{\!\eta'}\eta \overleftarrow{\partial}_{\!\eta''}\eta' \otimes[\eta''] \otimes \overrightarrow{\partial}_{\!\eta''}\eta' \overrightarrow{\partial}_{\!\eta'}\eta \\
&= \sum_{\codim(\eta^\prime, \eta) = 1}\sum_{\codim(\eta^{\prime\prime}, \eta^\prime) = 1}\varepsilon(\eta,\eta')\varepsilon(\eta',\eta'') \; \overleftarrow{\partial}_{\!\eta''}\eta \otimes[\eta''] \otimes \overrightarrow{\partial}_{\!\eta''}\eta.
\end{align*}
If $\eta'' \in \Delta_{k-2}$ is a face of $\eta$ then the only contributions in the double sum come from terms involving the facets $\eta'_1, \eta'_2\subset \eta$ from \eqref{eqn:facetsproperty} containing $\eta''$. The above sum is therefore 
$$
d_{k-1}(d_k(1 \otimes [\eta] \otimes 1)) = \sum_{\codim(\eta^{\prime\prime},\eta) = 2}\Big(\varepsilon(\eta,\eta'_1)\varepsilon(\eta'_1,\eta'')+\varepsilon(\eta,\eta'_2)\varepsilon(\eta'_2,\eta'')\Big)   \overleftarrow{\partial}_{\!\eta''}\eta \otimes[\eta''] \otimes \overrightarrow{\partial}_{\!\eta''}\eta,
$$
taken as a sum over codimension-two faces of $\eta$. This sum is zero by equation \eqref{eqn:signcondition}. For the second statement, let $\varepsilon'$ be another incidence function and consider the complex
$$
0 \longrightarrow P_n \xlongrightarrow{d'_n} \cdots \xlongrightarrow{d'_2} P_1 \xlongrightarrow{d'_1} P_0 \xlongrightarrow{\mu} A \longrightarrow 0
$$
determined by $\varepsilon'$. As \cite[Theorem 6.2.2]{BrunsHerzog} records, there exists a global sign function $\delta\colon \Delta \rightarrow \{\pm 1 \}$ such that $\varepsilon'(\eta,\eta')=\delta(\eta')\varepsilon(\eta,\eta')\delta(\eta)$ for all $\eta \in \Delta_k$, $\eta' \in \Delta_{k-1}$, $0 \leq k \leq n$. This implies that the bimodule homomorphisms $\phi_k\colon P_k \rightarrow P_k$ given by $\phi_k(1 \otimes [\eta]\otimes 1)=\delta(\eta)\otimes [\eta]\otimes 1$ define a chain map of complexes $\phi\lbdot\!\! \colon (P\lbdot\!\!,d\lbdot\!\!)  \rightarrow (P\lbdot\!\!,d_{{^{^{_{_{_{_{\bullet}}}}}}}}'\!\!)$. Since each $\phi_k$ is an isomorphism of $(A,A)$-bimodules, the chain map $\phi\lbdot\!\!$ is an isomorphism of complexes. This completes the proof.
\end{proof}

To demonstrate that the complex \eqref{eqn:Koszul} is a minimal projective $(A,A)$-bimodule resolution of $A$ we choose a suitable incidence function $\varepsilon$ on $\Delta$. For $\eta \in \Delta_k$, let $\eta' \subset \eta$ be a facet. We may write $x^{\div(\eta)}=x_{i_1}\cdots x_{i_k}$ and $x^{\div(\eta')}=x_{i_1}\cdots \widehat{x_{i_{\nu}}}\cdots x_{i_k}$ for $i_1<\cdots <i_k$, where $ \widehat{x_{i_{\nu}}}$ means that the factor $x_{i_{\nu}}$ is removed. We determine an incidence function on $\Delta$ by setting
\[
\varepsilon(\eta,\eta')=\left\{
\begin{array}{ll}
(-1)^{\nu} & \text{if $\head(\eta)=\head(\eta')$ and $x^{\div( \overleftarrow{\partial}_{\!\eta'}\eta)}=x_{i_{\nu}}$;} \\
(-1)^{\nu +1} & \text{if $\tail(\eta)=\tail(\eta')$ and $x^{\div(  \overrightarrow{\partial}_{\!\eta'}\eta)}=x_{i_{\nu}}$.}
\end{array} \right.
\]
For a fixed $\eta \in \Delta_k$ with head at $\rho:=\head(\eta)$, let $\eta'_1,\dots,\eta'_k$ denote the facets of $\eta$ with head at $\rho$; similarly, let $\overline{\eta}'_1,\dots,\overline{\eta}'_k$ denote the facets of $\eta$ with tail at $\tail(\eta)$. For the above choice of $\varepsilon$, the differential $d_k\colon P_k \rightarrow P_{k-1}$ is given by
\begin{equation}
\label{eqn:revisedcomplex}
d_k(1\otimes [\eta]\otimes 1)=\sum_{\nu=1}^k(-1)^{\nu} a_{i_{\nu}}^{\rho} \otimes [\eta'_{\nu}]\otimes 1 + \sum_{\nu=1}^k(-1)^{\nu+1} \otimes [\overline{\eta}'_{\nu}]\otimes a_{i_{\nu}}^{\rho \rho_{i_{1}}\cdots \widehat{\rho_{i_{\nu}}}\cdots \rho_{i_{k}}} .
\end{equation}

\begin{theorem}
\label{thm:McKayresolution}
The complex \eqref{eqn:Koszul} is a minimal projective $(A,A)$-bimodule resolution of $A$; this is the \emph{cellular resolution} of $A$.
\end{theorem}
\begin{proof}
Lemma~\ref{lem:skewgroup} implies that $A$ is Koszul, and Proposition~\ref{prop:McKaycomplex} ensures that we need only show that the complex with differentials given by \eqref{eqn:revisedcomplex} coincides with the bimodule Koszul complex. One approach is to write down explicitly the isomorphism between \eqref{eqn:Koszul} and the bimodule Koszul complex as presented, for example,  in Bocklandt--Schedler--Wemyss~\cite[Lemma~6.1]{BSW}. More directly,  by extending the ground category from vector spaces to $(U_0,U_0)$-bimodules, the bimodule Koszul complex constructed by Taylor~\cite[Equation (4.4)]{Taylor} provides the bimodule Koszul complex of $A$. For the basis $\chi_1,\dots, \chi_n$ of $V:=\kk^n$ and for a cell $\eta \in \Delta_{k}$ with $x^{\div(\eta)}=x_{i_1}\cdots x_{i_k}$, the assignment  $[\eta] \mapsto \chi_{i_1}\wedge \dots \wedge \chi_{i_k}$ determines an isomorphism from $U_k$ to $\bigwedge^k V$. It follows that the projective $(A,A)$-bimodules $P_k$ are those from the bimodule Koszul complex presented in \cite[Equation (4.4)]{Taylor}. In addition, our choice of signs in the differentials from \eqref{eqn:revisedcomplex} recovers those in Equations (4.1) and (4.4) from \cite{Taylor}. This completes the proof.
\end{proof}

  \section{Algebraically consistent dimer models}
  \label{sec:dimers}
 This section interprets algebraically consistent dimer models as consistent toric algebras.  The main result reproduces the $(A,A)$-bimodule resolution of $A$ from \cite{Broomhead, Davison, MozgovoyReineke} as a cellular resolution, and reconstructs the subdivision of the real two-torus determined by the dimer model in terms of anticanonical cycles in the quiver. The key step associates a label to each arrow in the quiver of the dimer model. The first appearance of this technique in the dimer model literature seems to be Eager~\cite[\S6]{Eager}.

 \subsection{On dimer models}
  A \emph{dimer model} $\Gamma$ on a torus is a polygonal cell decomposition of the surface of a real two-torus whose vertices and edges form a bipartite graph. Each vertex may be coloured either black or white so that each edge joins a black vertex to a white vertex. The dual cell decomposition of the torus has a vertex dual to every face, an edge dual to every edge, and face dual to every vertex of $\Gamma$.  In addition, we orient the edges of this dual decomposition so that a white vertex of the dimer lies on the left as the arrow crosses the dual edge of the dimer. The vertices and edges of this dual decomposition therefore define a quiver $Q=(Q_0, Q_1)$ embedded in the two-torus, with the additional property that the set of faces decomposes as the union $Q_2 = Q_2^+\cup Q_2^-$ of white faces (oriented anticlockwise) and black faces (oriented clockwise). 
  
To each face $F\in Q_2$ we associate the cycle  $w_F\in \kk Q_{\mathrm{cyc}}$ obtained by tracing all arrows around the boundary of $F$. The \emph{superpotential} of the dimer model $\Gamma$ is defined to be
  \[
  W_\Gamma := \sum_{F\in Q_2^+} w_F - \sum_{F\in Q_2^-} w_F.
  \]
For any face $F\in Q_2$ and arrow $a\in \supp(w_F)$,  choose $\head(a)\in Q_0$ as the starting point of $w_F$ and write $e_{\head(a)}w_Fe_{\head(a)} = aa_{l}\cdots a_1$. The partial derivative of the cycle $w_F$ with respect to $a$ is the path $\partial_a w_F = a_l\cdots a_1$ in $Q$. Extending $\kk$-linearly gives $\partial_a W_\Gamma\in \kk Q$ for each $a\in Q_1$, and consider the two-sided ideal $J_\Gamma:= (\partial_a W_\Gamma \mid a\in Q_1)$ in $\kk Q$. The \emph{superpotential algebra} of $\Gamma$ is
 \[
 A_\Gamma:=\kk Q/J_{\Gamma}.
 \]
 A \emph{perfect matching} $\Pi$ of the dimer model is a subset of the edges in $\Gamma$ such that every vertex is the endpoint of precisely one edge. Let $\supp(\Pi)$ denote the subset of $Q_1$ dual to the edges in $\Pi$. Since the arrows arising in any given term $\pm w_F$ of $W_\Gamma$ are dual to the set of edges emanating from the corresponding vertex of $\Gamma$, one can rewrite the superpotential in terms of any perfect matching $\Pi$ as $W_\Gamma = \sum_{a\in \supp(\Pi)} a\cdot \partial_a W$. Every arrow $a\in Q_1$ occurs in precisely two oppositely oriented faces, so every relation can be written as a path difference $\partial_aW_\Gamma = p_a^+ - p_a^-$, where $p_a^{\pm}$ are paths with tail at $\head(a)$ and head at $\tail(a)$. The binomials $\{p_a^+ -p_a^- \in \kk Q\mid a\in Q_1\}$ are the F\emph{-term relations} of $\Gamma$, and two paths $p_\pm$ in $Q$ are said to be F\emph{-term equivalent} if there is a finite sequence of paths $p_+=p_0, p_1, \dots, p_{k+1}=p_-$ in $Q$ such that for each $0\leq j\leq k$ we have $p_j- p_{j+1}=q_1(p_a^+-p_a^-)q_2$ for some paths $q_1, q_2$ in $Q$ and for some arrow $a\in Q_1$.   The F-term equivalence classes of paths form a $\kk$-vector space basis for $A_\Gamma$.

 Several notions of consistency for dimer models have been introduced in the literature, and here we consider that of algebraic consistency due to Broomhead~\cite{Broomhead}.  Put simply, a dimer model is \emph{algebraically consistent} if $A_\Gamma$ is isomorphic to an auxilliary algebra constructed from toric data encoded by $\Gamma$. We choose not to reconstruct this toric data here (though see the proof of Proposition~\ref{prop:dimerlabels}), but we do recall results of Broomhead~\cite{Broomhead} showing that for each algebraically consistent dimer model $\Gamma$ the centre of $A$ is a Gorenstein semigroup algebra $R=\kk[\sigma^\vee\cap M]$ of dimension three and, moreover, that there exists a collection of rank one reflexive $R$-modules $(\mathcal{B}_i \mid i\in Q_0)$ such that $A_\Gamma \cong \End_{R}(\bigoplus_{i\in Q_0} \mathcal{B}_i)$. To obtain our preferred normalisation, choose a vertex $0\in Q_0$ and consider instead the $R$-modules $E_i:= \mathcal{B}_i\otimes \mathcal{B}_0^{-1}$ for $i\in Q_0$. Thus, every algebraically consistent dimer model $\Gamma$ defines a collection of rank one reflexive sheaves
 \begin{equation}
 \label{eqn:dimercollection}
 \mathscr{E}:=(E_i \mid i\in Q_0)
 \end{equation} 
 on the Gorenstein toric variety $X:=\Spec R$ such that $A_\Gamma \cong \End_{R}(\bigoplus_{i\in Q_0} E_i)$. 

  \begin{lemma}
  \label{lem:dimerqos}
 The quiver $Q$ arising from an algebraically consistent dimer model $\Gamma$ is the quiver of sections of the collection $\mathscr{E}$ from \eqref{eqn:dimercollection}. In particular, $A_\Gamma\cong A_\mathscr{E}$. 
 \end{lemma}
\begin{proof}
 Let $Q^\prime$ denote the quiver of sections of $\mathscr{E}$, so $Q_0=Q_0^\prime$ by construction. For any $i,j\in Q_0$, the isomorphism $A_\Gamma \cong \End_{R}(\bigoplus_{i\in Q_0} E_i)$ implies $e_jA_\Gamma e_i \cong \Hom_R(E_i,E_j)$. The set of arrows in $Q$ from $i$ to $j$ provides a basis for the space spanned by irreducible elements of $e_jA_\Gamma e_i$, while the set of arrows in $Q^\prime$ from $i$ to $j$ does likewise for $\Hom_R(E_i,E_j)$. This gives $Q_1^\prime = Q_1$ as required.  The final statement follows from Lemma~\ref{lem:algebra}.
 \end{proof}

  \subsection{Labels on arrows in a dimer model}
We may not deduce from Lemma~\ref{lem:dimerqos} that $\mathscr{E}$ is consistent because the dimer model algebra $A_\Gamma$ is not a priori isomorphic to the superpotential algebra $A_W$ determined by the collection $\mathscr{E}$. To establish the link between $A_\Gamma$ and $A_W$ we investigate the labelling of arrows in $Q$.  To begin we present an example that illustrates how our labelling of arrows in $Q$ ties in with the traditional approach to a dimer model $\Gamma$.
   
 \begin{example}
 \label{exa:dimerF1}
 Consider the dimer model $\Gamma$ on the real two-torus shown in Figure~\ref{fig:dimerF1}(a) and the quiver $Q$ embedded in the dual cell decomposition from Figure~\ref{fig:dimerF1}(b). 
  \begin{figure}[!ht]
    \centering
    \mbox{
    \subfigure[]{
    \psset{unit=1.25cm}
   \begin{pspicture}(0.2,0)(3,2.9)
   \pnode(0,0){A}  \pnode(3,0){B} \pnode(0,3){C}  \pnode(3,3){D} %vertices of square
   \ncline{-}{A}{B}  \ncline{-}{A}{C}\ncline{-}{C}{D}  \ncline{-}{D}{B} %edges of square
   \pnode(1.5,0){K} \pnode(0,1.5){L} \pnode(1.5,3){M} \pnode(3,1.5){N} %midpoints on square
  \cnode[fillcolor=white](1.5,0.5){0.1}{E}\cnode[fillstyle=solid,fillcolor=black](1.5,2.5){0.1}{F} %dimer vertices
  \cnode[fillcolor=white](0.5,2.5){0.1}{G}\cnode[fillstyle=solid,fillcolor=black](0.5,1.5){0.1}{H}
  \cnode[fillcolor=white](2.5,1.5){0.1}{I}\cnode[fillstyle=solid,fillcolor=black](2.5,0.5){0.1}{J}
   \psset{linecolor=lightgray,linewidth=0.025}
    \cnodeput(2.4,2.3){Q}{ }\cnodeput(1.95,1.15){O}{ }\cnodeput(1.05,1.85){P}{ }\cnodeput(0.6,0.6){R}{ }%vertices of quiver
   \pnode(2.2,0){S} \pnode(2.2,3){T} \pnode(1.5,3){U} \pnode(3,1.5){V}%quiver intersects square
    \pnode(3,2.2){W} \pnode(0,2.2){X}\pnode(3,0.8){Y} \pnode(0,0.8){Z}\pnode(0.8,3){ZA} \pnode(0.8,0){ZB} 
  \ncline{->}{Q}{O}
  \ncline{-}{Q}{T}\ncline{->}{S}{O}
  \ncline{-}{Q}{W}\ncline{->}{X}{P}
  \ncline{->}{O}{P}
   \ncline{-}{O}{Y}\ncline{->}{Z}{R}
  \ncline{->}{P}{R}
  \ncline{-}{P}{ZA}\ncline{->}{ZB}{R}
  \ncline{-}{R}{K} \ncline{->}{M}{Q}
    \ncline{-}{R}{A}\ncline{->}{D}{Q}
  \ncline{-}{R}{L}\ncline{->}{N}{Q}
   \psset{linewidth=0.05, linecolor=black}
  \ncline{-}{E}{K}\ncline{-}{E}{J}\ncline{-}{E}{H}\ncline{-}{L}{H}\ncline{-}{H}{G}\ncline{-}{C}{G} %dimer edges
  \ncline{-}{F}{G}\ncline{-}{F}{M}  \ncline{-}{F}{I}\ncline{-}{I}{N} \ncline{-}{I}{J}\ncline{-}{B}{J}\ncline{E}{F}
     \end{pspicture}}
      \qquad  \qquad
      \subfigure[]{
          \psset{unit=1.25cm}   
  \begin{pspicture}(0.2,0)(3,2.9)
   \pnode(0,0){A}  \pnode(3,0){B} \pnode(0,3){C}  \pnode(3,3){D} %vertices of square
   \ncline{-}{A}{B}  \ncline{-}{A}{C}\ncline{-}{C}{D}  \ncline{-}{D}{B} %edges of square
   \pnode(1.5,0){K} \pnode(0,1.5){L} \pnode(1.5,3){M} \pnode(3,1.5){N} %midpoints on square
  \psset{linecolor=gray}
  \cnode[fillcolor=white](1.5,0.5){0.1}{E}\cnode[fillstyle=solid,fillcolor=lightgray](1.5,2.5){0.1}{F} %dimer vertices
  \cnode[fillcolor=white](0.5,2.5){0.1}{G}\cnode[fillstyle=solid,fillcolor=lightgray](0.5,1.5){0.1}{H}
  \cnode[fillcolor=white](2.5,1.5){0.1}{I}\cnode[fillstyle=solid,fillcolor=lightgray](2.5,0.5){0.1}{J}
  \ncline{-}{E}{K}\ncline{-}{E}{J}\ncline{-}{E}{H}\ncline{-}{L}{H}\ncline{-}{H}{G}\ncline{-}{C}{G} %dimer edges
  \ncline{-}{F}{G}\ncline{-}{F}{M}  \ncline{-}{F}{I}\ncline{-}{I}{N} \ncline{-}{I}{J}\ncline{-}{B}{J}\ncline{E}{F}
  \psset{linecolor=black}
  \cnodeput(2.4,2.3){Q}{0}\cnodeput(1.95,1.15){O}{1}\cnodeput(1.05,1.85){P}{2}\cnodeput(0.6,0.6){R}{3}%vertices of quiver
   \pnode(2.2,0){S} \pnode(2.2,3){T} \pnode(1.5,3){U} \pnode(3,1.5){V}%quiver intersects square
    \pnode(3,2.2){W} \pnode(0,2.2){X}\pnode(3,0.8){Y} \pnode(0,0.8){Z}\pnode(0.8,3){ZA} \pnode(0.8,0){ZB} 
  \ncline{->}{Q}{O}\lput*{:110}{\tiny{1}}
  \ncline{-}{Q}{T}\ncline{->}{S}{O}\lput*{:260}{\tiny{2}}
  \ncline{-}{Q}{W}\ncline{->}{X}{P}\lput*{:20}{\tiny{3}}
  \ncline{->}{O}{P}\lput*{:215}{\tiny{4}}
   \ncline{-}{O}{Y}\lput*{:20}{\tiny{5}}\ncline{->}{Z}{R}
  \ncline{->}{P}{R}\lput*{:110}{\tiny{6}}
  \ncline{-}{P}{ZA}\lput*{:260}{\tiny{7}}\ncline{->}{ZB}{R}
  \ncline{-}{R}{K} \ncline{->}{M}{Q}\lput*{:35}{\tiny{8}}
  \ncline{-}{R}{A}\ncline{->}{D}{Q}\lput*{:130}{\tiny{9}}
  \ncline{-}{R}{L}\ncline{->}{N}{Q}\lput*{:230}{\tiny{10}}
  \end{pspicture}}
     }
 \caption{(a) a dimer model $\Gamma$; (b) the quiver $Q$ in the dual cell decomposition}
 \label{fig:dimerF1}
 \end{figure}
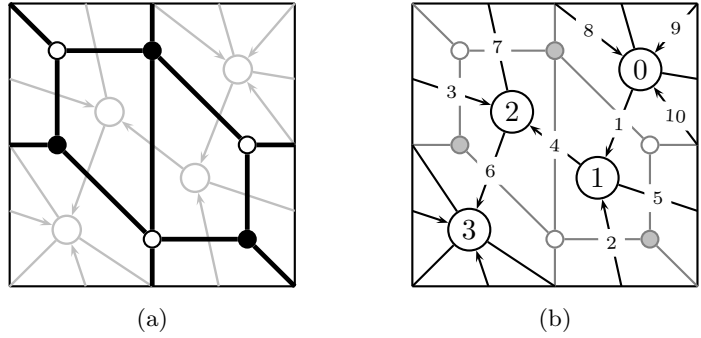
 Notice that $Q$ coincides with the quiver from Figure~\ref{fig:tiltingF1}(c), and we list the arrows $a_1,\dots,a_{10}$ in the same way. It is well known that the semigroup algebra $R=\kk[\sigma^\vee\cap M]$ arising from $\Gamma$ is determined by the cone $\sigma$ over the lattice polygon $P$ from Figure~\ref{fig:tiltingF1}(a). The extremal perfect matchings $\Pi_1, \Pi_2, \Pi_3, \Pi_4$ that correspond to the vertices $v_1, v_2, v_3, v_4\in P$ respectively are shown in Figure~\ref{fig:perfectmatchingsF1}. 
    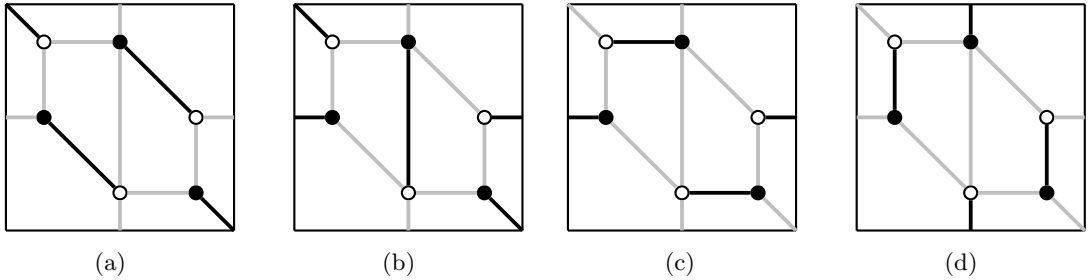
\begin{figure}[!ht]
    \centering
    \mbox{
    \subfigure[]{
    \psset{unit=1cm}
   \begin{pspicture}(0.5,0)(2.5,2.9)
   \pnode(0,0){A}  \pnode(3,0){B} \pnode(0,3){C}  \pnode(3,3){D} %vertices of square
   \ncline{-}{A}{B}  \ncline{-}{A}{C}\ncline{-}{C}{D}  \ncline{-}{D}{B} %edges of square
   \pnode(1.5,0){K} \pnode(0,1.5){L} \pnode(1.5,3){M} \pnode(3,1.5){N} %midpoints on square
  \cnode[fillcolor=white](1.5,0.5){0.1}{E}\cnode[fillstyle=solid,fillcolor=black](1.5,2.5){0.1}{F} %dimer vertices
  \cnode[fillcolor=white](0.5,2.5){0.1}{G}\cnode[fillstyle=solid,fillcolor=black](0.5,1.5){0.1}{H}
  \cnode[fillcolor=white](2.5,1.5){0.1}{I}\cnode[fillstyle=solid,fillcolor=black](2.5,0.5){0.1}{J}
 \psset{linewidth=0.05}\ncline{-}{E}{H}\ncline{-}{C}{G} %dimer edges
\ncline{-}{F}{I}\ncline{-}{B}{J}
  \psset{linecolor=lightgray}
    \ncline{-}{F}{G}\ncline{-}{E}{J}\ncline{-}{L}{H}\ncline{-}{I}{N} \ncline{-}{I}{J}\ncline{E}{F}\ncline{-}{H}{G}\ncline{-}{F}{M}    \ncline{-}{E}{K}
      \end{pspicture}}
      \qquad  \qquad
      \subfigure[]{
          \psset{unit=1cm}  
    \begin{pspicture}(0.5,0)(2.5,2.9)
   \pnode(0,0){A}  \pnode(3,0){B} \pnode(0,3){C}  \pnode(3,3){D} %vertices of square
   \ncline{-}{A}{B}  \ncline{-}{A}{C}\ncline{-}{C}{D}  \ncline{-}{D}{B} %edges of square
   \pnode(1.5,0){K} \pnode(0,1.5){L} \pnode(1.5,3){M} \pnode(3,1.5){N} %midpoints on square
  \cnode[fillcolor=white](1.5,0.5){0.1}{E}\cnode[fillstyle=solid,fillcolor=black](1.5,2.5){0.1}{F} %dimer vertices
  \cnode[fillcolor=white](0.5,2.5){0.1}{G}\cnode[fillstyle=solid,fillcolor=black](0.5,1.5){0.1}{H}
  \cnode[fillcolor=white](2.5,1.5){0.1}{I}\cnode[fillstyle=solid,fillcolor=black](2.5,0.5){0.1}{J}
 \psset{linecolor=black, linewidth=0.05} \ncline{-}{C}{G}\ncline{-}{B}{J}\ncline{-}{L}{H} \ncline{-}{I}{N} \ncline{E}{F}%dimer edges
  \psset{linecolor=lightgray}
  \ncline{-}{E}{K}\ncline{-}{E}{J}\ncline{-}{E}{H}\ncline{-}{H}{G}
  \ncline{-}{F}{G}\ncline{-}{F}{M}  \ncline{-}{F}{I} \ncline{-}{I}{J}
   \end{pspicture}}
      \qquad  \qquad
      \subfigure[]{
          \psset{unit=1cm}  
    \begin{pspicture}(0.7,0)(2.5,3)
   \pnode(0,0){A}  \pnode(3,0){B} \pnode(0,3){C}  \pnode(3,3){D} %vertices of square
   \ncline{-}{A}{B}  \ncline{-}{A}{C}\ncline{-}{C}{D}  \ncline{-}{D}{B} %edges of square
   \pnode(1.5,0){K} \pnode(0,1.5){L} \pnode(1.5,3){M} \pnode(3,1.5){N} %midpoints on square
  \cnode[fillcolor=white](1.5,0.5){0.1}{E}\cnode[fillstyle=solid,fillcolor=black](1.5,2.5){0.1}{F} %dimer vertices
  \cnode[fillcolor=white](0.5,2.5){0.1}{G}\cnode[fillstyle=solid,fillcolor=black](0.5,1.5){0.1}{H}
  \cnode[fillcolor=white](2.5,1.5){0.1}{I}\cnode[fillstyle=solid,fillcolor=black](2.5,0.5){0.1}{J}
 \psset{linecolor=black, linewidth=0.05} \ncline{-}{F}{G}\ncline{-}{L}{H} \ncline{-}{I}{N} \ncline{-}{E}{J}%dimer edges
  \psset{linecolor=lightgray}
  \ncline{-}{E}{K}\ncline{-}{E}{H}\ncline{-}{H}{G}\ncline{-}{C}{G} 
 \ncline{-}{F}{M}  \ncline{-}{F}{I} \ncline{-}{I}{J}\ncline{-}{B}{J}\ncline{E}{F}
     \end{pspicture}}
      \qquad  \qquad
      \subfigure[]{
          \psset{unit=1cm}  
    \begin{pspicture}(0.5,0)(2.5,2.9)
   \pnode(0,0){A}  \pnode(3,0){B} \pnode(0,3){C}  \pnode(3,3){D} %vertices of square
   \ncline{-}{A}{B}  \ncline{-}{A}{C}\ncline{-}{C}{D}  \ncline{-}{D}{B} %edges of square
   \pnode(1.5,0){K} \pnode(0,1.5){L} \pnode(1.5,3){M} \pnode(3,1.5){N} %midpoints on square
  \cnode[fillcolor=white](1.5,0.5){0.1}{E}\cnode[fillstyle=solid,fillcolor=black](1.5,2.5){0.1}{F} %dimer vertices
  \cnode[fillcolor=white](0.5,2.5){0.1}{G}\cnode[fillstyle=solid,fillcolor=black](0.5,1.5){0.1}{H}
  \cnode[fillcolor=white](2.5,1.5){0.1}{I}\cnode[fillstyle=solid,fillcolor=black](2.5,0.5){0.1}{J}
 \psset{linecolor=black, linewidth=0.05} 
\ncline{-}{H}{G}  \ncline{-}{E}{K} \ncline{-}{I}{J}\ncline{-}{F}{M}%dimer edges
 \psset{linecolor=lightgray}
 \ncline{-}{E}{J}\ncline{-}{E}{H}\ncline{-}{L}{H}\ncline{-}{C}{G} 
  \ncline{-}{F}{G}  \ncline{-}{F}{I}\ncline{-}{I}{N}\ncline{-}{B}{J}\ncline{E}{F}
    \end{pspicture}}
     }
 \caption{Perfect matchings: (a) $\Pi_1$; (b) $\Pi_2$; (c) $\Pi_3$; (d) $\Pi_4$.}
 \label{fig:perfectmatchingsF1}
 \end{figure} 
To compute the labels, note from Figure~\ref{fig:dimerF1}(b) that $\supp(\Pi_1) = \{a_1, a_6, a_9\}$. Since $\Pi_1$ is the only extremal perfect matching containing either $a_1$ or $a_6$, Proposition~\ref{prop:perfmatchlabels} implies that both $x^{\div(a_1)}$ and $x^{\div(a_6)}$ are pure powers of $x_1$, whereas $a_9\in \supp(\Pi_1)\cap\supp(\Pi_2)$, so $x_1x_2$ divides $x^{\div(a_9)}$. Lemma~\ref{lem:multiplicity} below shows that the labelling monomial on each arrow in a dimer model is reduced, so $x^{\div(a_1)}=x^{\div(a_6)}=x_1$ and $x^{\div(a_9)}=x_1x_2$. It is now easy to see that the labelling monomials on the arrows of $Q$ are precisely those from Figure~\ref{fig:tiltingF1}(b). The superpotential 
\[
W_\Gamma= -a_8a_7a_4a_1 + a_8a_6a_4a_2 -  a_9a_5a_2 + a_9a_7a_3 -a_{10}a_6a_3 + a_{10}a_5a_1
\]
coincides, up to the sign of each term,  with that from Example~\ref{exa:F1tilting} and hence $A_\Gamma\cong A_W$.
\end{example}

  \begin{proposition}
  \label{prop:dimerlabels}
 Let $Q$ denote the quiver arising from an algebraically consistent dimer model $\Gamma$, and let $\mathscr{E}$ be the collection from \eqref{eqn:dimercollection}. Then each $a\in Q_1$ satisfies
 \begin{equation}
 \label{eqn:label}
 x^{\div(a)} =  \prod_{\{\rho\in \sigma(1) \mid a\in \supp(\Pi_\rho)\}} x_\rho,
 \end{equation}
  where $\Pi_\rho$ is the unique perfect matching of $\Gamma$ corresponding to the vertex $\rho\in \sigma(1)$.
\end{proposition}
\begin{proof}
The quiver of sections $Q$ of $\mathscr{E}$ encodes the commutative diagram from Lemma~\ref{lem:diagram}. For each $a\in Q_1$, the F-term relation $p_a^+ -p_a^-$ is a binomial contained in the defining ideal of $A_{\Gamma}$.  Lemma~\ref{lem:dimerqos} gives $A_\Gamma\cong A_\mathscr{E}$, so $p_a^+ -p_a^-\in J_\mathscr{E}$ and we deduce that $p_a^{\pm}$ share not only the same head and tail but also the same labelling divisor. This implies that $v(p_a^+) - v(p_a^-)\in \Ker(\pi)$. The proof of Lemma~\ref{lem:Kergens} applies verbatim to show that  $\pi\colon \ZZ^{Q_1}\to \ZZ(Q)$ coincides with the map 
  \[
 \wt\colon \ZZ^{Q_1}\to \Lambda_\Gamma:= \ZZ^{Q_1}/\big(v(p_a^+) - v(p_a^-)\in \ZZ^{Q_1} \mid \;a\in Q_1\big)
    \]
from Mozgovoy--Reineke~\cite[\S 3]{MozgovoyReineke}.  Then $\NN(Q)$ coincides with the semigroup $\Lambda_\Gamma^+:=\wt(\NN^{Q_1})$ that was introduced by Broomhead~\cite[Example 5.5]{Broomhead}  in defining algebraic consistency (compare Mozgovoy~\cite[Remark~3.8]{Mozgovoy}). It follows that our cone $C$ coincides with the cone dual to $\Lambda_\Gamma^+$ from \cite{Broomhead}, so the perfect matchings from Definition~\ref{def:perfectmatching} agree with those from \cite[Lemma~2.11]{Broomhead}.  Each primitive lattice generator $v_\rho\in \rho$ on an extremal ray of the cone $\sigma$ defining $R=\kk[\sigma^\vee\cap M]$ supports only one extremal perfect matching on a dimer model  (see \cite[Proposition~6.5]{IshiiUeda2}), so Proposition~\ref{prop:perfmatchlabels} implies that this perfect matching is $\Pi_\rho=\pi_2^*(\chi_\rho)$ and, moreover, that $a\in \supp(\Pi_\rho)$ if and only if $x_\rho$ divides $x^{\div(a)}$. If $m_\rho(a)$ denotes the multiplicity of $x_\rho$ in $x^{\div(a)}$, we obtain $x^{\div(a)} =  \prod_{\{\rho\in \sigma(1) \mid a\in \supp(\Pi_\rho)\}} x_\rho^{m_\rho(a)}$. Lemma~\ref{lem:multiplicity} to follow establishes that each $m_\rho(a)$ is either 0 or 1. This completes the proof.
\end{proof}

\begin{lemma}
\label{lem:multiplicity} 
For any algebraically consistent dimer model $\Gamma$ with associated quiver $Q$, and for any arrow $a\in Q_1$, the monomial $x^{\div(a)}$ divides $\prod_{\rho\in \sigma(1)} x_\rho$.
\end{lemma}
\begin{proof}
In light of \eqref{eqn:adjoint}, we need only show that $\langle \Pi_\rho, \pi(\chi_a)\rangle \in \{0,1\}$. As Broomhead~\cite[\S 2.3]{Broomhead} remarks, we may regard each perfect matching $\Pi$ in $Q$ as a 1-cochain $\pi^*(\Pi)\in (\ZZ^{Q_1})^\vee$ with values in $\{0,1\}$, where $\pi^*\colon \ZZ(Q)^\vee\to (\ZZ^{Q_1})^\vee$ is the natural inclusion. In particular, for any arrow $a\in Q_1$, the dual pairing is simply $\langle \Pi_\rho, \pi(\chi_a)\rangle = \langle \pi^*(\Pi_\rho), \chi_a\rangle \in \{0,1\}$ as required.
\end{proof}

\begin{remark}
If we knew at this stage that $A_\mathscr{E}$ was consistent then Lemma~\ref{lem:Kergens} could be applied directly in the proof of Proposition~\ref{prop:dimerlabels} and Proposition~\ref{prop:div(a)} would make Lemma~\ref{lem:multiplicity} superfluous. However, we do not establish this fact until Theorem~\ref{thm:unsignedsuperpotentials} below.
\end{remark}

We are now in a position to show that algebraically consistent dimer models define consistent toric algebras. 
 It is convenient to introduce temporarily the unsigned version of the dimer model superpotential, namely, the element $\overline{W}_\Gamma :=  \sum_{F\in Q_2} w_F$.
 
\begin{theorem}
\label{thm:unsignedsuperpotentials}
For an algebraically consistent dimer model $\Gamma$, the unsigned version of the dimer model superpotential coincides with the superpotential $W$ associated to $\mathscr{E}$, namely
\[
\overline{W}_\Gamma = \sum_{p\in \ACcyc} p.
\]
In particular, the toric algebra $A_\mathscr{E}$ is consistent.
 \end{theorem}
\begin{proof}
For any face $F$ in the cell decomposition $\Gamma$,  Proposition~\ref{prop:dimerlabels} implies that
\begin{equation}
\label{eqn:divp}
x^{\div(w_F)} = \prod_{a\in \supp(w_F)}  \prod_{\{\rho\in \sigma(1) \mid a\in \supp(\Pi_\rho)\}} x_\rho.
\end{equation}
 Since each $\Pi_\rho$ is a perfect matching and since $w_F$ is a term in $\overline{W}_\Gamma$, the set $ \supp(w_F)\cap \supp(\Pi_\rho)$ consists of precisely one arrow. This gives $x^{\div(w_F)} =  \prod_{\rho\in \sigma(1)} x_\rho$, so every term of $\overline{W}_\Gamma$ is an anticanonical cycle. Conversely, let $p$ be an anticanonical cycle in $Q$. Fix $\rho\in \sigma(1)$ and write 
\begin{equation}
\label{eqn:PhiGamma}
W_\Gamma = \sum_{a\in \supp(\Pi_\rho)} a \cdot\partial_aW_\Gamma  = \sum_{a\in \supp(\Pi_\rho)} a (p_a^+-p_a^-).
\end{equation}
Proposition~\ref{prop:perfmatchlabels} provides an arrow $a^\prime\in \supp(p)\cap \supp(\Pi_\rho)$ which gives one summand in \eqref{eqn:PhiGamma}, so after choosing a new starting point of $p$ if necessary, we may write $p=a^\prime q$ for some path $q$ in $Q$. Since each term $w_F$ in $\overline{W}_\Gamma$ satisfies $x^{\div(w_F)} =  \prod_{\rho\in \sigma(1)} x_\rho$,  this holds true for the terms $a^\prime p_{a^\prime}^+$ and $a^\prime p_{a^\prime}^-$. But now, each of $q, p_{a^\prime}^+, p_{a^\prime}^-$ is a path in $Q$ from $\head(a^\prime)$ to $\tail(a^\prime)$ with labelling monomial $(\prod_{\rho\in \sigma(1)} x_\rho)/x^{\div(a^\prime)}$. Since $Q$ is a quiver of sections and since $A_\Gamma$ is algebraically consistent, the ideal $(\partial_a W_\Gamma \mid a\in Q_1)$ must contain each of $\pm(q-p_{a^\prime}^-)$ and $\pm(p_{a^\prime}^+-q)$ in addition to $\pm(p_{a^\prime}^+-p_{a^\prime}^-)$. However, $a^\prime\in Q_1$ is the only arrow from $\tail(a^\prime)$ to $\head(a^\prime)$ with label $\div(a^\prime)$, and since this arrow appears in precisely two terms of $W_\Gamma$, it follows that $q$ must equal one of $p_a^{\pm}$. This gives $p= w_F$, so the anticanonical cycle $p$ is a term of $\overline{W}_{\Gamma}$ and hence $\overline{W}_\Gamma = \sum_{p\in \ACcyc} p$.

To show that $A_\mathscr{E}$ is consistent it suffices by Lemma~\ref{lem:dimerqos} to show that $A_\Gamma\cong A_W$. Since the terms of the superpotentials $W_\Gamma$ and $W$ coincide up to sign we have $\mathscr{P}=Q_1$, so the inclusion $J_W\subseteq J_\mathscr{E} = (\partial_a W_\Gamma \mid a\in Q_1)$ is equality as required.
\end{proof}

 \begin{remark}
 \label{rem:dimerconclusion}
For generic $\theta\in \Wt(Q)$, Theorem~\ref{thm:cohcomp} implies that the crepant resolution $\mathcal{M}_\theta\to X$ from Ishii--Ueda~\cite{IshiiUeda1} coincides with the morphism $\tau_\theta\colon Y_\theta\to Y_0$ obtained by variation of GIT quotient. This strengthens slightly an observation of Mozgovoy~\cite[Proposition~4.4]{Mozgovoy}.\end{remark}

\subsection{Reconstructing the dimer}
For an algebraically consistent dimer model $\Gamma$, let $Q$ denote the quiver of sections of the collection $\mathscr{E}$ from \eqref{eqn:dimercollection}. Before introducing the toric cell complex for $Q$, we pause to show how $\Gamma$ can be reconstructed from the covering quiver $\widetilde{Q}\subset \RR^d$.

List the standard basis of $\ZZ^d$ to be compatible with a cyclic order $v_1, \dots, v_d$ on the vertices of the polygon $P\subset N\otimes_\ZZ \RR$, and choose coordinates $N=\ZZ\langle \mathbf{e}_1,\mathbf{e}_2,\mathbf{e}_3\rangle$ so that $P$ lies in the affine plane at height one. Write $M= \ZZ\langle x,y,z\rangle$ for the dual basis and $B$ for the matrix defining $\iota\colon M\hookrightarrow \ZZ^d$ from \eqref{eqn:diagram}. Extend scalars on the map $\iota^*$ defined by $B^t$ to obtain $\iota^*_{\RR}\colon \RR^d \to N\otimes_\ZZ \RR$. For the standard inner product on $\RR^d$, orthogonal projection $f\colon \RR^d\to M\otimes_\ZZ\RR$ is defined by $(B^tB)^{-1}B^t$, so $f$ is simply the composition of $\iota^*_\RR$ with the change of basis $(B^tB)^{-1}$ from $N\otimes_\ZZ \RR$ to $M\otimes_\ZZ \RR$. For the sublattice  $M^\prime =\ZZ\langle x,y\rangle\hookrightarrow M$,  consider the (not-necessarily-orthogonal) projection $M\to M^\prime$ down the $z$-axis. After extending scalars and composing with $f$, we obtain a map $f^\prime\colon \RR^d\to \RR^2:=M^\prime\otimes_\ZZ \RR$ whose restriction to $\ZZ^d$ fits in to the diagram 
 \begin{equation}
 \label{eqn:2torusdiagram}
  \begin{CD}   
    0@>>> M  @>>> \ZZ^d    @>{\deg}>> \Cl(X)@>>> 0\\
     @.   @VVV            @VV{f^\prime\vert_{\ZZ^d}}V   @.      @.          \\
0 @>>> M^\prime @>>> \RR^2  @>>> \mathbb{T}^2 @>>> 0 
 \end{CD}
 \end{equation}
 where $\mathbb{T}^2:= \RR^2/M^\prime$ is a real two-torus.  To study the image $f^\prime(\widetilde{Q})$, lift a spanning tree from $Q$ to $\RR^d$ to obtain a preferred lift $u_i\in \deg^{-1}(E_i)$ for $i\in Q_0$ as in Definition~\ref{def:coveringquiver}. Set $\mathbf{u}^\prime_i:=f^\prime(u_i)\in \RR^2$ for each $i\in Q_0$, and let $Q^\prime_0\subset \RR^2$ denote the set of all $M^\prime$-translates of such points. Similarly, lift $a\in Q_1$ with tail at $i\in Q_0$ to the unique arrow $\widetilde{a}\in \widetilde{Q}_1$ with tail at $u_i$ and set  $\mathbf{v}^\prime_{a}:=f^\prime(\widetilde{a})$. Note that $\mathbf{v}^\prime_{a}$ is the translation of the vector $f^\prime(\div(a)) = \sum_{\{\rho\in \sigma(1) \mid a\in \supp(\Pi_\rho\}} f^\prime(\chi_\rho)$ to the point $\mathbf{u}^\prime_i$. Write $Q^\prime_1\subset \RR^2$ for the set of all $M^\prime$-translates of such vectors.  
  
\begin{definition}
Let $Q^\prime$ denote the quiver in $\mathbb{T}^2$ defined by the $M^\prime$-periodic quiver in $\RR^2$ with vertex set $Q^\prime_0$ and arrow set $Q^\prime_1$.
\end{definition}

Every vertex from $\widetilde{Q}_0$ is an $M$-translate of some $u_i$ and every arrow from $\widetilde{Q}_1$ is an $M$-translate of an arrow $\widetilde{a}$ with tail at some $u_i$,  so commutativity of \eqref{eqn:2torusdiagram} gives $Q^\prime_0 = f^\prime(\widetilde{Q}_0)$ and $Q^\prime_1 = f^\prime(\widetilde{Q}_1)$. Vertices and arrows in $Q^\prime$ may a priori overlap, so it is not obvious that $Q_0^\prime$ and $Q_1^\prime$ form the 0-skeleton and the 1-skeleton of a cell decomposition of $\mathbb{T}^2$. Nevertheless, the following result confirms that this is indeed the case:

\begin{theorem}
\label{thm:reconstructingdimer}
Every algebraically consistency dimer model $\Gamma$ is homotopy equivalent to the cell decomposition of $\mathbb{T}^2$ dual to that induced by the subquiver $Q^\prime\subset \mathbb{T}^2$.
\end{theorem}
\begin{proof}
The first step is to show that the images $f^\prime(\chi_1), \dots, f^\prime(\chi_d)$ of the standard basis vectors are cyclically ordered in $\RR^2$.  For this, let $N\to N^\prime$ denote the map dual to the inclusion $M^\prime \hookrightarrow M$. Explicitly, $N^\prime = N/\ZZ\langle \mathbf{e}_z\rangle$ where the vector $\mathbf{e}_z = \sum_{\rho\in \sigma(1)} \iota^*(\chi_\rho)$ is the  image of $z$ under the change of basis $B^tB$ from $M\otimes_\ZZ \RR$ to $N\otimes_\ZZ \RR$. Since $\mathbf{e}_z$ is the sum of the generators of the cone $\sigma$, the vector $\mathbf{e}_z$ lies in the interior of $\sigma$ and hence the cyclic order of the vertices of the slice $P\subset \sigma$ is maintained under the projection to $N^\prime\otimes_\ZZ \RR$. It remains to note that the vectors $f^\prime(\chi_1), \dots, f^\prime(\chi_d)$ are obtained from these cyclically ordered vectors in $N^\prime\otimes_\ZZ \RR$ by the change of basis from $N^\prime\otimes_\ZZ \RR$ to $M^\prime\otimes_\ZZ \RR$. 

We now associate a convex polygon in $M^\prime\otimes_\ZZ \RR$ to every face $F\in Q_2$ in the cell decomposition of $\mathbb{T}^2$ dual to $\Gamma$. Theorem~\ref{thm:unsignedsuperpotentials} implies that each $F\in Q_2$ determines an anticanonical cycle $p_F$ in $Q$. The lift of $p_F$ is an anticanonical path in $\widetilde{Q}$ whose image under $f^\prime$ is a closed piecewise-linear curve in $M^\prime\otimes_\ZZ \RR$ that traverses arrows in $Q^\prime_1$ arising from the arrows $a\in \supp(p_F)$. To see that this curve is the boundary of a convex polygon, consider a pair of complete fans in $M^\prime\otimes_\ZZ \RR$ introduced by Broomhead~\cite[\S4.4-4.5]{Broomhead},  namely, the global zig-zag fan $\Xi$ and the local zig-zag fan $\xi(F)$. The two-dimensional cones $\sigma_\rho$ in $\Xi$ are indexed by extremal perfect matchings $\Pi_\rho$,  the two-dimensional cones $\sigma_a$ in $\xi(F)$ are indexed by arrows $a\in \supp(p_F)$, and $\sigma_\rho\subseteq \sigma_a$ if and only if $a\in \supp(\Pi_\rho)$. In particular, equation \eqref{eqn:label} can be written as
\begin{equation}
\label{eqn:zigzaglabels}
x^{\div(a)} = \prod_{\{\rho\in \sigma(1) \mid \sigma_a\supseteq \sigma_\rho\}} x_\rho.
\end{equation}
Since $\Xi$ is the common refinement of the fans $\xi(F)$ for each $F\in Q_2$, \cite[Remark 4.16]{Broomhead} implies that the cyclic order of the cones $\sigma_\rho$ in $\Xi$ is the same as the cyclic order of the vertices of the polygon $P$. We deduce from \eqref{eqn:zigzaglabels} that labels on a given arrow are consecutive and, furthermore, that the boundary of each face $F$ consists of arrows with consecutive labels; these increase around black faces and decrease around white.  Every arrow in $Q^\prime$ is an $M^\prime$-translate of a vector $\mathbf{v}_a^\prime = \sum_{\{\rho\in \sigma(1) \mid a\in \supp(\Pi_\rho)\}} f^\prime(\chi_\rho)$, and the first step above establishes that the vectors $f^\prime(\chi_1), \dots, f^\prime(\chi_d)$ are cyclically ordered in $\RR^2$, so the set of edges $\{\mathbf{v}^\prime_a \mid a\in\supp(p_F)\}$ of the closed piecewise-linear curve is cyclically ordered,  clockwise for black faces and anticlockwise for white. It follows that this curve is the boundary of a convex polygon $\Tile(F)$ in $M^\prime\otimes_\ZZ \RR$.

It suffices to see that these convex polygons are the 2-cells in a decomposition of $\mathbb{T}^2$ that is homotopy equivalent to that induced by $Q$.  For this, fix $i\in Q_0$ and list cyclically all arrows $a_1,b_1,a_2, b_2, \dots a_k, b_k\in Q_1$ with $\tail(a_\nu)=i$ and $\head(b_\nu)=i$ for $1\leq \nu\leq k$. Let $F^-_\nu\in Q_2$ denote the unique black face containing $a_\nu, b_\nu$ in its boundary, and similarly, $F^+_\nu\in Q_2$ the white face containing $b_\nu, a_{\nu+1}$, with $a_{k+1}:=a_1$. Since the boundaries of $\Tile(F^-_\nu)$ and $\Tile(F^+_\nu)$ are oriented clockwise and anticlockwise respectively, the polygons $\Tile(F^-_1), \Tile(F^+_1), \dots, \Tile(F^-_k), \Tile(F^+_k)$ glue cyclically around vertex $i$. Note that these polygons do not cycle more than once around $i$, because otherwise the edges dual to the outward-pointing arrows $a_1,\dots, a_k$ cycle more than once around the face dual to $i\in Q_0$ which is absurd. In this way, each vertex $i\in Q_0$ gives rise to a \emph{tile} obtained as union of convex polygons $\tile(i):=\bigcup_{1\leq\nu\leq k} (\Tile(F^-_\nu)\cup \Tile(F^+_\nu))$. For each incoming arrow $b_\nu$, the tile $\tile(i)$ glues to $\tile(\tail(b_\nu))$ along $\Tile(F_\nu^-)\cup\Tile(F^+_\nu)$ and, similarly,  for each outgoing arrow $a_\nu$, the tile $\tile(i)$ glues to $\tile(\head(a_\nu))$ along $\Tile(F_\nu^-)\cup\Tile(F^+_{\nu-1})$. It follows that the convex polygons $\Tile(F)$ arising from faces $F\in Q_2$ tesselate the plane and, moreover, the 0-skeleton and 1-skeleton coincide with $Q_0^\prime$ and $Q_1^\prime$ respectively. The assignment $F\mapsto \Tile(F)$ shows that this cell decomposition coincides with that arising from $Q$ up to homotopy.
    \end{proof}
   
 \begin{example}
 \label{exa:F1tiltingSubdivisionT3}
 For the dimer model $\Gamma$ and quiver $Q$ from Example~\ref{exa:dimerF1}, the quivers $\widetilde{Q}(i)$ for $i=0,1$ are drawn in black  in Figure~\ref{fig:F1tiltingSubdivisionT3}, each superimposed on a 4-cube $\mathsf{C}(u_i)$ drawn in grey. 
     \begin{figure}[!ht]
    \centering
      \subfigure[]{
       \psset{unit=0.7cm}
   \begin{pspicture}(0.6,-0.3)(5.5,5.2)
      \cnodeput(3,5.4){E1}{\tiny{0}}\cnodeput(0.6,4.2){D1}{\tiny{3}}
       \cnodeput(0,2.7){C1}{\tiny{2}}\cnodeput(2.4,2.7){C4}{\tiny{2}}
       \cnodeput(3.6,2.7){C3}{\tiny{3}}\cnodeput(6,2.7){C6}{\tiny{3}}
       \cnodeput(0.6,1.2){B1}{\tiny{1}}\cnodeput(3.8,1.2){B3}{\tiny{1}}
       \cnodeput(5.4,1.2){B4}{\tiny{2}}\cnodeput(3,0){A1}{\tiny{0}}
      \psset{linecolor=lightgray}
       \cnodeput(2.2,4.2){D2}{\tiny{}}\cnodeput(3.8,4.2){D3}{\tiny{}}
       \cnodeput(5.4,4.2){D4}{\tiny{}}\cnodeput(1.2,2.7){C2}{\tiny{}}
       \cnodeput(4.8,2.7){C5}{\tiny{}}\cnodeput(2.2,1.2){B2}{\tiny{}}
       \ncline{->}{A1}{B2}\ncline{->}{B1}{C2}\ncline{->}{B2}{C1}\ncline{->}{B2}{C4}
       \ncline{->}{B2}{C5}\ncline{->}{B3}{C2}\ncline{->}{B4}{C5}\ncline{->}{C1}{D2}
       \ncline{->}{C2}{D1}\ncline{->}{C2}{D3}\ncline{->}{C3}{D2}\ncline{->}{C3}{D3}
       \ncline{->}{C4}{D4}\ncline{->}{C5}{D2}\ncline{->}{C5}{D4}\ncline{->}{C6}{D3}
       \ncline{->}{C6}{D4}\ncline{->}{D2}{E1}\ncline{->}{D3}{E1}\ncline{->}{D4}{E1}
       \psset{linecolor=black,linewidth=0.04}
       \ncline{->}{A1}{B1}\lput*{:180}{\tiny{$x_1$}}
       \ncline{->}{A1}{B3}\lput*{:U}{\tiny{$x_3$}}
       \ncline{->}{A1}{B4}\lput*{:U}{\tiny{$x_4$}}
       \ncline{->}{B1}{C1}\lput*{:U}{\tiny{$x_2$}}
       \ncline{->}{B3}{C4}\lput*{:180}{\tiny{$x_2$}}
       \ncline{->}{B3}{C6}\lput*{:U}(0.6){\tiny{$x_4$}}
       \ncline{->}{B4}{C6}\lput*{:U}{\tiny{$x_3$}}
       \ncline{->}{C1}{D1}\lput*{:U}{\tiny{$x_3$}}
       \ncline{->}{C4}{D1}\lput*{:180}{\tiny{$x_1$}}
       \ncline{->}{C6}{E1}\lput*{:180}(0.25){\tiny{$x_1x_2$}}
       \ncline{->}{D1}{E1}\lput*{:U}{\tiny{$x_4$}}
        \ncline{->}{C3}{E1}\lput*{:180}(0.25){\tiny{$x_2x_3$}}
       \ncline{->}{B1}{C3}\lput*{:U}{\tiny{$x_4$}}
       \ncline{->}{B4}{C3}\lput*{:180}(0.6){\tiny{$x_1$}}
   \end{pspicture}      }
      \quad \qquad  \qquad \qquad
      \subfigure[]{    
       \psset{unit=0.7cm}
       \begin{pspicture}(0.6,-0.3)(5.5,5.2)
       \cnodeput(3,5.4){E1}{\tiny{1}}
       \cnodeput(2.2,4.2){D2}{\tiny{0}}
       \cnodeput(5.4,4.2){D4}{\tiny{0}}
       \cnodeput(0,2.7){C1}{\tiny{3}}
       \cnodeput(2.4,2.7){C3}{\tiny{3}}
       \cnodeput(5.4,1.2){B4}{\tiny{3}}
       \cnodeput(2.2,1.2){B2}{\tiny{2}}
       \cnodeput(3,0){A1}{\tiny{1}}
      \psset{linecolor=lightgray}
       \cnodeput(0.6,4.2){D1}{\tiny{}}
       \cnodeput(3.8,4.2){D3}{\tiny{}}
       \cnodeput(1.2,2.7){C2}{\tiny{}}
       \cnodeput(3.6,2.7){C4}{\tiny{}}
       \cnodeput(4.8,2.7){C5}{\tiny{}}
       \cnodeput(6,2.7){C6}{\tiny{}}
       \cnodeput(0.6,1.2){B1}{\tiny{}}
       \cnodeput(3.8,1.2){B3}{\tiny{}}    
       \ncline{->}{A1}{B1}\ncline{->}{A1}{B3}\ncline{->}{B1}{C1}\ncline{->}{B1}{C2}
       \ncline{->}{B1}{C4}\ncline{->}{B2}{C1}\ncline{->}{B2}{C5}\ncline{->}{B3}{C2}
       \ncline{->}{B3}{C3}\ncline{->}{B3}{C5}\ncline{->}{B4}{C5}\ncline{->}{B4}{C3}
       \ncline{->}{B4}{C6}\ncline{->}{C1}{D1}\ncline{->}{C2}{D1}\ncline{->}{C2}{D3}
       \ncline{->}{C3}{D1}\ncline{->}{C4}{D2}\ncline{->}{C4}{D3}\ncline{->}{C5}{D2}
       \ncline{->}{C5}{D4}\ncline{->}{C6}{D3}\ncline{->}{C6}{D4}\ncline{->}{D1}{E1}
       \ncline{->}{D2}{E1}\ncline{->}{D3}{E1}
        \psset{linecolor=black,linewidth=0.04}
       \ncline{->}{A1}{B2}\lput*{:180}{\tiny{$x_2$}}
       \ncline{->}{A1}{B4}\lput*{:U}{\tiny{$x_4$}}
       \ncline{->}{B2}{C1}\lput*{:180}{\tiny{$x_1$}}
       \ncline{->}{B2}{C3}\lput*{:U}{\tiny{$x_3$}} 
       \ncline{->}{C1}{D2}\lput*{:U}{\tiny{$x_4$}}
       \ncline{->}{C3}{D4}\lput*{:U}{\tiny{$x_4$}}
       \ncline{->}{B4}{D4}\lput*{:U}{\tiny{$x_2x_3$}} 
       \ncline{->}{D2}{E1}\lput*{:U}{\tiny{$x_3$}}
       \ncline{->}{D4}{E1}\lput*{:180}{\tiny{$x_1$}}
        \ncline{->}{B4}{D2}\lput*{:180}{\tiny{$x_1x_2$}}
      \end{pspicture}
          }
  \caption{Unit 4-cubes in $\RR^4$ containing the quivers: (a) $\widetilde{Q}(0)$; and (b) $\widetilde{Q}(1)$.}
  \label{fig:F1tiltingSubdivisionT3} 
  \end{figure}  
To construct the quiver $Q^\prime$ as in Theorem~\ref{thm:reconstructingdimer}, note that $\sigma$ is the cone from Example~\ref{exa:F1tilting}, so the map $\iota\colon M\to \ZZ^4$ is defined by the matrix $B$ with columns $\div(x)=(1,0,-1,0)$, $\div(y)=(0,1,1,-1)$ and $\div(z)=(1,1,1,1)$. Orthogonal projection $f^\prime\colon \RR^4\to M^\prime\otimes_\ZZ \RR$ with respect to the standard inner product on $\RR^4$ is defined by the first two rows of $(B^tB)^{-1}B^t$, namely
  \[
 \begin{pmatrix}\frac{5}{9} & \frac{1}{6} & -\frac{4}{9} & -\frac{5}{18} \\ \frac{1}{9} & \frac{1}{3} & \frac{1}{9} & -\frac{5}{9} \end{pmatrix}.
 \]
 A black face $F$ determines the anticanonical path from Figure~\ref{fig:F1tiltingSubdivisionT3}(a) that traverses arrows labelled $x_1, x_2, x_3, x_4$, and this projects via $f^\prime$ to define the anticlockwise boundary of the convex polygon with vertices $\mathbf{u}^\prime_0=(0,0)$, $\mathbf{u}_1^\prime=(\frac{5}{9},\frac{1}{9})$, $\mathbf{u}^\prime_2=(\frac{13}{18},\frac{4}{9})$ and $\mathbf{u}^\prime_3=(\frac{5}{18},\frac{5}{9})$ in the quiver $Q^\prime$ from Figure~\ref{fig:F1SubdivisionT2}(a). A white face determines the adjacent anticanonical path that traverses arrows labelled $x_3, x_2, x_1, x_4$, and this projects to the clockwise boundary of an adjacent polygon in Figure~\ref{fig:F1SubdivisionT2}(a). Each of the six anticanonical paths from Figure~\ref{fig:F1tiltingSubdivisionT3}(a) defines one of the six convex polygons in  Figure~\ref{fig:F1SubdivisionT2}(a) with $\mathbf{u}^\prime_0$ as a vertex, and these polygons define $\tile(0)$. The paths in $\widetilde{Q}(1)$ from Figure~\ref{fig:F1tiltingSubdivisionT3}(b) define the convex polygons that make up $\tile(1)$, and vertices $i=2,3$ are similar. The resulting cell decomposition of $\mathbb{T}^2$ is homotopy equivalent to that from Figure~\ref{fig:dimerF1}(b). 
\end{example}

    \begin{figure}[!ht]
    \centering
  \subfigure[]{
    \psset{unit=1.2cm}
   \begin{pspicture}(0,0)(3,2.9)
 \psset{linecolor=gray}
    \pnode(0,0){A}  \pnode(3,0){B} \pnode(0,3){C}  \pnode(3,3){D} %vertices of square
   \ncline{-}{A}{B}  \ncline{-}{A}{C}\ncline{-}{C}{D}  \ncline{-}{D}{B} %edges of square
   \pnode(1.5,0){K} \pnode(0,1.5){L} \pnode(1.5,3){M} \pnode(3,1.5){N} %midpoints on square
  \psset{linecolor=black}
  \cnodeput(0,0){Q}{0}\cnodeput(1.667,0.333){O}{1}\cnodeput(2.1667,1.333){P}{2}\cnodeput(0.833,1.667){R}{3}
  \cnodeput(3,0){Q1}{0}\cnodeput(0,3){Q2}{0}\cnodeput(3,3){Q3}{0}%vertices of quiver
  \cnodeput(-1.3333,0.333){N1}{1}\cnodeput(-0.83333,1.333){N2}{2}
   \pnode(3,1.5){S} \pnode(0,1.5){T} \pnode(1.5,0){U} \pnode(1.5,3){V}%quiver intersects square
    \pnode(3,2.2){W} \pnode(0,2.2){X}\pnode(3,0.8){Y} \pnode(0,0.8){Z}\pnode(0.8,3){ZA} \pnode(0.8,0){ZB} 
  \ncline{->}{Q}{O}\lput*{:350}{\small{$x_1$}}
  \ncline{->}{Q1}{O}\lput*{:195}{\small{$x_3$}}
  \ncline{->}{Q3}{P}\lput*{:115}{\small{$x_4$}}
  \ncline{->}{O}{P}\lput*{:295}{\small{$x_2$}}
  \ncline{->}{P}{R}\lput*{:195}{\small{$x_3$}}
  \ncline{-}{P}{S}\lput*{:350}{\small{$x_1$}}\ncline{->}{T}{R}
  \ncline{-}{O}{U}\ncline{->}{V}{R}\lput*{:115}{\small{$x_4$}}
  \ncline{->}{R}{Q}\lput*{:120}{\small{$x_4$}}
  \ncline{->}{R}{Q3}\lput*{:330}{\small{$x_1x_2$}}
  \ncline{->}{R}{Q2}\lput*{:240}{\small{$x_2x_3$}}
   \ncline{->}{Q}{N1}\lput*{:195}{\small{$x_3$}}
    \ncline{->}{N1}{N2}\lput*{:295}{\small{$x_2$}}
\ncline{-}{N2}{T}\lput*{:350}{\small{$x_1$}}
  \end{pspicture}
      }
      \qquad  \qquad
      \subfigure[]{    \psset{unit=1.2cm}
   \begin{pspicture}(0,0)(3,3)
 \psset{linecolor=gray}
    \pnode(0,0){A}  \pnode(3,0){B} \pnode(0,3){C}  \pnode(3,3){D} %vertices of square
   \ncline{-}{A}{B}  \ncline{-}{A}{C}\ncline{-}{C}{D}  \ncline{-}{D}{B} %edges of square
   \pnode(1.5,0){K} \pnode(0,1.5){L} \pnode(1.5,3){M} \pnode(3,1.5){N} %midpoints on square
  \psset{linecolor=black}
  \cnodeput(0,0){Q}{0}\cnodeput(1.667,0.333){O}{1}\cnodeput(2.1667,1.333){P}{2}\cnodeput(0.833,1.667){R}{3}\cnodeput(1.333,2.667){E}{4}
  \cnodeput(3,0){Q1}{0}\cnodeput(0,3){Q2}{0}\cnodeput(3,3){Q3}{0}%vertices of quiver
   \pnode(3,1.5){S} \pnode(0,1.5){T} \pnode(1.8,0){U} \pnode(1.8,3){V}\pnode(1.2,3){F}\pnode(1.2,0){G}%quiver intersects square
    \pnode(3,2.2){W} \pnode(0,2.2){X}\pnode(3,0.8){Y} \pnode(0,0.8){Z}\pnode(0.8,3){ZA} \pnode(0.8,0){ZB} 
  \ncline{->}{Q}{O}\lput*{:350}{\small{$x_1$}}
  \ncline{->}{Q1}{O}\lput*{:195}{\small{$x_3$}}
  \ncline{->}{Q3}{P}\lput*{:115}{\small{$x_4$}}
  \ncline{->}{O}{P}\lput*{:295}{\small{$x_2$}}
  \ncline{->}{P}{R}\lput*{:195}{\small{$x_3$}}
  \ncline{-}{P}{S}\lput*{:350}{\small{$x_1$}}\ncline{->}{T}{R}
  \ncline{-}{O}{U} \nccurve[angleA=270,angleB=30]{->}{V}{R}\lput*{:120}{\small{$x_4$}}
  \ncline{->}{F}{E} \nccurve[angleA=205,angleB=90]{-}{P}{G}\lput*{:120}{\small{$x_4$}}
  \ncline{->}{R}{E}\lput*{:295}{\small{$x_2$}}
  \ncline{->}{R}{Q}\lput*{:120}{\small{$x_4$}}
  \ncline{->}{E}{Q3}\lput*{:350}{\small{$x_1$}}
  \ncline{->}{E}{Q2}\lput*{:190}{\small{$x_3$}}
  \end{pspicture} 
      }
  \caption{(a) Tesselation induced by $Q^\prime_1$ from \ref{exa:F1tiltingSubdivisionT3}; (b) the quiver $f^\prime(\widetilde{Q})$ for \ref{rem:notcelldivision}}
  \label{fig:F1SubdivisionT2}
  \end{figure}

  \begin{remark}
\label{rem:notcelldivision} 
 It is not sufficient for Theorem~\ref{thm:reconstructingdimer} that $Q$ arises from a consistent toric algebra in dimension three. For example, for the collection $\mathscr{E}^{\prime\prime}$ from Example~\ref{exa:superpotentialF1tilting}\three, the vertex labelled 4 in Figure~\ref{fig:F1subandsupertilting}(b) determines $\mathbf{u}_4^\prime=(\frac{4}{9},\frac{8}{9})$ in Figure~\ref{fig:F1SubdivisionT2}(b), and a pair of arrows (drawn as curves) labelled $x_4$ from vertices 1 and 2 cross other arrows at points that are not vertices.  In this case, the toric algebra is consistent but it does not arise from an algebraically consistent dimer model ($A_{\mathscr{E}^{\prime\prime}}$ is not Calabi--Yau) and we do not obtain a subdivision of $\mathbb{T}^2$.  
 \end{remark}
  
 \subsection{The toric cell complex}
Let $f\colon \RR^d\to \RR^3:=M\otimes_\ZZ \RR$ denote the orthogonal projection onto the subspace spanned by $M\subseteq \ZZ^d$ as in the previous subsection.   Then by construction the diagram \eqref{eqn:2torusdiagram} factors through the commutative diagram
 \begin{equation}
 \label{eqn:3torusdiagram}
 \begin{CD}
   0@>>> M  @>>> \ZZ^d    @>{\deg}>> \Cl(X)@>>> 0\\
    @.   @|            @VV{f\vert_{\ZZ^d}}V   @.      @.          \\
0 @>>> M @>>> \RR^3  @>>> \mathbb{T}^3 @>>> 0
 \end{CD}
 \end{equation}
 where $\mathbb{T}^3:= \RR^3/M$ is a real three-torus.  For each $i\in Q_0$, our chosen vertex $u_i\in \widetilde{Q}_0$ determines a point $\mathbf{u}_i:=f(u_i)\in \RR^3$. Set $\mathbf{v}_z := f(\div(z))$, and consider the family of affine planes
 \[
H_\lambda(i):= \left\{\mathbf{u}+\mathbf{u}_i+ \lambda\mathbf{v}_z\in \RR^3=M\otimes_\ZZ \RR \mid \mathbf{u}\in M^\prime\otimes_\ZZ \RR\right\}\quad \text{for }0\leq \lambda\leq 1
 \]
Note that $H_\lambda(i)$ is the translation by $\lambda\mathbf{v}_z$ of the affine plane through $\mathbf{u}_i$ parallel to $M^\prime\otimes_\ZZ \RR$.  These planes slice the image $f(\mathsf{C}(u_i))$ of the unit box  and hence the image $f(\widetilde{Q}(i))$ of the quiver.

\begin{proposition}
\label{prop:3cellpolygons}
For $i\in Q_0$, let $n_i$ be the number of arrows $a\in Q_1$ with $\tail(a)=i$. Either:
\begin{enumerate}
\item[\one] $n_i > 2$, when $f(\widetilde{Q}(i))\cap H_\lambda(i)$ is the vertex set of a polygon $P_\lambda(i)$ for all $0< \lambda < 1$, and it is a singleton $P_\lambda(i)$ for $\lambda =0,1$; or
\item[\two] $n_i=2$, when there exists $\lambda_{\tail}(i)< \lambda_{\head}(i)$ such that $f(\widetilde{Q}(i))\cap H_\lambda(i)$ is the vertex set of
\begin{enumerate}
\item a polygon $P_\lambda(i)$ in $H_\lambda(i)$ for $\lambda_{\tail}(i)< \lambda < \lambda_{\head}(i);$
\item a line segment $P_{\lambda}(i)$ in $H_\lambda(i)$ for $0<\lambda\leq \lambda_{\tail}(i)$ and $\lambda_{\head}(i)\leq \lambda< 1;$
\item a singleton $P_\lambda(i)$ for $\lambda =0,1$.
\end{enumerate}
\end{enumerate}
\end{proposition}
\begin{proof}
For $0\leq \lambda\leq 1$, the image under $f$ of any anticanonical path in $\widetilde{Q}(i)$ touches $H_\lambda(i)$ once because the projection of any such path onto the real line spanned by $\mathbf{v}_z$ is bijective onto its image. We proceed by reconstructing the information of the slice $f(\widetilde{Q}(i))\cap H_\lambda(i)$ in $\RR^3$ using our knowledge of the projection to $\RR^2$ from Theorem~\ref{thm:reconstructingdimer}.

Fix $a\in Q_1$ with $\tail(a)=i$ and write $q_a^\pm$ for the anticanonical paths from $u_i$ covering $p_a^\pm a$ that begin by traversing the unique lift $\widetilde{a}\in \widetilde{Q}_1(i)$ of $a$. Let $\gamma_a^\pm\colon [0,1]\to \RR^3$ denote the curves with images $f(q_a^\pm)$ in $f(\widetilde{Q}(i))$ such that $\gamma_a^\pm(\lambda)\in H_\lambda(i)$ for $0\leq \lambda \leq 1$. Composing with the projection to $\RR^2$ gives piecewise-linear curves $\overline{\gamma}_a^\pm\colon [0,1]\to \RR^2$ satisfying $\overline{\gamma}_a^\pm(0) = \overline{\gamma}_a^\pm(1) = \mathbf{u}^\prime_i$ that traverse the paths $Q^\prime$ corresponding to $p_a^\pm a$. In the notation of the previous proof, list the arrows $a_1,\dots, a_k\in Q_1$ with tail at vertex $i$. Then as $0\leq \lambda \leq 1$ increases, the set of points $\Omega_\lambda(i):=\{\overline{\gamma}_{a_1}^+(\lambda),\overline{\gamma}_{a_1}^-(\lambda) \dots, \overline{\gamma}_{a_k}^+(\lambda), \overline{\gamma}_{a_1}^-(\lambda)\}$ in $\RR^2$ flow from $\mathbf{u}^\prime_i$ out along $\mathbf{v}^\prime_{a_1}, \dots, \mathbf{v}^\prime_{a_k}$ before splitting and returning to $\mathbf{u}_i^\prime$ along the paths in $Q^\prime$ corresponding to $p_{a_1}^+, p_{a_1}^-, \dots, p_{a_k}^+, p_{a_k}^-$.  If $n_i>2$ then $\Omega_\lambda(i)$ comprises three or more points in cyclic order around $\mathbf{u}_i^\prime\in \RR^2$  for $0<\lambda<1$, and hence forms the vertex set of a (nondegenerate) polygon.  For $n_i=2$, let $\mathbf{v}^\prime_{a_1}, \mathbf{v}^\prime_{a_2}\in Q^\prime_1$ denote the arrows with tail at $\mathbf{u}_i^\prime$ and $\mathbf{v}^\prime_{b_1}, \mathbf{v}^\prime_{b_2}\in Q^\prime_1$ the arrows with head at $\mathbf{u}_i^\prime$. Set $\lambda_{\tail}(i):=\frac{1}{d} \min\{\deg(x^{\div(a_1)}),\deg(x^{\div(a_2)})\}$ and $\lambda_{\head}(i):=1-\frac{1}{d} \min\{\deg(x^{\div(b_1)}),\deg(x^{\div(b_2)})\}$. The set $\Omega_\lambda(i)$ comprises two points for $\lambda\leq \lambda_{\tail}(i)$ and $\lambda\geq \lambda_{\head}(i)$,  but otherwise forms the vertex set of a polygon as above. The result follows since translation by $\mathbf{u}_i+ \lambda\mathbf{v}_z$ canonically identifies  $\Omega_\lambda(i)\subset M^\prime \otimes_\ZZ \RR$ with the slice $f(\widetilde{Q})\cap H_\lambda(i)$ for each $0\leq \lambda\leq 1$.
 \end{proof}

 \begin{corollary}
 \label{lem:2-celllinesegments}
 For $a\in Q_1$, set $i:=\tail(a)\in Q_0$ and $\lambda_a:= \frac{1}{d}\deg(x^{\div(a)})\in \QQ$. Let $q_a^\pm$ be the anticanonical paths in $\widetilde{Q}(i)$ that cover the cycles $ap_a^\pm$ in $Q$. The slice $f(q_a^\pm)\cap H_\lambda(i)$ is the vertex set of a line segment $\ell_{\lambda}(a)$ in $H_\lambda(i)$ for $\lambda_a<\lambda < 1$, and it is a singleton $\ell_\lambda(a)$ for $\lambda=\lambda_a, 1$.
   \end{corollary}
   \begin{proof}
One need only focus attention on a single arrow in the course of the proof above.
    \end{proof}
   
 Proposition~\ref{prop:3cellpolygons} and its corollary are introduced to facilitate the following key definition. 

\begin{definition}
\label{def:cellcomplex}
 We associate to the collection $\mathscr{E}$ on $X$ a set $\Delta$ of closed subsets of $\mathbb{T}^3$. We describe the construction in each dimension separately (compare Remark~\ref{rem:CWcomplex}): 
 \begin{enumerate}
 \item[$\Delta_0$:] Each $i\in Q_0$ defines a point $\mathbf{u}_i=f(u_i)\in \RR^3$, and we let $\Delta_0$ denote the set of $M$-translates of all such points. 
  \item[$\Delta_1$:] Each $a\in Q_1$ with $\tail(a)=i$ lifts uniquely to an arrow $\widetilde{a}\in \widetilde{Q}_1$ with tail at $u_i$, and each $a\in Q_1$ with $\head(a)=i$ lifts uniquely to an arrow $\widetilde{a}\in \widetilde{Q}_1$ with head at $u_i$. Let $\Delta_1$ denote the set of  $M$-translates of the supports $\supp(\widetilde{a})\subset \RR^3$ of all such arrows.
    \item[$\Delta_2$:] For $a\in Q_1$, set $i:=\tail(a)\in Q_0$ and $\lambda_a:= \frac{1}{d}\deg(x^{\div(a)})\in \QQ$. Define $\eta_a:= \bigcup_{\lambda_a\leq \lambda\leq 1}\ell_{\lambda}(a)$ to be the union of the sets from Lemma~\ref{lem:2-celllinesegments}. Let $\Delta_2$ denote the set of $M$-translates of all such closed subsets $\eta_a$ defined by arrows in $Q$.
   \item[$\Delta_3$:] For $i\in Q_0$, define $\eta_i:= \bigcup_{0\leq \lambda\leq 1} P_{\lambda}(i)$ to be the union of all subsets introduced in Proposition~\ref{prop:3cellpolygons}. Let $\Delta_3$ denote the set of $M$-translates of such closed subsets $\eta_i$ corresponding to all vertices in $Q$.
  \end{enumerate}
Let $\Delta$ denote the set of closed subsets in $\mathbb{T}^3$ determined by these $M$-periodic subsets in $\RR^3$. As before, for $\eta\in \Delta_k$ we write `$\codim(\eta^\prime, \eta)=1$' as shorthand for those $\eta^\prime\in \Delta_{k-1}$ satisfying $\eta^\prime\subset \eta$.
\end{definition}

\begin{example}
\label{exa:F1tiltingGamma}
Returning to Example~\ref{exa:F1tiltingSubdivisionT3}, the elements of $\Delta$ determined by arrows from $\widetilde{Q}(0)$ as shown in Figure~\ref{fig:F1tiltingSubdivisionT3}(a) are the $M$-translates of the faces of the three-dimensional convex polytope obtained as the convex hull of the vertex set $f(\widetilde{Q}(0))$. We work with the projection, so labels on arrows in this figure should now be $f(\div(a))\in \RR^3$ rather than $\div(a)\in \ZZ^4$. After taking into account $M$-periodicity, we see that Figure~\ref{fig:F1tiltingSubdivisionT3}(a) contributes 1 subset to $\Delta_3$, 6 subsets to $\Delta_2$, all 10 subsets to $\Delta_1$ and all 4 subsets to $\Delta_0$. Similarly, Figure~\ref{fig:F1tiltingSubdivisionT3}(b) illustrates $f(\widetilde{Q}(1))$ in $\RR^3$ which contributes 1 subset to $\Delta_3$, 4 to $\Delta_2$, 9 to $\Delta_1$ and all 4 to $\Delta_0$ in $\mathbb{T}^3$. One computes similarly the subsets determined by arrows from $\widetilde{Q}(2)$ and $\widetilde{Q}(3)$ to see that $\Delta$ satisfies $\vert \Delta_0\vert = \vert \Delta_3\vert = 4$ and $\vert \Delta_1\vert = \vert \Delta_2\vert = 10$.
 \end{example}

\begin{remark}
\label{rem:CWcomplex}
Example~\ref{exa:F1tiltingGamma} shows that elements of $\Delta_3$ need not be equidimensional, and subsets from $\Delta_2$ may intersect along their interiors, so $\Delta$ is not a regular cell complex as in Section~\ref{sec:CWcomplex}. Nevertheless, the closed subsets in $\Delta$ arise from a mild variant of the topological notion of the attaching map of a cell. Indeed, for each subset $\eta\in \Delta_k$ there is a (not necessarily surjective) continuous map $\varphi_\eta \colon B^k\to \eta\subset \mathbb{T}^3$  from the closed $k$-ball such that the boundary satisfies 
\[
\varphi_\eta(S^{k-1}) = \bigcup_{\codim(\eta^\prime, \eta)=1} \varphi_\eta(B^k)\cap\eta^{\prime}.
\]
Thus, while our closed subsets in $\mathbb{T}^3$ are not actually cells as defined in Section~\ref{sec:CWcomplex}, they are very similar.  In addition, Lemma~\ref{lem:dimerincidence} below shows that $\Delta$ shares key properties with regular cell complexes. We therefore choose to adopt the terminology `cells' and `cellular' from now on.
\end{remark}

\begin{definition}
The \emph{toric cell complex} of the algebraically consistent dimer model $\Gamma$ or, equivalently, of the collection $\mathscr{E}$ from \eqref{eqn:dimercollection}, is the set $\Delta$ comprising closed subsets of $\mathbb{T}^3$ as in Definition~\ref{def:cellcomplex}. We refer to each element $\eta\in \Delta$ as a \emph{cell} in $\Delta$ (see Remark~\ref{rem:CWcomplex}).
\end{definition}

% We provide a topological interpretation of cells from $\Delta$ in Remark~\ref{rem:CWcomplex} below.

 \begin{lemma}
 \label{lem:dimerbijection}
There are canonical bijections between $\Delta_0$ and $Q_0$, between $\Delta_1$ and $Q_1$, and between $\Delta_2$ and the set $\{p_a^+ - p_a^- \in \kk Q\mid a\in Q_1\}$ of minimal generators of the ideal $J_\mathscr{E}$. 
\end{lemma}
\begin{proof}
This is similar to the proof of Lemma~\ref{lem:McKaybijections}.
\end{proof}

For $k\leq 2$ and $\eta\in \Delta_k$, the \emph{head}, \emph{tail} and \emph{label} of $\eta$ are defined to be the head, tail and label respectively of the element of $\kk Q$ that is associated to $\eta$ by Lemma~\ref{lem:dimerbijection}. Each $\eta\in \Delta_3$ is constructed from some quiver $\widetilde{Q}(i)$, and we define both the \emph{head} and \emph{tail} of $\eta$ to be the vertex $i\in Q_0$, while the \emph{label} of $\eta$ is $(1,\dots,1)\in \ZZ^n$. The notions of right- and left-differentiation of cells with respect to faces are defined precisely as in Section~\ref{sec:McKay}. Indeed, for $\eta\in \Delta$ and for any face $\eta^\prime\subset \eta$ there is a (not necessarily unique) path in $\widetilde{Q}$ from $\head(\eta^\prime)$ to $\head(\eta)$ that defines (uniquely) an element $\overleftarrow{\partial}_{\!\eta'}\eta\in A$. Similarly, any path in $\widetilde{Q}$ from $\tail(\eta)$ to $\tail(\eta^\prime)$ defines uniquely an element $\overrightarrow{\partial}_{\!\eta'}\eta\in A$.

\begin{definition}
\label{def:leftrightderivativesdimer}
For $\eta\in \Delta$ and any face $\eta^\prime\subset \eta$, the element $\overleftarrow{\partial}_{\!\eta'}\eta\in A$ is the \emph{left-derivative} of $\eta$ with respect to $\eta'$. Similarly, $\overrightarrow{\partial}_{\!\eta'}\eta\in A$ is the \emph{right-derivative} of $\eta$ with respect to $\eta'$.
\end{definition}

Again, the cells of $\Delta$ satisfy a duality property (compare Proposition~\ref{prop:dualityMcKay}):

\begin{proposition}
The map $\tau\colon \Delta\to \Delta$ that assigns to each $\eta\in \Delta_k$ the unique cell $\eta^\prime\in \Delta_{3-k}$ with $\tail(\eta^\prime)=\head(\eta)$, $\head(\eta^\prime)=\tail(\eta)$ and $x^{\div(\eta^\prime)} = \prod_{\rho\in \sigma(1)}x_\rho/x^{\div(\eta)}$ is an involution.
\end{proposition}
\begin{proof}
This is evident from the construction.
\end{proof}

  \subsection{The cellular resolution}
The minimal projective resolution of a dimer model algebra $A$ as an $(A,A)$-bimodule has been studied extensively in the literature \cite{Broomhead, Davison,MozgovoyReineke} under various assumptions on the dimer model. Here we assume algebraic consistency and describe in a uniform way the maps in the resolution using the toric cell complex $\Delta$. As a first step we show that $\Delta$ shares some key properties with regular cell complexes.

\begin{lemma}
\label{lem:dimerincidence}
The toric cell complex $\Delta$ of an algebaically consistent dimer model satisfies \eqref{eqn:facetsproperty}. In addition, $\Delta$ admits an incidence function, that is, a function $\varepsilon\colon \Delta\times \Delta \to \{0,\pm 1\}$ such that:
\begin{enumerate}
\item[\one] $\varepsilon(\eta,\eta^\prime)= 0$ unless $\eta^\prime$ is a facet of $\eta$;
\item[\two] $\varepsilon(\eta,\emptyset) = 1$ for all 0-cells $\eta$; and
\item[\three] if $\eta\in \Delta_k$ and $\eta^{\prime\prime}\in \Delta_{k-2}$ is a face of $\eta$, then for $\eta_1^\prime, \eta_2^\prime\in \Delta_{k-1}$ from \eqref{eqn:facetsproperty} we have 
\begin{equation}
\label{eqn:signconditiondimer}
\varepsilon(\eta,\eta_1^\prime) \varepsilon(\eta_1^\prime,\eta^{\prime\prime}) +\varepsilon(\eta,\eta_2^\prime) \varepsilon(\eta_2^\prime,\eta^{\prime\prime})= 0.
\end{equation}
\end{enumerate}
\end{lemma}
\begin{proof}
Statement \eqref{eqn:facetsproperty} is immediate for $k=2$. As for $k=3$, fix a codimension-two face $\eta^{\prime\prime}\in \Delta_1$ of a 3-cell $\eta\in \Delta_3$. Consider three cases: either the tails coincide $\tail(\eta)=\tail(\eta^{\prime\prime})$; the heads coincide $\head(\eta)=\head(\eta^{\prime\prime})$; or neither heads nor tails coincide. In the first two cases the arrow $a\in Q_1$ whose support is $\eta^{\prime\prime}$ is contained in precisely two anticanonical cycles, say $p_1, p_2$, each of which traverses only arrows supported on 1-cells in $\eta$. In the first case, if $a_j\in \supp(p_j)$ denotes the arrow with head at $\head(\eta)$ for $j=1,2$, then the 2-cells $\eta_1^\prime, \eta_2^\prime\subset \eta$ dual to the arrows $a_1, a_2$ are the unique 2-cells in $\eta$ containing $\eta^{\prime\prime}$. In the second case, if $a_j\in \supp(p_j)$ denotes the arrow with tail $\tail(\eta)$ for $j=1,2$, then the cells $\eta_1^\prime, \eta_2^\prime\subset \eta$ dual to $a_1, a_2$ are the unique 2-cells in $\eta$ containing $\eta^{\prime\prime}$. In the third case,  the arrow $a\in Q_1$ whose support is $\eta^{\prime\prime}$ is contained in a unique anticanonical cycle $p$ that traverses only arrows supported on 1-cells in $\eta$. If we let $a_1, a_2\in \supp(p)$ denote the unique arrows satisfying $\head(a_1)=\head(p)$ and $\tail(a_2)=\tail(p)$, then the 2-cells $\eta_1^\prime, \eta_2^\prime\subset \eta$ dual to the arrows $a_1, a_2$ are the unique 2-cells in $\eta$ containing $\eta^{\prime\prime}$. This establishes \eqref{eqn:facetsproperty}.

To construct an incidence function we may assume properties \one\ and \two. Each cell $\eta\in \Delta_1$ is supported on a unique arrow $a\in Q_1$ and contains precisely two 0-cells $\head(a), \tail(a)\in \Delta_0$. Choosing $\varepsilon(\eta, \head(a)) = 1$ forces $\varepsilon(\eta, \tail(a)) = -1$ by \eqref{eqn:signconditiondimer}. Similarly, every 2-cell $\eta\in \Delta_2$ corresponds uniquely to a minimal generator of $J_\mathscr{E}$ and we use the signs from $W_\Gamma$ (see Remark~\ref{rem:signsW_G} below) to write this generator as $p_a^+-p_a^- = a_l^+\cdots a_1^+ - a_m^-\cdots a_1^-$ where the boundary of $\eta$ is supported on the arrows $\{a_1^+, \dots ,a_l^+, a_1^-, \dots a_m^-\}$. Identify each 1-cell with the corresponding arrow   and choose $\varepsilon(\eta,a_j^+)=1$ for $1\leq j\leq l$, in which case \eqref{eqn:signconditiondimer} forces $\varepsilon(\eta, a_j^-) = -1$ for $1\le j\leq m$. Finally, each 3-cell $\eta\in \Delta_3$ is such that every facet $\eta^\prime\subset \eta$ satisfies either $\head(\eta^\prime)=\head(\eta)$ or $\tail(\eta^\prime)=\tail(\eta)$. Choose $\varepsilon(\eta, \eta^\prime) = 1$ if $\head(\eta^\prime)=\head(\eta)$ in which case \eqref{eqn:signconditiondimer} forces $\varepsilon(\eta, \eta^\prime) = -1$ for $\tail(\eta^\prime)=\tail(\eta)$.

It remains to show that equation \eqref{eqn:signconditiondimer} holds for any $\eta\in \Delta_k$ and codimension-two face $\eta^{\prime\prime}\subset \eta$. There are three cases, where either the tails coincide $\tail(\eta)=\tail(\eta^{\prime\prime})$, the heads coincide $\head(\eta)=\head(\eta^{\prime\prime})$, or neither heads nor tails coincide. For $k=2$, the proof in each case is straightforward because $\eta^{\prime\prime}$ is a 0-cell. For $k=3$, consider the case $\tail(\eta)=\tail(\eta^{\prime\prime})$ where $\varepsilon(\eta,\eta_1^\prime)=\varepsilon(\eta,\eta^\prime_2)=-1$. The cell $\eta^{\prime\prime}\in \Delta_1$ is supported on an arrow $a\in Q_1$, and the signs $\varepsilon(\eta^\prime_1,\eta^{\prime\prime}), \varepsilon(\eta^\prime_2,\eta^{\prime\prime})$ differ as required because the pair of anticanonical paths that traverse arrow $a$ have opposite signs in $W_\Gamma$. The case with $\head(\eta)=\head(\eta^{\prime\prime})$ is similar. In the final case, the arrow corresponding to $\eta^{\prime\prime}$ lies in a unique anticanonical path and hence $\varepsilon(\eta_1^\prime,\eta^{\prime\prime}) = \varepsilon(\eta_2^\prime,\eta^{\prime\prime})$, but then one of $\eta_1^\prime, \eta_2^\prime$ shares head with $\eta$ while the other shares tail. Thus, the signs $\varepsilon(\eta, \eta_1^{\prime}), \varepsilon(\eta,\eta_2^{\prime})$ differ as required.
\end{proof}

Our main result below uses the toric cell complex to provide a simple, uniform description of all terms and differentials in the standard resolution associated to a quiver with superpotential in dimension three as studied by Ginzburg~\cite{Ginzburg}, Mozgovoy-Reineke~\cite{MozgovoyReineke}, Davison~\cite{Davison} and Broomhead~\cite{Broomhead}.

  \begin{theorem}
  \label{thm:dimerresolution}
 Let $\Gamma$ be an algebraically consistent dimer model with toric algebra $A$ and let $\varepsilon\colon \Delta\times \Delta \to \{0,\pm 1\}$ be any incidence function on $\Delta$. The minimal bimodule resolution of $A$ is the \emph{cellular resolution}
\[
 0\longrightarrow  P_3  \stackrel{d_3}{\longrightarrow} P_2  \stackrel{d_2}{\longrightarrow} P_1  \stackrel{d_1}{\longrightarrow} P_0 \stackrel{\mu}{\longrightarrow} A \longrightarrow 0
\]
 where for $0\leq k\leq 3$ we have
\[
P_k:=   \bigoplus_{\eta\in \Delta_k} A e_{\head(\eta)} \otimes [\eta] \otimes e_{\tail(\eta)} A,
 \]
where $\mu\colon P_0\to A$ is the multiplication map and  where $ d_k\colon P_k\longrightarrow P_{k-1}$ satisfies
\[
d_k\big(1\otimes[\eta]\otimes 1\big) = \sum_{\codim(\eta^\prime,\eta)=1} \varepsilon(\eta,\eta^\prime) \overleftarrow{\partial}_{\!\!\eta^\prime}\eta\otimes [\eta^\prime]\otimes \overrightarrow{\partial}_{\!\!\eta^\prime}\eta.
\]
 \end{theorem} 
\begin{proof}
This result can be proved directly by modifying the proofs from \cite{Broomhead, Davison,MozgovoyReineke}. However, we choose instead to consider a particular incidence function that realises precisely the maps from the resolution of Broomhead~\cite{Broomhead}. Then, just as in the proof of Proposition~\ref{prop:McKaycomplex}, choosing any alternative incidence function merely provides an isomorphic resolution of $A$.

Let $\varepsilon\colon \Delta\times \Delta \to \{0,\pm 1\}$  denote the incidence function constructed in the proof of Lemma~\ref{lem:dimerincidence}.  To compute the differentials for this choice of incidence function, note first that
\[
d_1\big(1\otimes[a]\otimes 1\big) =  1\otimes[\head(a)]\otimes a - a\otimes [\tail(a)]\otimes 1.
\]
For the 2-cell $\eta\in \Delta$ corresponding to the relation $p_a^+-p_a^- = a_l^+\cdots a_1^+ - a_m^-\cdots a_1^-$, we have
\[
d_2\big(1\otimes[\eta]\otimes 1\big) = \sum_{j=1}^l   a_l^+\cdots a_{j+1}^+\otimes [a_j^+]\otimes a_{j-1}^+\cdots a_1^+  -  \sum_{j=1}^m  a_m^-\cdots a_{j+1}^-\otimes [a_j^-]\otimes a_{j-1}^-\cdots a_1^- .
\]
Finally, consider $\eta\in \Delta_3$ with $i:= \tail(\eta)=\head(\eta)\in Q_0$. The facets $\eta^\prime\subset \eta$ that satisfy $\head(\eta^\prime)=\head(\eta)$ have left-derivative $\overleftarrow{\partial}_{\!\eta'}\eta = e_i$ and right-derivative $\overrightarrow{\partial}_{\!\eta'}\eta =a$ for an arrow $a\in Q_1$ with $\tail(a)=i$. Similarly, the  facets $\eta^\prime\subset \eta$ that satisfy $\tail(\eta^\prime)=\tail(\eta)$ have left-derivative $\overleftarrow{\partial}_{\!\eta'}\eta = a$ for an arrow $a\in Q_1$ with $\head(a)=i$ and right-derivative $\overrightarrow{\partial}_{\!\eta'}\eta =e_i$. In each case relabel the facet as $\eta^\prime_a:=\eta^\prime$ for the corresponding arrow $a\in Q_1$. Then
\begin{eqnarray*}
d_3\big(1\otimes[\eta]\otimes 1\big) & = & \sum_{\codim(\eta_a^\prime,\eta)=1,\; \head(\eta_a^\prime)=i} 1\otimes [\eta^\prime_a]\otimes a - \sum_{\codim(\eta_a^\prime,\eta)=1,\; \tail(\eta_a^\prime)=i} a\otimes [\eta^\prime_a]\otimes 1\\
 & = &  \sum_{\{a\in Q_1 \mid \tail(a)=i\}} 1\otimes [\eta^\prime_a]\otimes a - \sum_{\{a\in Q_1 \mid \head(a)=i\}} a\otimes [\eta^\prime_a]\otimes 1.
\end{eqnarray*}
Our differentials are seen to coincide with those from \cite[Theorem 7.3]{Broomhead}, though note that our convention for composing arrows (where $a^\prime a$ means `$a^\prime$ follows $a$') differs from that in \cite{Broomhead}. 
\end{proof}

\begin{remark}
\label{rem:signsW_G}
In the course of the proof we use the signs in the dimer superpotential $W_\Gamma$ to write $p_a^+-p_a^-$, and hence to choose the incidence function $\varepsilon$. However, we chose this incidence function only to reproduce precisely Broomhead's resolution. Since any incidence function on $\Delta$ suffices for Theorem~\ref{thm:dimerresolution},  knowledge of the signs of $W_\Gamma$ is unnecessary in general.

\end{remark}

 \section{The cellular resolution conjecture}
 \label{sec:conjecture}
We conclude by formulating a conjecture on the existence of toric cell complexes and cellular resolutions for consistent toric algebras in arbitrary dimension. As a first step we illustrate two key ingredients by presenting an example of a four-dimensional consistent toric algebra.  A nice class of such algebras arises from tilting bundles on smooth toric Fano threefolds, and we study here a representative example of this class.

\subsection{A consistent fourfold example}
\label{sec:conjectureexample}
 Let $X=\Spec \Bbbk[\sigma^{\vee}\cap M]$ be the Gorenstein toric fourfold determined by the cone $\sigma$ generated by the vectors $v_1=(1,0,0,1)$, $v_2=(0,1,0,1)$, $v_3=(0,0,1,1)$, $v_4=(-1,-1,2,1)$, $v_5=(-1,-1,1,1)$ and $v_6=(0,0,-1,1)$. For $1 \leq \rho \leq 6$,  write $D_\rho$ for the toric divisor in $X$ defined by the ray of $\sigma$ generated by $v_\rho$. Then $\Cl (X)$ is the quotient of the free abelian group generated by $\mathcal{O}_X(D_1)$, $\mathcal{O}_X(D_5+D_6)$, $\mathcal{O}_X(D_6)$, by the subgroup generated by $\mathcal{O}_X(D_1+D_5+2D_6)$. The singularity $X$ admits several crepant resolutions $\tau \colon Y \rightarrow X$, one of which is given by the total space $\mathrm{tot}(\omega_Z)$ of the canonical bundle of the smooth Fano threefold $Z$ listed as number 11 by Oda~\cite[Figure~2.7]{Oda}.  Consider the collection
$$
\mathscr{E}  = \left(\begin{array}{cc} \mathcal{O}_X, \mathcal{O}_X(D_1), \mathcal{O}_X(2D_1),\mathcal{O}_X(D_6),\mathcal{O}_X(D_5+D_6),\\
\mathcal{O}_X(D_1+D_6), \mathcal{O}_X(D_1+D_5+D_6),\mathcal{O}_X(2D_1+D_6) \end{array}\right)
$$
on $X$. This collection is obtained from a tilting bundle\footnote{Our chosen Fano $Z$ provides an interesting starting point since the construction of tilting bundles by Bondal~\cite{BondalMFO} does not apply. The tilting bundle here was constructed originally by Greg Smith~\cite[report by Craw]{BondalMFO}.} on $Z$ by pulling back each summand via $\mathrm{tot}(\omega_Z)\to Z$ and then pushing forward via the crepant resolution $\mathrm{tot}(\omega_Z)\to X$. 

The quiver of sections $Q$ from Figure~\ref{fig:3foldFano11} 
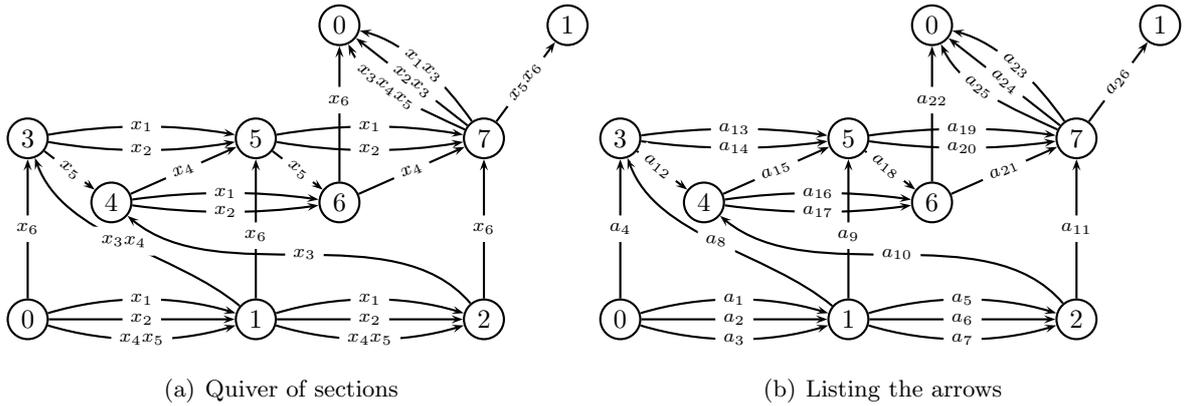
\begin{figure}[!ht]
    \centering
   \subfigure[Quiver of sections]{
  \psset{unit=1cm}
   \begin{pspicture}(0.4,-0.5)(6.4,4)
        \cnodeput(0,0){A}{\small{0}}
        \cnodeput(3,0){B}{\small{1}} 
        \cnodeput(6,0){C}{\small{2}} 
        \cnodeput(0,2.4){D}{\small{3}}
           \cnodeput(1.1,1.55){E}{\small{4}} 
         \cnodeput(3,2.4){F}{\small{5}} 
            \cnodeput(4.1,1.55){G}{\small{6}} 
        \cnodeput(6,2.4){H}{\small{7}}
          \cnodeput(4.1,3.9){I}{\small{0}}  
     \cnodeput(7.1,3.9){J}{\small{1}}   
    \psset{nodesep=0pt}
     \nccurve[angleA=17,angleB=163]{->}{A}{B}\lput*{:U}{\tiny{$x_1$}}
     \ncline{->}{A}{B}\lput*{:U}{\tiny{$x_2$}}
     \nccurve[angleA=343,angleB=197]{->}{A}{B}\lput*{:U}{\tiny{$x_4x_5$}}
     \ncline{->}{A}{D}\lput*{:270}{\tiny{$x_6$}}
     \nccurve[angleA=17,angleB=163]{->}{B}{C}\lput*{:U}{\tiny{$x_1$}}
     \ncline{->}{B}{C}\lput*{:U}{\tiny{$x_2$}}
     \nccurve[angleA=343,angleB=197]{->}{B}{C}\lput*{:U}{\tiny{$x_4x_5$}}   
     \nccurve[angleA=140,angleB=290]{->}{B}{D}\lput*{:205}{\tiny{$x_3x_4$}}
     \nccurve[angleA=130,angleB=320]{->}{C}{E}\lput*{:180}{\tiny{$x_3$}}
      \ncline{->}{B}{F}\lput*{:270}(0.45){\tiny{$x_6$}} 
     \ncline{->}{C}{H}\lput*{:270}{\tiny{$x_6$}}
     \ncline{->}{G}{I}\lput*{:270}(0.6){\tiny{$x_6$}}
     \ncline{->}{D}{E}\lput*{:U}{\tiny{$x_5$}}   
     \ncline{->}{F}{G}\lput*{:U}{\tiny{$x_5$}}       
     \ncline{->}{E}{F}\lput*{:U}{\tiny{$x_4$}}
     \ncline{->}{G}{H}\lput*{:U}{\tiny{$x_4$}}         
     \nccurve[angleA=10,angleB=170]{->}{D}{F}\lput*{:U}{\tiny{$x_1$}}
     \nccurve[angleA=350,angleB=190]{->}{D}{F}\lput*{:U}{\tiny{$x_2$}}
     \nccurve[angleA=10,angleB=170]{->}{F}{H}\lput*{:U}{\tiny{$x_1$}}
     \nccurve[angleA=350,angleB=190]{->}{F}{H}\lput*{:U}{\tiny{$x_2$}}
     \nccurve[angleA=125,angleB=345]{->}{H}{I}\lput*{:180}{\tiny{$x_1 x_3$}}
     \ncline{->}{H}{I}\lput*{:180}{\tiny{$x_2x_3$}}               
     \nccurve[angleA=155,angleB=295]{->}{H}{I}\lput*{:180}(0.6){\tiny{$x_3 x_4 x_5$}}
     \nccurve[angleA=8,angleB=172]{->}{E}{G}\lput*{:U}{\tiny{$x_1$}}
     \nccurve[angleA=352,angleB=188]{->}{E}{G}\lput*{:U}{\tiny{$x_2$}}
     \ncline{->}{H}{J}\lput*{:U}{\tiny{$x_5x_6$}} 
  \end{pspicture}}
     \qquad \qquad  
     \subfigure[Listing the arrows]{
     \psset{unit=1cm}
     \begin{pspicture}(0.4,-0.5)(6.5,4)
      \cnodeput(0,0){A}{\small{0}}
      \cnodeput(3,0){B}{\small{1}} 
      \cnodeput(6,0){C}{\small{2}} 
      \cnodeput(0,2.4){D}{\small{3}}
      \cnodeput(1.1,1.55){E}{\small{4}} 
      \cnodeput(3,2.4){F}{\small{5}} 
      \cnodeput(4.1,1.55){G}{\small{6}} 
      \cnodeput(6,2.4){H}{\small{7}}
      \cnodeput(4.1,3.9){I}{\small{0}}  
      \cnodeput(7.1,3.9){J}{\small{1}}    
      \psset{nodesep=0pt}
      \nccurve[angleA=17,angleB=163]{->}{A}{B}\lput*{:U}{\tiny{$a_1$}}
      \ncline{->}{A}{B}\lput*{:U}{\tiny{$a_2$}}
      \nccurve[angleA=343,angleB=197]{->}{A}{B}\lput*{:U}{\tiny{$a_3$}}
      \ncline{->}{A}{D}\lput*{:270}{\tiny{$a_4$}}
      \nccurve[angleA=17,angleB=163]{->}{B}{C}\lput*{:U}{\tiny{$a_5$}}
      \ncline{->}{B}{C}\lput*{:U}{\tiny{$a_6$}}
      \nccurve[angleA=343,angleB=197]{->}{B}{C}\lput*{:U}{\tiny{$a_7$}}
      \nccurve[angleA=140,angleB=290]{->}{B}{D}\lput*{:205}{\tiny{$a_8$}}
      \ncline{->}{B}{F}\lput*{:270}(0.45){\tiny{$a_9$}} 
      \nccurve[angleA=130,angleB=320]{->}{C}{E}\lput*{:180}{\tiny{$a_{10}$}}
      \ncline{->}{C}{H}\lput*{:270}{\tiny{$a_{11}$}}
      \nccurve[angleA=8,angleB=172]{->}{D}{F}\lput*{:U}{\tiny{$a_{13}$}}
      \nccurve[angleA=352,angleB=188]{->}{D}{F}\lput*{:U}{\tiny{$a_{14}$}}
      \ncline{->}{D}{E}\lput*{:U}(0.4){\tiny{$a_{12}$}}   
      \ncline{->}{E}{F}\lput*{:U}{\tiny{$a_{15}$}}
      \nccurve[angleA=8,angleB=172]{->}{E}{G}\lput*{:U}{\tiny{$a_{16}$}}
      \nccurve[angleA=352,angleB=188]{->}{E}{G}\lput*{:U}{\tiny{$a_{17}$}}
      \ncline{->}{F}{G}\lput*{:U}(0.4){\tiny{$a_{18}$}}       
      \nccurve[angleA=8,angleB=172]{->}{F}{H}\lput*{:U}{\tiny{$a_{19}$}}
      \nccurve[angleA=352,angleB=188]{->}{F}{H}\lput*{:U}{\tiny{$a_{20}$}}
      \ncline{->}{G}{H}\lput*{:U}{\tiny{$a_{21}$}}         
      \ncline{->}{G}{I}\lput*{:270}(0.6){\tiny{$a_{22}$}}
      \nccurve[angleA=125,angleB=345]{->}{H}{I}\lput*{:180}{\tiny{$a_{23}$}}
      \ncline{->}{H}{I}\lput*{:180}{\tiny{$a_{24}$}}               
      \nccurve[angleA=155,angleB=295]{->}{H}{I}\lput*{:180}(0.6){\tiny{$a_{25}$}}
      \ncline{->}{H}{J}\lput*{:U}{\tiny{$a_{26}$}} 
  \psset{linecolor=lightgray}
        \end{pspicture}
             }
    \caption{A cyclic quivers of sections on the Gorenstein toric fourfold $X$}
  \label{fig:3foldFano11}
  \end{figure}
 is depicted in $\ZZ^3$, but we work in the class group of $X$ and hence take $ \mathcal{O}_X \sim  \mathcal{O}_X(D_1+D_5+2D_6)$.  If we order the arrows as in Figure~\ref{fig:3foldFano11}(b), then the superpotential is the sum
\begin{align*}
W&=a_{22}a_{17}a_{10}a_{5}a_3 + a_{22}a_{16}a_{10}a_{6}a_3 +a_{22}a_{18}a_{15}a_{10}a_6a_1+a_{22}a_{18}a_{15}a_{10}a_5a_2+ a_{22}a_{18}a_{13}a_{8}a_2 \\
& \quad+a_{22}a_{17}a_{10}a_{7}a_{1} + a_{22}a_{17}a_{12}a_{8}a_1 +a_{22}a_{18}a_{14}a_{8}a_1 + a_{22}a_{16}a_{10}a_{7}a_2 +a_{22}a_{16}a_{12}a_{8}a_2  \\ 
& \quad+ a_{23}a_{21}a_{18}a_{9}a_2 + a_{23}a_{20}a_{15}a_{12}a_4+ a_{23}a_{21}a_{17}a_{12}a_4 + a_{23}a_{21}a_{18}a_{14}a_4 + a_{23}a_{11}a_{7}a_{2} \\
& \quad+ a_{23}a_{11}a_{6}a_{3} + a_{23}a_{20}a_{9}a_{3}  + a_{24}a_{11}a_{7}a_{1} +a_{24}a_{11}a_{5}a_{3}+a_{24}a_{21}a_{18}a_{13}a_4  +a_{24}a_{21}a_{16}a_{12}a_4 \\ 
&  \quad + a_{24}a_{19}a_{9}a_{3} + a_{24}a_{21}a_{18}a_{9}a_1 + a_{24}a_{19}a_{15}a_{12}a_4+ a_{25}a_{19}a_{14}a_{4} + a_{25}a_{11}a_{6}a_{1} + a_{25}a_{20}a_{9}a_{1}  \\
 & \quad  + a_{25}a_{11}a_{5}a_{2} + a_{25}a_{19}a_{9}a_{2} + a_{25}a_{20}a_{13}a_{4} + a_{26}a_{21}a_{17}a_{10}a_5 + a_{26}a_{19}a_{14}a_{8}  \\
& \quad + a_{26}a_{19}a_{15}a_{10}a_6 + a_{26}a_{21}a_{16}a_{10}a_6 + a_{26}a_{20}a_{13}a_{8} + a_{26}a_{20}a_{15}a_{10}a_5
\end{align*}
of all anticanonical cycles in $Q$. By taking partial derivatives, we compute that
    \[
  J_W=\left(\begin{array}{rl} 
  a_6a_1-a_5a_2, a_{25}a_{11} - a_{22}a_{18}a_{15}a_{10}, & a_9a_1 - a_{13}a_4, a_{24}a_{21}a_{18} - a_{25}a_{20} \\
  a_{14}a_{4} - a_9a_2, a_{23}a_{21}a_{18} - a_{25}a_{19},  & a_5a_3-a_7a_1, a_{22}a_{17}a_{10} - a_{24}a_{11} \\
  a_{9}a_3 - a_{15} a_{12}a_{4}, a_{24}a_{19} - a_{23}a_{20}, &  a_{11}a_{6} - a_{20}a_{9}, a_{1}a_{25} - a_3a_{23}\\
  a_{11}a_{5} - a_{19}a_{9}, a_2a_{25} - a_3a_{24}, & a_{20}a_{13} - a_{19} a_{14}, a_4a_{25}-a_8a_{26} \\
  a_{11}a_{7} - a_{21}a_{18}a_{9}, a_1a_{24} - a_2a_{23}, & a_{19}a_{15} - a_{21}a_{16}, a_{12}a_{4}a_{24} - a_{10}a_{6}a_{26} \\
  a_{16}a_{12} - a_{18}a_{13}, a_{8}a_2a_{22} - a_4a_{24}a_{21},&  a_7a_2 - a_6a_3, a_{22}a_{16}a_{10}-a_{23}a_{11} \\
  a_{17}a_{12} - a_{18}a_{14}, a_8a_1a_{22} - a_4a_{23}a_{21}, & a_{20}a_{15} - a_{21}a_{17}, a_4a_{23}a_{12} - a_{10}a_5a_{26} \\
  a_{17}a_{10}a_{5} - a_{16}a_{10}a_{6}, a_3a_{22} - a_{26}a_{21}, & a_{10}a_{7}-a_{12}a_{8}, a_1a_{22}a_{17} - a_2a_{22}a_{16} \\ 
  a_{14}a_{8} - a_{15}a_{10}a_{6}, a_1a_{22}a_{18} - a_{26}a_{19}, & a_{15}a_{10}a_{5} - a_{13}a_{8}, a_2a_{22}a_{18} - a_{26}a_{20}
    \end{array}\right).
 \]
 This ideal is equal to $J_{\mathscr{E}}$, so the toric algebra $A:=A_\mathscr{E}$ is consistent. 

\begin{remark}
 For each minimal generator $p^+-p^-$ of $J_W$ there exists another minimal generator $q^+-q^-$ with the property that $\tail(q^{\pm})=\head(p^{\pm})$, $\head(q^{\pm})=\tail(p^{\pm})$ and $x^{\div(q^{\pm})}=\prod_{\rho=1}^6 x_\rho / x^{\div(p^{\pm})}$. We list two such pairs on each line in $J_W$ above. This phenomenon is one aspect of the duality property of $\Delta$ described in Proposition~\ref{prop:duality4fold} below.
\end{remark}
  
\subsection{The cellular resolution for the fourfold example}
\label{sec:conjectureresolution}
We now sketch the construction of the toric cell complex $\Delta\subset \mathbb{T}^4$ for this example.  The analogue of diagram \eqref{eqn:3torusdiagram} is
 \begin{equation}
 \label{eqn:4torusdiagram}
  \begin{CD}   
    0@>>> M  @>>> \ZZ^6    @>{\deg}>> \Cl(X)@>>> 0\\
     @.   @|            @VV{f\vert_{\ZZ^6}}V   @.      @.          \\
0 @>>> M @>>> \RR^4  @>>> \mathbb{T}^4 @>>> 0 
 \end{CD}
 \end{equation}
where $f\colon \RR^6\to \RR^4:=M\otimes_\ZZ \RR$ is orthogonal projection on to the subspace spanned by $M$. Since $A$ is consistent, Corollary~\ref{cor:arrowsinA} shows that the covering quiver $\widetilde{Q}\subset \RR^6$ is the union of all $M$-translates of the quivers $\widetilde{Q}(i)$ for $i\in Q_0$. Explicit computation of the image $f(\widetilde{Q}(i))$ for $i\in Q_0$ shows that $\bigcup_{i\in Q_0} f(\widetilde{Q}(i))$ is an embedded quiver in $\RR^4$.  Define $\Delta_0$ and $\Delta_1$ to be the union of all $M$-translates of the vertex set and arrow set respectively of this quiver. The resulting subsets of $\mathbb{T}^4$ define the collections $\Delta_0$ of $0$-cells and $\Delta_1$ of 1-cells and, just as in Lemma~\ref{lem:McKaybijections}, there are canonical bijections between $Q_0$ and $\Delta_0$, and between $Q_1$ and $\Delta_1$. Moreover,  a lengthy and tedious calculation shows that we may define collections $\Delta_k$ of $M$-periodic subsets in $\RR^4$ for $k=2,3,4$, where the canonical bijection for $\Delta_2$ from  Lemma~\ref{lem:McKaybijections} also holds.  The corresponding closed subsets in $\mathbb{T}^4$ define the \emph{toric cell complex} $\Delta\subset \mathbb{T}^4$. As before, each cell $\eta$ in $\Delta$ has a well-defined head $\head(\eta)\in \Delta_0$, tail $\tail(\eta)\in \Delta_0$ and label $\div(\eta)\in \NN^6$. 

In this case, we have $\vert \Delta_0\vert = \vert \Delta_4\vert = 8$, that $\vert \Delta_1\vert = \vert \Delta_3\vert = 26$, and that $\vert \Delta_2\vert = 36$. Explicit computation and inspection shows that $\Delta$ satisfies the following duality property:

\begin{proposition}
\label{prop:duality4fold}
The map $\tau\colon \Delta\to \Delta$ that assigns to each $\eta\in \Delta_k$ the unique cell $\eta^\prime\in \Delta_{4-k}$ with $\tail(\eta^\prime)=\head(\eta)$, $\head(\eta^\prime)=\tail(\eta)$ and $x^{\div(\eta^\prime)} = \prod_{\rho=1}^6x_\rho/x^{\div(\eta)}$ is an involution. 
\end{proposition}

Figure~\ref{fig:4foldToricCell} depicts several cells of $\Delta$. Figure~\ref{fig:4foldToricCell}(a) shows arrow $a_{23}$ and the dual 3-cell, where paths from vertex 0 at the bottom to vertex 0 at the top traverse anticanonical cycles in $Q\subset \mathbb{T}^4$.  Figure~\ref{fig:4foldToricCell}(b) shows the 3-cells dual to arrows $a_1$ and $a_{24}$ intersect along a given shaded 2-cell, illustrating that $\Delta$ satisfies property \eqref{eqn:facetsproperty}. Note that both 3-cells from Figure~\ref{fig:4foldToricCell}(b) lie in the 4-cell dual to vertex $0$, though we do not draw every edge in the 4-cell for the sake of clarity. Figures~\ref{fig:4foldToricCell}(c) is similar, and shows for example that the 3-cell dual to $a_5$ is not equidimensional.
     \begin{figure}[!ht]
    \centering
   \subfigure[]{
 \psset{unit=1.2cm}             
                \begin{pspicture}(0.2,-0.3)(3,4.4)
        \cnodeput(1.5,0){A}{\tiny{0}}
        \cnodeput(3,0.8){B}{\tiny{1}}
        \cnodeput(0.3,0.8){C}{\tiny{1}} 
        \cnodeput(1.5,0.8){D}{\tiny{3}}
        \cnodeput(0.8,2.1){E}{\tiny{4}} 
        \cnodeput(2.3,1.6){F}{\tiny{5}} 
        \cnodeput(1.8,2.6){G}{\tiny{6}} 
        \cnodeput(0,2.9){H}{\tiny{5}}
        \cnodeput(3,2.9){I}{\tiny{2}}  
        \cnodeput(1.7,3.6){J}{\tiny{7}}
        \cnodeput(0.7,4.5){K}{\tiny{0}}
        \psset{nodesep=0pt}
        \ncline{->}{A}{B}\lput*{:U}{\tiny{$a_{2}$}}
        \ncline{->}{A}{C}\lput*{:180}{\tiny{$a_{3}$}}
        \ncline{->}{A}{D}\lput*{:270}{\tiny{$a_{4}$}}
        \ncline{->}{B}{F}\lput*{:180}{\tiny{$a_{9}$}}
        \ncline{->}{B}{I}\lput*{:270}{\tiny{$a_{7}$}}
        \ncline{->}{C}{H}\lput*{:270}{\tiny{$a_{9}$}}
        \ncline{->}{D}{F}\lput*{:U}{\tiny{$a_{14}$}}
        \ncline{->}{E}{H}\lput*{:180}{\tiny{$a_{15}$}}
        \ncline{->}{E}{G}\lput*{:U}{\tiny{$a_{17}$}}
        \ncline{->}{G}{J}\lput*{:280}{\tiny{$a_{21}$}}
        \ncline{->}{H}{J}\lput*{:U}{\tiny{$a_{20}$}}
        \ncline{->}{I}{J}\lput*{:180}{\tiny{$a_{11}$}}
        \ncline{->}{J}{K}\lput*{:230}{\tiny{$a_{23}$}}
        \psset{linestyle=solid}
        \ncline{->}{C}{I}\lput*{:U}{\tiny{$a_{6}$}}
        \psset{linestyle=solid}
        \ncline{->}{D}{E}\lput*{:180}{\tiny{$a_{12}$}}
        \ncline{->}{F}{G}\lput*{:180}{\tiny{$a_{18}$}}
    \end{pspicture}             
}
   \qquad \qquad  \qquad
 \subfigure[]{
  \psset{unit=0.71cm}
   \begin{pspicture}(-0.1,-0.3)(3,7.5)
    \pspolygon[linecolor=lightgray,fillstyle=hlines*,hatchangle=34](0,1)(0.52,5)(1.15,6)(2.3,7)(1.5,2)
      \cnodeput(1.5,0){A}{\tiny{0}}
      \cnodeput(0,1){B}{}
      \cnodeput(3,1){C}{} 
      \cnodeput(-0.9,2){D}{}
      \cnodeput(1.5,2){E}{} 
      \cnodeput(-0.3,3){F}{} 
      \cnodeput(2.25,3){G}{} 
      \cnodeput(-2,4){H}{}
      \cnodeput(0.8,4){I}{}  
      \cnodeput(3,4){J}{}
      \cnodeput(-1.85,5){K}{}
      \cnodeput(0.52,5){L}{}
      \cnodeput(4.25,5){M}{} 
      \cnodeput(-1.55,6){N}{}
      \cnodeput(1.15,6){O}{}    
      \cnodeput(-0.15,7){P}{}
      \cnodeput(2.3,7){Q}{}     
      \cnodeput(0.25,8){R}{\tiny{0}}            
        \psset{nodesep=0pt}
            \ncline{->}{A}{B}\lput*{:180}{\tiny{$a_{1}$}}
              \ncline{->}{A}{C}
               \ncline{->}{A}{G}
          \ncline{->}{B}{D}
         \ncline{->}{B}{F}
         \ncline{->}{B}{E}
         \ncline{->}{B}{L}
         \ncline{->}{C}{E}
          \ncline{->}{C}{M}
          \ncline{->}{D}{H}
          \ncline{->}{D}{P}
          \ncline{->}{E}{I}
          \ncline{->}{E}{Q}
          \ncline{->}{F}{K}
          \ncline{->}{F}{I}
          \ncline{->}{G}{L}
          \ncline{->}{G}{J}
          \ncline{->}{H}{K}
          \ncline{->}{I}{N}
          \ncline{->}{J}{O}
          \ncline{->}{J}{M}
          \ncline{->}{K}{N}
          \ncline{->}{L}{O}
          \ncline{->}{L}{P}
          \ncline{->}{M}{Q}
          \ncline{->}{N}{R}
          \ncline{->}{O}{Q}
          \ncline{->}{P}{R}
          \ncline{->}{Q}{R}\lput*{:180}{\tiny{$a_{24}$}}
        \end{pspicture}  
  }
   \qquad \qquad  
   \qquad 
     \subfigure[]{
     \psset{unit=0.71cm}
  \begin{pspicture}(-0.2,-0.3)(3.3,7.5)
   \pspolygon[linecolor=lightgray,fillstyle=hlines*,
          hatchangle=45](1,5)(0.75,6)(2.35,7)(2.95,6)
        \cnodeput(0.5,0){A}{\tiny{2}}
                \cnodeput(2.25,2){B}{}
        \cnodeput(3.5,3){C}{} 
        \cnodeput(0.15,4){D}{}
           \cnodeput(3.85,4){E}{} 
         \cnodeput(1,5){F}{} 
         \cnodeput(1.72,5){G}{} 
        \cnodeput(-0.5,6){H}{}
         \cnodeput(0.75,6){I}{}  
            \cnodeput(2.95,6){J}{}
              \cnodeput(-0.75,7){K}{}
                \cnodeput(0.25,7){L}{}
            \cnodeput(2.35,7){M}{} 
             \cnodeput(0,8){N}{\tiny{2}}
         \psset{nodesep=0pt}
          \ncline{->}{A}{B}
                \ncline{->}{A}{F}\lput*{:U}{\tiny{$a_{11}$}}
             \ncline{->}{B}{C}
         \ncline{->}{B}{D}
             \ncline{->}{C}{E}
          \ncline{->}{C}{G}
          \ncline{->}{D}{G}
          \ncline{->}{D}{I}
          \ncline{->}{E}{J}
          \ncline{->}{F}{H}
          \ncline{->}{F}{I}
          \ncline{->}{F}{J}
          \ncline{->}{G}{M}
          \ncline{->}{H}{K}
          \ncline{->}{H}{L}
          \ncline{->}{I}{K}
          \ncline{->}{I}{M}
          \ncline{->}{J}{L}
            \ncline{->}{J}{M}
              \ncline{->}{K}{N}
                \ncline{->}{L}{N}
                  \ncline{->}{M}{N}\lput*{:180}{\tiny{$a_{5}$}}
             \end{pspicture}
          }
         \caption{Cells of $\Delta$: (a) the 3-cell dual to $a_{23}$; (b) the $3$-cells in $\Delta(0)$ dual to arrows $a_1$ and $a_{24}$;  (c) the $3$-cells in $\Delta(2)$ dual to arrows $a_{11}$ and $a_{5}$.}
  \label{fig:4foldToricCell}
  \end{figure}
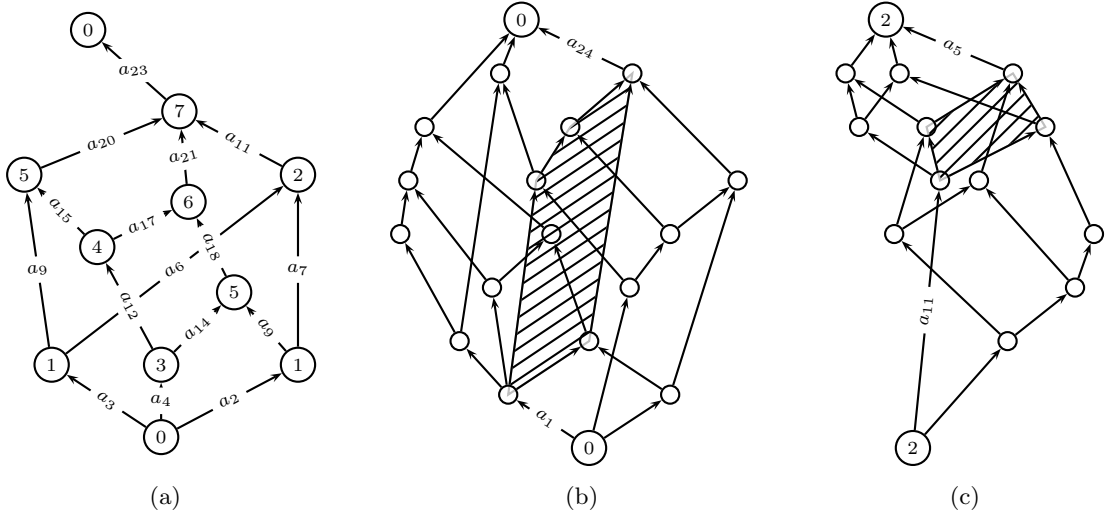
 
 We verify by an exhaustive examination that $\Delta$ that satisfies property \eqref{eqn:facetsproperty} and, in addition, that $\Delta$ admits an incidence function $\varepsilon\colon \Delta\times \Delta\to\{0,\pm 1\}$. As a result,  the minimal projective $(A,A)$-bimodule resolution of the consistent toric algebra $A=A_\mathscr{E}$ can be constructed as a cellular resolution. Indeed, it can be shown directly, by adapting the proof in \cite{Broomhead}, that the minimal projective resolution of $A$ as a $(A,A)$-bimodule is
 \begin{equation}
 \label{eqn:4foldresolution}
0 \longrightarrow P_4 \xlongrightarrow{d_4} P_3 \xlongrightarrow{d_3} P_2 \xlongrightarrow{d_2} P_1 \xlongrightarrow{d_1} P_0 \xlongrightarrow{\mu} A \longrightarrow 0,
 \end{equation}
  with terms $\displaystyle{P_k=\bigoplus_{\eta \in \Delta_k} Ae_{\head(\eta)}\otimes [\eta] \otimes Ae_{\tail(\eta)}}$ and differentials
 $$
 d_{k}(1 \otimes [\eta] \otimes 1)=\sum_{\codim(\eta',\eta)=1}\varepsilon(\eta,\eta')  \overleftarrow{\partial}_{\!\eta'}\eta\otimes[\eta'] \otimes \overrightarrow{\partial}_{\!\eta'}\eta,
$$
where $\mu\colon P_0=\bigoplus_{i \in \Delta_0} Ae_{i}\otimes Ae_{i} \rightarrow A$ is the multiplication map.

\subsection{On signs and syzygies}
Before stating the main conjecture we make a key observation which explains and justifies our decision to introduce no signs in the superpotentials throughout this paper, namely, that \emph{it is impossible to introduce signs in the superpotential $W$ above so that the generators of $J_W$ can be recovered directly by taking partial derivatives}. 

Indeed, consider only terms of $W$ that involve $a_{23}\in Q_1$, namely, those arising in $a_{23}\partial_{a_{23}}W$. Figure~\ref{fig:4foldToricCell}(a) illustrates all seven of the corresponding anticanonical cycles in $Q\subset \mathbb{T}^4$: the 3-cell $\eta_{23}$ dual to arrow $a_{23}$ is drawn as a convex 3-polytope with arrow $a_{23}$ sticking out of the top. The facets of $\eta_{23}$ correspond to those generators of $J_W$ arising from partials of $W$ with respect to paths involving $a_{23}$, e.g.,  relation $a_{14}{a_4}-a_9a_2$ arises from $\partial_{a_{23}a_{21}a_{18}}W$.   We now attempt to introduce signs in $W$ so that the generators of $J_W$ are recovered directly from partial derivatives of $W$.  If, say, we fix the sign of $a_{23}a_{21}a_{18}a_{9}a_2$ to be $+1$, then the relation $a_{14}a_{4} - a_9a_2$ forces the sign of  $a_{23}a_{21}a_{18}a_{14}a_4$ to be $-1$, but then $a_{17}a_{12} - a_{18}a_{14}$ forces the sign of $a_{23}a_{21}a_{17}a_{12}a_4$ to be $+1$, and so on. By repeating, we hop from one anticanonical cycle to another around the surface of the polytope $\eta_{23}$. There are an odd number of paths, so we obtain sign $-1$ for the original path $a_{23}a_{21}a_{18}a_{9}a_2$ after passing once around $\eta_{23}$. This contradiction shows that introducing signs in $W$ cannot produce the necessary signs in $J_W$. 
  
\begin{remark}
Comparing the third term $P_3$ in the cellular resolution \eqref{eqn:NCBSresolution} with the third term in the resolution as described by Butler--King~\cite[(1.1)]{ButlerKing} shows that the set of 3-cells $\Delta_3$ provides a minimal set of bimodule generators for the space of syzygies
\[
A\otimes\text{Tor}^A_3(U_0,U_0)\otimes A\cong A\otimes \frac{JI\cap IJ}{I^2+JIJ}\otimes A,
\]
where $I:=J_\mathscr{E}$ and $J$ is the ideal in $\kk Q$ generated by the set of  arrows. As Alastair King remarks, the syzygy corresponding to the cell $\eta_{23}\in \Delta_3$ from Figure~\ref{fig:4foldToricCell}(a) provides an equation with signs that includes all seven terms of $W$ involving $a_{23}$. This is indeed the case, but this does not contradict the assertion above.  To see this, list all seven relations defined by codimension-one faces of $\eta_{23}$ as $r_1 := a_{14}a_4 - a_9a_2$, $r_2 := a_9a_3 - a_{15}a_{12}a_4$, $r_3 := a_7a_2 - a_6a_3$, $r_4 := a_{11}a_7 - a_{21}a_{18}a_9$, $r_5 :=  a_{20}a_9-a_{11}a_6$, $r_6 := a_{21}a_{17}-a_{20}a_{15}$ and $r_7 := a_{17}a_{12} - a_{18}a_{14}$. The equation
\begin{equation}
\label{eqn:syzygyeqn}
a_{21}a_{18}r_1 + a_{20}r_2+a_{11}r_3  - r_4a_2 - r_5a_3 - r_6a_{12}a_4 = a_{21}r_7a_4\in JIJ
\end{equation}
shows how to pass between two presentations of the syzygy 
corresponding to $\eta_{23}\in \Delta_3$:
\[
s_{23} = [a_{21}a_{18}r_1 + a_{20}r_2+a_{11}r_3] = [r_4a_2 + r_5a_3 + r_6a_{12}a_4]\in \text{Tor}^A_3(U_0,U_0)
\] Multiplying \eqref{eqn:syzygyeqn} on the left by $a_{23}$ and expanding provides an equation with signs linking the seven terms of $W$ involving $a_{23}$, but each term appears \emph{twice} with opposite signs.
\end{remark}

\subsection{The main conjecture}
To formulate the cellular resolution conjecture, assume that $A$ is the consistent toric algebra associated to a collection $\mathscr{E}$ on a Gorenstein affine toric variety $X$ of dimension $n$. Let $\widetilde{Q}\subset \RR^d$ denote the covering quiver of the quiver of sections $Q$ of $\mathscr{E}$. Write $f\colon \RR^d\to \RR^n:=M\otimes_\ZZ \RR$ for the orthogonal projection. A priori, the vertices and arrows in the image $f(\widetilde{Q})$ may collide and intersect. Nevertheless, in every case considered in this paper, the toric cell complex $\Delta\subset \mathbb{T}^n$ is constructed from an $M$-periodic quiver in $\RR^n$ whose 0-cells and 1-cells are supported in the image $f(\widetilde{Q})\subset \RR^n$, and we suggest that this can always be done if the global dimension of $A$ is equal to $n$. More precisely, we formulate the following conjecture.

\begin{conjecture}
If the global dimension of a consistent toric algebra $A$ equals the dimension of $X$, then the toric cell complex $\Delta\subset\mathbb{T}^n$ exists and is constructed as above. Moreover, \eqref{eqn:NCBSresolution} is the minimal projective $(A,A)$-bimodule resolution of $A$ as in Theorem~\ref{thm:1.2}.
\end{conjecture}

As Remark~\ref{rem:notcelldivision} shows, consistent toric algebras of global dimension $n$ need not be Calabi--Yau in general. Nevertheless, assuming the conjecture, a sufficient condition for such algebras to be Calabi--Yau can be read off directly from $\Delta$:
%toric cell complex as follows. 

\begin{corollary}
Let $A$ be a consistent toric algebra of global dimension $n$ with toric cell complex $\Delta$. If $\Delta$ satisfies the duality property as stated in Proposition~\ref{prop:dualityMcKay}, then $A$ is Calabi--Yau.
\end{corollary}

The algebras that satisfy the conditions of the corollary, including the example presented in Sections~6.1-6.2,  generalise to arbitrary dimension the three-dimensional algebras constructed from algebraically consistent dimer models.

\def\cprime{$'$}

\end{document}